\documentclass[a4paper, 11pt]{scrreprt}

\addtokomafont{disposition}{\rmfamily}
\usepackage{amssymb}
\usepackage[utf8]{inputenc}
\usepackage{amsmath}
\usepackage{amsthm}
\usepackage{import}
\usepackage{currvita}
\usepackage{tabto}
\usepackage[arrow, matrix, curve]{xy}
\usepackage{cite}
\usepackage{hyperref}

\usepackage{graphicx}
\usepackage{color}

\usepackage{setspace}

\newcommand{\RR}{\mathbb {R}}
\newcommand{\ZZ}{\mathbb {Z}}
\newcommand{\NN}{\mathbb {N}}
\newcommand{\CC}{\mathbb {C}}

\newcommand{\sphere}{\mathbb{S}}

\newcommand{\T}{\mathsf{T}}
\newcommand{\crcl}{\mathsf S^1}
\newcommand{\G}{\mathsf G}

\newcommand{\on}[1]{\operatorname{#1}}
\newcommand{\curv}{\operatorname{curv}}

\newcommand{\Iso}{\operatorname{Iso}}
\newcommand{\Ric}{\operatorname{Ric}}
\newcommand{\Fix}{\operatorname{Fix}}

\renewcommand{\epsilon}{\varepsilon}
\newtheorem{theorem}{Theorem}[chapter]
\newtheorem{proposition}[theorem]{Proposition}
\newtheorem{lemma}[theorem]{Lemma}
\newtheorem{corollary}[theorem]{Corollary}
\newtheorem*{corollary*}{Corollary}
\newtheorem*{theorem*}{Theorem}
\theoremstyle{definition}
\newtheorem{definition}[theorem]{Definition}
\theoremstyle{remark}
\newtheorem*{remark}{Remark}
\newtheorem*{texttheorem*}{Theorem}
\begin{document}
\pagenumbering{roman}
\addtokomafont{disposition}{\rmfamily}
\title{$\crcl$-actions on $4$-manifolds and fixed point homogeneous manifolds of nonnegative curvature}

\author{Wolfgang Lorenz Spindeler}
\date{2014}
\pagenumbering{Alph}
\maketitle
\pagenumbering{arabic}
\newpage

\pagenumbering{roman}
\begin{abstract}
\noindent This is a slightly altered version of the authors thesis \cite{spindeler14}. Changes are mostly due to typographical errors, different layout and updated references. Mathematical changes are scarce and noted when they occur.
\\ \\
\begin{center}\textbf{Abstract}\end{center}
Let $M$ be a closed, nonnegatively curved and simlpy connected Riemannian $4$-manifold equipped with an isometric action of the circle $\crcl$ with only isolated fixed points. The first main theorem of this thesis shows that there exists a sequence of smooth, positively curved metrics on the quotient space $M/\crcl$ that converges in Gromov-Hausdorff topology to the singular quotient metric of $M/\crcl$, and the same holds for the two fold branched cover of $M/\crcl$ along a singular closed curve. This result leads to a solely geometric proof (not appealing to the Poincaré conjecture) of the equivariant classification theorem for nonnegatively curved, simply connected $4$-manifolds with circle symmetry obtained by Grove and Wilking in \textit{``A knot characterization and 1-connected nonnegatively curved $4$-manifolds with circle symmetry''}, Geometry \& Topology 18 (2014) 3091–3110.

The second main theorem deals with fixed point homogeneous manifolds. These are by definition Riemannian manifolds admitting an effective, isometric action by a Lie group $\G$ with nonempty fixed point set such that a fixed point component has codimension $1$ in the orbit space $M/\G$. It is shown that a closed, nonnegatively curved, fixed point homogeneous $\G$-manifold admits a double disk bundle decomposition over a fixed point component of maximal dimension and another smooth $\G$-invariant submanifold of $M$. As an application of this result it is shown that a closed, simply connected, nonnegatively curved torus manifold is rationally elliptic.
\end{abstract}

\tableofcontents


\addchap{Introduction}
A classical subject in Riemannian geometry is the study of complete Riemannian manifolds of positive or nonnegative (sectional) curvature, in particular their topology. In dimensions $\leq 3$, complete, nonnegatively and positively curved manifolds are classified. In the compact case a classification in dimension $2$ follows from the Gau\ss-Bonnet theorem and in dimension $3$ from the work of Hamilton \cite{hamilton82}, \cite{hamilton86}. In dimensions $\geq 4$ a classification of compact, nonnegatively or positively curved manifolds is still open. For example, a famous unsolved conjecture by Hopf states that the product of two spheres $\sphere^2 \times \sphere^2$ does not admit a metric of positive curvature (whereas the product of the standard metrics has nonnegative curvature). Motivated by this conjecture it was proven by Hsiang and Kleiner in \cite{hsiang-kleiner89}  that a compact, simply connected, positively curved $4$-manifold that additionally admits a nontrivial, isometric action of the circle $\crcl$ is homeomorphic to $\sphere^4$ or the complex projective space $\CC P^2$.

There are two questions that arise naturally. First, one may ask whether there is an analogous result in nonnegative curvature and secondly, whether one can classify such manifolds up to (equivariant) diffeomorphism.

Addressing the first question, it was shown independently by Kleiner \cite{kleiner90} and Searle and Yang \cite{searle-yang94} that a compact, simply connected, nonnegatively curved $4$-manifold admitting a nontrivial, isometric $\crcl$-action is homeomorphic to $\sphere^4,\ \CC P^2,\ \sphere^2 \times \sphere^2$ or $\CC P^2 \# \pm \CC P^2$.

The second question was answered eventually by combined work of Grove and Wilking \cite{grove-wilking13} and Galaz-Garcia and Kerin \cite{galaz-garcia-kerin14}. In particular in \cite{grove-wilking13} it was shown that a compact, simply connected, nonnegatively curved $4$-manifold $M$ admitting an isometric $\crcl$-action is diffeomorphic to $\sphere^4,\ \CC P^2,\ \sphere^2 \times \sphere^2$ or $\CC P^2 \# \pm \CC P^2$ and the action extends to a smooth (effective) action of the $2$-torus $\mathsf T^2$. All such actions have been classified and they all admit invariant metrics of nonnegative curvature (compare the disucssion in \cite{grove-wilking13}). Therefore, it follows an equivariant classification of closed, nonnegatively curved $4$-manifolds admitting an effective, isometric $\crcl$-action. 

Important ingredients of the proof of this theorem are the study of the geometry of the orbit space $M/\crcl$ (as it was for the theorems quoted before) and the solution to the Poincar\'e conjecture \cite{perelman02}, \cite{perelman03} and its equivariant version \cite{dinkelbach-leeb09}. The latter is needed to deduce that $M/\crcl$ is homeomorphic to $\sphere^3$, in the case that $\crcl$ has only isolated fixed points, and to obtain a crucial characterization of knots in $M/\crcl$ (for the details see \cite{grove-wilking13}). It is an interesting question whether there is a geometric proof which does not make use of the Poincar\'e conjecture. The first main theorem of this thesis gives an affirmative answer (in the case $M$ has Euler characteristic $4$, see also section 4 of \cite{grove-wilking13}).
\newpage
\begin{theorem}\label{resolution thm_intro}
  Let $\mathsf S^1$ act isometrically and with only isolated fixed points on a simply connected, compact, nonnegatively curved $4$-manifold $M$ with quotient space $(M/\crcl,d)$. Then the following is true:
\begin{itemize}
 \item[(a)] There exists a sequence of smooth positively curved Riemannian metrics $(g_n)_{n \in \NN}$ on $M/\crcl$ such that  $(M/\crcl, g_n)$ has Gromov-Hausdorff limit $(M/\crcl,d)$.
 \item[(b)]  Assume additionally that the nonregular part of $M/\crcl$ contains a closed curve $c$. Then there exists a sequence of smooth, positively curved orbifold metrics $(h_n)_{n \in \NN}$ on $M/\crcl$ such that $(M/\crcl,h_n)$ has Gromov-Hausdorff limit $(M/\crcl,d)$ and the only singularities of $h_n$ are given by $\mathbb Z_2$-singularities along $c$.
\end{itemize}
\end{theorem}
This result is proven in chapter $2$ after some preliminary discussions in chapter $1$. The proof is elementary, but rather technical, and based on a careful deformation of the geometry of the space $M/\crcl$. For an overview of the arguments we refer to the introduction of chapter $2$.

From the classification of positively curved $3$-manifolds by Hamilton \cite{hamilton82} it follows (ignorant of the Poincar\'e conjecture):
\begin{corollary*}
 With the conditions as in the above theorem $M/\crcl$ is homeomorphic to $\sphere^3$.
\end{corollary*}
The same conclusion also holds for the two fold branched cover over a singular closed curve and it follows that such a curve is unknotted (see section \ref{final proof}). It is thus possible to replace the references to the Poincar\'e conjecture in \cite{grove-wilking13} by referring to Theorem \ref{resolution thm_intro} and Hamilton's classification of positively curved $3$-manifolds instead.

One might further hope that this result helps to find resolutions of other quotient spaces, which in the optimal case leads to new examples of positively curved manifolds. However, we do not follow this idea here any further, but we note that it was shown in \cite{dyatlov} that such resolutions exist for a particular class of $\crcl$- or $\mathsf S^3$-actions on compact, positively curved manifolds.

After this work was done we learned from Petrunin that in \cite{petrunin14} a similar technique was used to smooth $3$-dimensional polyhedral spaces.
\\ \\
The second main result of this thesis deals with so called \textit{fixed point homogeneous} ma\-ni\-folds. As it is illustrated by the theorems quoted before, adding assumptions on the degree of symmetry of a, say, nonnegatively or positively curved manifold may lead to strong conclusions about the topology or even diffeomorphism type of the manifold. Following this approach, many results for nonnegatively and positively curved manifolds that admit a large amount of symmetry have been obtained (for an overview see \cite{wilking07}).

An isometric action of a compact Lie group $\G$ on $M$ is called \textit{fixed point homogeneous} if it has nonempty fixed point set $\operatorname{Fix}(M)$ and
$$\dim M/\G = \dim \operatorname{Fix}(M) + 1.$$
It is clear that $\dim M/\G \geq \dim \operatorname{Fix}(M) + 1$ for any nontrivial $\G$-action on $M$. Hence, regarding the above approach, fixed point homogeneous actions can be thought of as the largest possible actions given their fixed point components.

In \cite{grove-searle97} Grove and Searle gave a diffeomorphism classification of compact, positively curved, fixed point homogeneous manifolds (in the simply connected case the list consists of the compact rank one symmetric spaces).

Many of the techniques used in their arguments are applicable to the case of nonnegative curvature and led to classification results in dimensions $\leq 5$, see \cite{galaz-garcia12}, \cite{galaz-garcia-spindeler12}. However, note that a classification of nonnegatively curved, fixed point homogeneous manifolds in arbitrary dimensions is equivalent to a classification of nonnegatively curved manifolds (given a nonnegatively curved manifold $N$, the product manifold $N \times \sphere^2$ admits a fixed point homogeneous $\crcl$-action that leaves the product metric invariant). To obtain the mentioned classification results, a decomposition of the manifold as a double disk bundle was instrumental. The second main result of this thesis is that such a decomposition exists for nonnegatively curved, fixed point homogeneous manifolds in arbitrary dimensions.
\begin{theorem}\label{double disk bundle}
 Let $M$ be a closed, fixed point homogeneous $\G$-manifold of nonnegative curvature and $F$ be a fixed point component of maximal dimension. Then there exists a smooth, $\G$-invariant submanifold $N$ (without boundary) of $M$ such that $M$ is equivariantly diffeomorphic to the normal disk bundles $D(F)$ and $D(N)$ of a $F$ and $N$, respectively, glued together along their boundaries;
 $$M \cong D(F) \cup_\partial D(N).$$
 Further $\G$ acts freely on $M \setminus (F \cup N)$.
\end{theorem}
In the case of positive curvature $N$ is a point and the theorem was proven in \cite{grove-searle97}. The case of nonnegative curvature is considerably more difficult. For an overview of the arguments see the introduction to chapter $3$, where the theorem is proven. As an application we show the following corollary:
\begin{corollary*}
 Let $M$ be a closed, nonnegatively curved, simply connected torus manifold. Then $M$ is rationally elliptic.
\end{corollary*}
$M$ is called a torus manifold if it has dimension $2n$, is orientable, and admits an effective action by the $n$-dimensional torus $\mathsf T^n$ with nonempty fixed point set. A stronger result than this corollary is obtained by Wiemeler in \cite{wiemeler14}, where he gives a classification of nonnegatively curved torus manifolds up to equivariant diffeomorphism (note that he uses also our results to obtain this classification).
\\ \\
Finally, we note that the chapters $2$ and $3$, dealing with the proof of theorem \ref{resolution thm_intro} and \ref{double disk bundle} respectively, can be read independently from each other.
\subsection*{Acknowledgements}
It is a great pleasure to thank my advisor Burkhard Wilking for his support and many helpful discussions. I thank  Fernando Galaz-Garcia for pointing out the reference \cite{dyatlov} to me. I thank Michael Wiemeler for noticing an error in the original version of of this work \cite{spindeler14}.
\chapter{Background from Riemannian geometry, group actions and Alexandrov spaces.}\label{bckgrnd}
\pagenumbering{arabic}
We assume that the reader is familiar with the basic concepts and results of Riemannian geometry.  Also, we assume a background on the theory of smooth transformation groups at about the level of the slice theorem, see for example theorem 5.4 in \cite{bredon72}. 

This, together with the material discussed in this chapter, is the necessary background for chapter \ref{chapter resolving}, dealing with the proof of theorem \ref{resolution thm_intro}. It is suggested to the reader of chapter \ref{chapter resolving} that (the details of) the results presented here are skipped in a first reading, and checked when they are referred to. 

Chapter \ref{fph} can be read independently of chapter \ref{chapter resolving}. The reader only interested in the results on nonnegatively curved, fixed point homogeneous manifolds may skip chapters \ref{bckgrnd} and \ref{chapter resolving} entirely, after possibly checking our notations presented in the first paragraphs of sections \ref{group actions} and \ref{orbs}. But note that in chapter $\ref{fph}$ we additionally assume a more detailed background on Alexandrov spaces with curvature bounded below than it is presented in this section (a nice discussion can be found in \cite{burago-burago-ivanov01}).

\section{Riemannian geometry}\label{Riemannian geometry}
\subsection{Basic definitions, formulas and convex functions}\label{Dfacf}
Since some conventions, for example on the sign of the curvature tensor, differ in the literature we start recalling some basic definitions.
Let $(M,g)$ be a Riemannian $n$-manifold with Levi-Civita connection $\nabla$. The curvature tensor is defined as
\begin{align*}
 R_{XY}Z = \nabla_X\nabla_YZ - \nabla_Y\nabla_XZ - \nabla_{\left[X,Y\right]}Z\nonumber
\end{align*}
for smooth vector fields $X,Y$ and $Z : M \to TM$. An inner product on $\Lambda^2T_pM$, the space of  bivectors, is defined via
$$\langle v_1 \wedge v_2, w_1 \wedge w_2\rangle = \det (g(v_i,w_j)_{i,j \in \{1,2\}})$$
and extending bilinearly.
The curvature operator $\mathcal R_p : \Lambda^2T_pM \to \Lambda^2T_pM$ is  the linear operator satisfying 
\begin{align*}
 \langle \mathcal R_p(b_i \wedge b_j),b_k \wedge b_l \rangle = g(R_{b_ib_j}b_l,b_k)
\end{align*}
for an orthonormal basis $\{b_1, \dots, b_n\}$ of $T_pM$ and for all $1 \leq i < j \leq n$. 

In chapter $2$ we do several calculations on surfaces $(M^2,g)$. On such a surface we regard the sectional curvature as a smooth function $\sec : M^2 \to \RR$. The following formula is then frequently used:

\begin{lemma}\label{curvature}
 Let $(r,\varphi)$ be coordinates on $U \subset \RR^2$ and $f : U \to \RR$ be a positive smooth function. Consider the Riemannian metric
 $$g = dr^2 + f^2d\varphi^2$$
 on $U$. Then the sectional curvature of $g$ at $(r,\varphi)$ is given by
 $$\sec(r,\varphi) = -\frac{\partial^2_{r}f(r,\varphi)}{f(r,\varphi)}.$$
\end{lemma}
\begin{proof}
The coordinate vector fields of the coordinates $(r,\varphi)$ are denoted by $\partial_r$ and $\partial_\varphi$. First observe that for fixed $\varphi$ the curve $r \mapsto (r,\varphi)$ is a geodesic, so $\nabla_{\partial_r}\partial_r = 0$ (one may also directly calculate this from the Christoffel symbols). It follows that 
 $$g(\nabla_{\partial_r}\partial_{\varphi},\partial_r) = \partial_rg(\partial_\varphi,\partial_r) - g(\partial_\varphi,\nabla_{\partial_r}\partial_r) = 0.$$
 Further
 $$g(\nabla_{\partial_r}\partial_{\varphi},\partial_\varphi) = \frac 1 2 \partial_r g(\partial_\varphi,\partial_\varphi) = f\partial_rf.$$
Since $\partial_r$ and $f^{-1}\partial_\varphi$ are orthonormal, it follows 
$$\nabla_{\partial_r}\partial_\varphi = f^{-1}\partial_rf\partial_\varphi.$$
Consequently, since $\left[\partial_r,\partial_{\varphi}\right] = 0$,
\begin{align*}
\sec(r,\varphi) &= f^{-2}(r,\varphi)g(R_{\partial_\varphi \partial_r}\partial_r,\partial_\varphi)(r,\varphi)\\
&= -f^{-2}(r,\varphi)g(\nabla_{\partial_r} \nabla_{\partial_r} \partial_\varphi,\partial_\varphi)(r,\varphi)\\
&= -f^{-2}(r,\varphi)(\partial_rg(\nabla_{\partial_r} \partial_\varphi,\partial_\varphi) - g(\nabla_{\partial_r} \partial_\varphi,\nabla_{\partial_r} \partial_\varphi))(r,\varphi)\\
&= -f^{-2}(r,\varphi)(\partial_r(f\partial_rf) - f^{-2}(\partial_rf)^2g(\partial_\varphi,\partial_\varphi))(r,\varphi)\\
&= -f^{-2}(r,\varphi)(f\partial^2_{r}f + (\partial_rf)^2 - (\partial_rf)^2)(r,\varphi)\\
&= -\partial^2_{r}f(r,\varphi)/f(r,\varphi). \qedhere
\end{align*}
\end{proof}


For a twice differentiable function $f : M \to \RR$ with gradient $\nabla f$ and $v,w \in T_pM$ recall that the Hessian of $f$ is defined via
\begin{align}
 \nabla^2f(v,w) = g(v,\nabla_w\nabla f).\nonumber
\end{align}

A function $f : M \to \RR$ is \textit{convex} (\textit{concave}) if $f \circ \gamma$ is convex (concave) in the usual sense for every geodesic $\gamma$. If $f$ is twice differentiable, and $\gamma$ is a geodesic, it follows that 
$$\nabla^2f(\dot \gamma(t),\dot \gamma(t)) = (f \circ \gamma)''(t).$$
Thus $f$ is convex if and only if $\nabla^2 f \geq 0$. A smooth function $f : M \to \RR$ is \textit{strictly convex} (\textit{strictly concave}) if $\nabla^2 f > 0$ ($\nabla^2f  < 0$). A generalization of this definition to nonsmooth functions was given in \cite{greene-wu76}.
\begin{definition}
 A function $f : M \to \RR$ on a Riemannian manifold is \textit{strictly convex} if for every $p \in M$, for every strictly convex, smooth function $\varphi$ defined in a neighborhood of $p$ there exists an $\epsilon > 0$ such that $f - \epsilon \varphi$ is convex in a neighborhood of $p$. A function $f : M \to \RR$ is \textit{strictly concave} if $-f$ is strictly convex.
\end{definition}
The proof of the following lemma is straight forward and we omit it here.
\begin{lemma}
 Given strictly convex (strictly concave) functions $f$ and $g$, their maximum $\max (f,g)$ (minimum $\min(f,g)$) is strictly convex (strictly concave).
\end{lemma}
\subsection{The gluing lemma}\label{section gluing}
In this section we derive a simple generalization of a result from Dyatlov \cite{dyatlov}, which is the basis for the arguments in section \ref{section symmetry}.

Let $N \subset M$ be a closed submanifold. We address the question, whether it is possible to glue two metrics defined in a neighborhood of $N$ without changing a common lower curvature bound too much. In case of positively curved metrics this question is answered by the following lemma, see \cite[Lemma 4.3]{dyatlov}.

\begin{lemma}\label{positive gluing}
Let $(M,g)$ be a positively curved manifold and $N \subset M$ be a closed submanifold. Let $\tilde g$ be a positively curved metric defined in an open neighborhood $U$ of $N$. Assume that $\tilde g$ coincides with $g$ up to first order at all points of $N$. Then there exists a smooth, positively curved metric $h$ on $M$ such that $h$ coincides with $g$ on $M \setminus U$ and $h$ coincides with $\tilde g$ on an open neighborhood of $N$.
\end{lemma}

We give an overview of the proof and show that one can in fact deduce a more general statement from it, for the details compare \cite{dyatlov}.

First, it is shown that for all $\epsilon > 0$ there exists $0 < \delta < \epsilon$ and a smooth cut off function $\varphi_\epsilon : [0,\infty[ \to \RR$, supported in $[0,\epsilon]$, such that $\varphi_\epsilon = 1$ on $[0,\delta]$, and for all $x \in [0,\infty[$ one has
\begin{align}\label{cutoff}
 \vert x \varphi_\epsilon'(x)\vert \leq \epsilon \text{ and } 
 \vert x^2 \varphi_\epsilon''(x)\vert \leq \epsilon.
\end{align}
Now let $N \subset M$ be a closed submanifold and $g, \tilde g$ be two metrics defined on a neighborhood of $N$ such that $g$ and $\tilde g$, together with its first order derivatives, coincide at all points of $N$. For $q \in M$ let $\psi_\epsilon(q) = \varphi_\epsilon(d^g(N,q))$ and set 
\begin{align*}
h_\epsilon(q) = \psi_\epsilon(q)\tilde g(q) + (1 - \psi_\epsilon(q))g(q).
\end{align*}
Fix $p \in N$. A direct calculation in a coordinate system at $p$, using \eqref{cutoff}, shows that there exist a neighborhood $U$ of $p$ and a constant $m$ (where $U$ and $m$ are independent of $\epsilon$) such that
\begin{align}\label{banane}
R^{h_\epsilon}_{ijkl}(q) = \psi_\epsilon(q) R^{\tilde g}_{ijkl}(q) + (1 - \psi_\epsilon(q))R^g_{ijkl}(q) + r^\epsilon_{ijkl} (q)
\end{align}
for all $q \in U$ and $\vert r^\epsilon_{ijkl}(q) \vert < m\epsilon$ for all $\epsilon$ sufficiently small. For a given common lower curvature bound $c$ of $g$ and $\tilde g$ it follows that for all sufficiently small $\epsilon > 0$ the lower curvature bound of $h_\epsilon$ is arbitrary close to $c$ on some open neighborhood of $p$. Since $N$ is compact, it follows the first part of the following lemma.

\begin{lemma}[Gluing lemma]\label{gluing}
 Let $N \subset (M,g)$ be a closed submanifold and $\tilde g$ be a smooth metric defined on an open neighborhood $U$ of $N$ which coincides with $g$ up to first order at all points of $N$. Assume that $g$ and $\tilde g$ satisfy a common lower curvature bound $c$. Then for all $\epsilon > 0$ there exists a smooth metric $h = h(\epsilon)$ on $M$ such that $h = g$ on $M\setminus U$, $h = \tilde g$ on some open neighborhood of $N$, and $h_\epsilon$ has lower curvature bound $c - \epsilon$. Further, $h$ may be chosen in a way that $\operatorname{Iso}_N(g_{|U}) \cap \operatorname{Iso}_N(\tilde g) \subset \operatorname{Iso}(h_{|U})$. 
\end{lemma}
Here $\operatorname{Iso}_N(g_{|U}) \subseteq \operatorname{Iso}(g_{|U})$ and $\operatorname{Iso}_N(\tilde g) \subseteq \operatorname{Iso}(\tilde g)$ denote the corresponding subgroups of the isometry groups $\operatorname{Iso}(g_{|U})$ and $\operatorname{Iso}(\tilde g)$ that leave $N$ invariant. Since also $\operatorname{Iso}(g_{|U}) \subset \operatorname{Diffeo}(U)$ and $\operatorname{Iso}(\tilde g) \subset \operatorname{Diffeo}(U)$, their intersection is well defined. 

It remains to prove that $\operatorname{Iso}_N(g_{|U}) \cap \operatorname{Iso}_N(\tilde g) \subset \operatorname{Iso}(h_{|U})$: Let $f \in \Iso_N(g_{|U}) \cap \Iso_N(\tilde g)$. Since $\psi$ is radial with respect to $g$ and $N$ and further $f$ is an isometry leaving $N$ invariant, we conclude that $\psi \circ f = \psi$. Consequently, since also $f \in \Iso_N(\tilde g)$,
\begin{align}
 h_{f(p)}(dfX,dfY) &= \psi(f(p)) \tilde g_{f(p)}(dfX,dfY) + (1 - \psi(f(p)))g_{f(p)}(dfX,dfY)\nonumber\\
 &= \psi(p)\tilde g_p(X,Y) + (1 - \psi(p))g_p(X,Y)\nonumber\\
 &= h_p(X,Y),\nonumber
\end{align}
so $f \in \operatorname{Iso}(h_{|U})$.

\begin{corollary}\label{constant curvature at fixed point}
Let $(M,g)$ be a Riemannian manifold with $\sec \geq c$ and $p \in M$. Then for a given open neighborhood $U$ of $p$ and all $\epsilon > 0$ and $\kappa \geq c$ there exists a Riemannian metric $h = h(U, \epsilon, \kappa)$ on $M$ such that $h = g$ on $M \setminus U$, $\sec_h \geq c - \epsilon$ and $h$ has constant curvature $\kappa$ on an open neighborhood $V$ of $p$. Further, $\Iso_{p}g \subseteq \Iso_ph$.
\end{corollary}
As above, $\Iso_{p}g$ denotes the group of isometries of $g$ that fix $p$, and analogously for $h$.
\begin{proof}
Since $\exp_p$ is a local diffeomorphism, we may assume that $M = \RR^n$, $p = 0$, $g_{ij}(0) = \delta_{ij}$ and $\partial_kg_{ij}(0) = 0$ with respect to the standard coordinate system of $\RR^n$. Since for every $f \in \Iso_p(M)$ we have $f \circ \exp_p = \exp_p \circ df_p$ by equivariance of $\exp_p$, we may further assume that $\Iso_{0}(g) \subset \mathsf O(n)$.
Now let $\tilde h_{ij}$ be the standard metric of constant curvature $\kappa$ defined on some open neighborhood $\tilde V$ of $0$. Then $\tilde h_{ij}(0) = \delta_{ij}$ and $\partial_i \tilde h_{jk}(0) = 0$. So $g$ and $\tilde h$ coincide up to first order at $0$ and we can apply the gluing lemma to obtain the metric $h$. Since $\Iso_p(\tilde h) = \mathsf O(n)$, the claim follows.
\end{proof}
\section{Isometric group actions}\label{group actions}
Let $\mathsf G$ be a compact Lie group acting smoothly and from the left on a compact manifold $M$. For $g \in \G$ we often identify $g$ with the induced diffeomorphism of $M$ sending $p$ to $g.p$. We denote the orbit of a point $p \in M$ by $\G(p)$ and the isotropy group by $\G_p$. Assuming that $\G$ acts isometrically with respect to some Riemannian metric $g$, the tangent space at $p$ decomposes orthogonally as $T_pM = T_p\G(p) \oplus N_p\G(p)$, and this decomposition is invariant under the action of $\G_p$ on $T_pM$ via $g.v = dg_pv$ for $v \in T_pM$. Hence $\G_p$ also acts on $N_p\G(p)$. We refer to these actions on $T_pM$ and $N_p\mathsf G(p)$ as the \textit{isotropy representation} and \textit{slice representation}, respectively. By the slice theorem a tubular neighborhood of $\G(p)$ is equivariantly diffeomorphic to $\G \times_{\mathsf G_p} N_p \G(p)$. Two orbits are said to be of the \textit{same type} if their isotropy groups are conjugate, and $\G(q)$ has \textit{bigger type} than $\G(p)$ if a conjugate of $\G_q$ is properly contained in $\G_p$. Since $M$ is compact, it follows from the slice theorem that there exist only finitely many orbit types and a unique maximal type (meaning that it is bigger than any other type). An orbit of maximal type is called \textit{principal}. \textit{Exceptional} orbits are orbits that are not principal but have the same dimension as a principal orbit. A \textit{fixed point} $p$ of the action satisfies $\G_p = \G$.
\subsection{\texorpdfstring{Polar $\crcl$-actions}{Polar circle-actions}}
\begin{definition}
An isometric action by a compact Lie group $\mathsf G$ on a Riemannian manifold $M$ is called \textit{polar} if there exists a connected, immersed submanifold $\Sigma$ which intersects each orbit of the action and does so orthogonally. $\Sigma$ is called a \textit{section} of the action.
\end{definition}
\begin{remark}
 Note that our definition of a polar action differs from the common definition. Usually $M$ as well as $\Sigma$ are assumed to be complete (see for example \cite{grove-ziller12} for an introduction to polar actions). We will also need to deal with noncomplete manifolds, but we stick with the common terminology for convenience.
\end{remark}
\begin{lemma}
A section of a polar action is totally geodesic.
\end{lemma}
\begin{proof}
A proof can be found in \cite{grove-ziller12} (the proof applies without assuming completeness of $M$ or $\Sigma$).
\end{proof}
\begin{lemma}\label{idontknow}
Let $M$ be a polar $\crcl$-manifold, $p \in M$ and $\Sigma$ a section with $p \in \Sigma$. Then for all $x,y,z \in T_p\Sigma$ and $u,v,w \in N_p\Sigma$ 
$$\langle R_{xy}z, v \rangle = \langle R_{xu}v, w \rangle = 0.$$
\end{lemma}
\begin{proof} 
Since $\Sigma$ is totally geodesic, it follows that $R_{xy}z$ is tangent to $\Sigma$. Hence $\langle R_{xy}z, v \rangle = 0.$ Since $N_p\Sigma$ is one dimensional, it follows that $R_{vw}u = 0$, so $\langle R_{xu}v, w \rangle = 0$ by the symmetries of $R$.
\end{proof}
Given an isometric action of $\crcl = \{e^{i\theta} \vert \theta \in \RR\}$ on $M$ the \textit{associated Killing vector field of the action} is the vector field $X(p) = \frac{d}{d\theta}_{|\theta = 0} e^{i\theta}p$.
\begin{lemma}\label{Hessian}
Let $(M,g)$ be a Riemannian manifold admitting a polar $\crcl$-action with associated Killing vector field $X$. Let $\Sigma$ be a section and $\varphi(p) := ||X(p)||$. Then for all $p \in \Sigma$ with $X(p) \neq 0$ and $v,w \in T_p\Sigma$ the following identity holds:
\begin{align}\label{hes = R}
-\varphi(p)\nabla^2\varphi(v,w) = \langle R_{X(p)v}w,X(p)\rangle.
\end{align}
Hence, for $||v|| = 1$, the sectional curvature of the plane spanned by $X(p)$ and $v$ is given by 
$$-\varphi(p)^{-1}\nabla^2\varphi(v,v).$$
In particular $\varphi : (U,g) \to \RR$ defines a concave  (strictly concave) function if $(M,g)$ is nonnegatively (positively) curved, where $U$ denotes the set of points $p \in \Sigma$ with $X(p) \neq 0$.
\end{lemma}
\begin{proof}
By linearity it is enough to prove
\begin{align*}
\nabla^2\varphi(v,v) = -\varphi(p)^{-1} \langle R_{X(p)v}v,X(p)\rangle
\end{align*}
for all $p \in \Sigma$ with $\varphi(p) > 0$ and $v \in T_p\Sigma$. Let such $p$ and $v$ be given and $\gamma_v$ be a geodesic with $\gamma(0) = p$ and $\dot \gamma (0) = v$. First note that 
\begin{align}\label{prrl}
\nabla_t(||X||^{-1}X) = 0,
\end{align}
i.e. $X/||X||$ is parallel along $\gamma$: In fact, $\Sigma$ is invariant under parallel translation since it is totally geodesic and consequently so is $X / ||X||$ because of $\dim M - \dim \Sigma = 1$. 

Now we calculate the Hessian of $\varphi$ as follows:
\begin{align*}
\nabla^2\varphi(v,v) &= \partial^2_{t}(\varphi \circ \gamma)\\
 &= \partial_t(||X||^{-1}\langle \nabla_tX,X\rangle)\\
 &= ||X||^{-1}\langle \nabla_t\nabla_t X,X\rangle\\
 &= -\varphi^{-1}\langle R_{X v}v,X\rangle.
\end{align*}
Here we used that $\gamma$ is a geodesic, \eqref{prrl} and that $X$ is a Jacobi field along $\gamma$.
\end{proof}
\begin{corollary}\label{curvature and Hessian}
 With the conditions as in Lemma \ref{Hessian} let $\{b_1 \dots b_n\}$ be an orthonormal basis of $T_pM$ such that $\{b_2, \dots b_n\}$ defines an orthonormal basis of $T_p\Sigma$. Then the curvature operator of $M$ at $p$ is given by
 \begin{align} 
&(<\mathcal R(b_i\wedge b_j),b_k\wedge b_l>)_{(1 \leq i<j \leq n,\ 1 \leq k<l \leq n)}\nonumber \\
= &\begin{pmatrix}
   (-\varphi^{-1}(p)\nabla^2\varphi(b_j,b_l))_{(j,l \in \{2, \dots, n\})} & 0 \\
   0  & (<\mathcal R^{\Sigma}(b_i\wedge b_j),b_k\wedge b_l>)_{(2 \leq i<j \leq n,\ 2 \leq k<l \leq n)}\nonumber
   \end{pmatrix},
\end{align}
where $\mathcal R^\Sigma$ denotes the curvature operator of $\Sigma$.
\end{corollary}
\begin{proof}
 The upper left block is a consequence of  lemma \ref{Hessian}, since up to sign we have $b_1 = ||X(p)||^{-1}X(p)$. The lower right block follows from $\mathcal R_{|\Sigma} = \mathcal R^\Sigma$, since the section is totally geodesic. From lemma \ref{idontknow} we see that $\langle \mathcal R(b_i\wedge b_j),b_k\wedge b_l\rangle = g(R_{b_ib_j}b_l,b_k) = 0$ if precisely one of the vectors $b_i, b_j, b_k$ and $b_l$ equals $\pm||X||^{-1}X$.
\end{proof}
\subsection{Quotient spaces and orbifolds}\label{orbs}
Let $M$ be a Riemannian manifold equipped with an isometric action by a compact Lie Group $\G$. The quotient space $M/\G$ is denoted by $M^*$ with projection map $\pi : M \to M^*$. Throughout our analysis of group actions we will mostly focus on the geometry of $M^*$ rather than investigating the action of $\G$ on $M$ itself. Therefore we introduce the following terminology: For a point $p \in M^*$ by its \textit{isotropy group} we mean (the conjugacy class of) $G_{\hat p}$ for any point $\hat p \in M$ with $\pi(\hat p) = p$. A \textit{fixed point} $p \in M^*$ is a point with $\G_{\hat p} = \G$. A \textit{regular point} of $M^*$ is a point with principal isotropy group, and a \textit{nonregular point} is one that is not regular. We also use the word \textit{singular} synonymously for nonregular.

The space $M^*$ carries an intrinsic metric $d$, which is defined via the distance of orbits. By the slice theorem, the set $M^*_{reg}$ of regular points is convex in $M^*$. There exists a unique smooth Riemannian metric $g$ on $M_{reg}^*$ which induces $d$ and with respect to this metric, $\pi : \pi^{-1}(M^*_{reg}) \to M_{reg}^*$ is a Riemannian submersion.
\begin{definition}
 A \textit{smooth Riemannian orbifold}, or just \textit{orbifold}, is a metric space $X$ such that for all points $p \in X$ there exists an open neighborhood $U$ of $p$ that is isometric to $\hat U/\Gamma$, where $\hat U$ is a smooth Riemannian manifold and $\Gamma$ is a finite group acting isometrically on $\hat U$. An \textit{orbifold point} of a metric space is a point that has an open neighborhood which is an orbifold. A \textit{good orbifold} is a metric space $X$ which is isometric to $M/\Gamma$ where $M$ is a Riemannian manifold and $\Gamma$ is a finite group of isometries.
\end{definition}
Obvious examples of (good) orbifolds are quotient spaces of Riemannian manifolds by finite groups of isometries. In general quotient spaces of Riemannian manifolds by isometric group actions may fail to be orbifolds. A characterization is obtained via the following result from Lytchak-Thorbergsson \cite{lytchak-thorbergson}:
\begin{theorem}\label{lytchak-thorbergson}
 Let $\G$ be a compact Lie group acting isometrically on a Riemannian manifold $M$. Then $p \in M/\G$ is an orbifold point if and only if the isotropy representation of $\G_p$ is polar.
\end{theorem}
\begin{remark}
 It is not hard to see that for an orthogonal action on an Euclidean vector space our notion of a polar action is equivalent to the complete version. So there is no ambiguity with this statement.
\end{remark}
In particular all points $p \in M^*$ with finite isotropy group are orbifold points.
\section{Alexandrov spaces}\label{Alexandrov}
We just saw that quotient spaces of particular nice actions are orbifolds. In general quotient spaces of isometric actions by compact Lie groups on compact Riemannian manifolds are Alexandrov spaces with curvature bounded below. Recall that a \textit{length space} is a metric space $(X,d)$ such that the distance of any two points of it is realized as the infimum of the lengths of all continuous curves connecting them. A \textit{geodesic} in a length space is a continuous curve that is locally a shortest curve between points on it. A triangle $\Delta(p,q,r)$ in a length space $X$ is a collection of three points $p,q$ and $r \in X$, together with a minimal (meaning of minimal length) geodesic for each pair of points connecting them. A comparison triangle $\tilde \Delta_k(p,q,r)$ is a triangle in the $2$-dimensional, simply connected Riemannian manifold of constant curvature $k$, with side lengths equal to the corresponding sides of $\Delta(p,q,r)$. For fixed $k$ such a comparison triangle exists for all $\Delta(p,q,r)$ of sufficiently small perimeter. Denote by $\tilde \measuredangle_k (p,q,r)$ the angle of $\tilde \Delta_k(p,q,r)$ opposite to a side of length $d(p,r)$.
\begin{definition}
Let $k \in \RR$. A complete length space $(A,d)$ is called an \textit{Alexandrov space with curvature bounded below  by $k$}, denoted $\on{curv}(A) \geq k$, if for all points in $A$ there exists an open neighborhood $U$ such that for all $p,q,r,s \in U$ we have
$$\tilde \measuredangle_k (q,p,r) + \tilde \measuredangle_k (s,p,r) + \tilde \measuredangle_k (q,p,s) \leq 2\pi.$$
\end{definition}
There are various possible reformulations of this definition, compare \cite{burago-gromov-perelman92}. Most of the time we call a space as in this definition simply an Alexandrov space. Given a set $X$, a metric $d$ on $X$ is an \textit{Alexandrov metric}, if the metric space $(X,d)$ is an Alexandrov space.

The properties of the cone and spherical suspension over a metric space will be of particular importance, so we recall the definition here together with a basic result. For a metric space $X$ we denote its diameter by $\operatorname{diam}(X)$.

\begin{definition}
 Let $(X,d_X)$ be a metric space with $\operatorname{diam} (X) \leq \pi$. The \textit{cone over $X$} is the metric space $C(X) = (X \times [0,\infty[)/\sim$, where $(x,t) \sim (y,s)$ if $s = t = 0$, equipped with the metric
$$d((x,t),(y,s)) = \sqrt{t^2 + s^2 - 2ts\cos(d_X(x,y))}.$$ 
The \textit{spherical suspension of $X$} is the metric space $S(X) = (X \times [0,\pi])/\sim$, with $(x,t) \sim (y,s)$ if $s = t = 0$ or $s = t = \pi$, equipped with the metric defined via
$$\cos(d((x,t),(y,s))) = \cos(t)\cos(s) + \sin(t)\sin(s)\cos(d_X(x,y)).$$ We refer to the points $[x,0]$ and $[x,\pi]$ as the \textit{tips} of $S(X)$.
\end{definition}
These definitions are motivated by fact that the spherical suspension of an $n$-sphere $\sphere^n$ is isometric to $\sphere^{n + 1}$, and the cone over $\sphere^n$ is isometric to $\RR^{n + 1}$. 

By the dimension $\dim A$ of an Alexandrov space $A$ we mean its Hausdorff dimension. $\dim A$ is an integer and behaves well (for example, open subsets always have the same dimension as the whole space). For more details see \cite{burago-gromov-perelman92} or \cite{burago-burago-ivanov01}. In \cite{burago-gromov-perelman92} the following Lemma is proven:
\begin{lemma}\label{cone and suspension}
 Let $X$ be an Alexandrov space with curvature bounded below satisfying $\operatorname{diam}(X) \leq \pi$.  Then the following are equivalent:
 \begin{itemize}
  \item[(i)] $\curv X \geq 1$.
  \item[(ii)] $\curv C(X) \geq 0$ and $\dim X \geq 2$.
	\item[(iii)] $\curv S(X) \geq 1$ and $\dim X \geq 2$.
 \end{itemize}
\end{lemma}
The geometric properties near a point $p$ of an Alexandrov space $A$ are reflected by its \textit{tangent cone} $T_pA$ at $p$ and the \textit{space of directions} $\Sigma_pA$ at $p$, which resemble the tangent space and its unit sphere of a Riemannian manifold. As the name already indicates $\Sigma_pA$ is the (metric completion of the) space of directions of geodesics emanating from $p$, with metric being the angle (for a precise definition see again \cite{burago-gromov-perelman92} or \cite{burago-burago-ivanov01}). The tangent cone is then defined as the cone over $\Sigma_pA$. We conclude this chapter with the following basic and well known result:

\begin{lemma}
Let $\G$ be a compact Lie group acting isometrically on a complete Riemannian manifold $(M,g)$ satisfying $\sec \geq c$. Then $(M^*,d)$ is an Alexandrov space with $\on{curv} \geq c$. For $p \in M^*$ with $p = \pi(\hat p)$, the tangent cone $T_{p}M^*$ and the space of directions $\Sigma_{p}M^*$ at $p$ are respectively isometric to $N_{\hat p}\G(\hat p)/\G_{\hat p}$ and $N^1_{\hat p}\G({\hat p})/\G_{\hat p}$.
\end{lemma}
 Here $N^1_{\hat p}\G({\hat p})$ denotes the unit sphere of $N_{\hat p}\G({\hat p})$. Finally note that a \textit{regular point} of an Alexandrov space is a point whose tangent cone is isometric to $\RR^n$ equipped with the standard metric. From this lemma we see that the regular points of a quotient space $M^*$ are precisely the regular points as defined in the previous section.
\\ \\
Some more results on Alexandrov spaces are presented in chapter $3$.
\chapter{\texorpdfstring{Resolving the singularities of $M^4/\crcl$}{Resolving the singularities of M/circle}}\label{chapter resolving}
In this chapter we proof our first main theorem.
\begin{theorem}\label{resolution thm}
  Let $\mathsf S^1$ act isometrically and with only isolated fixed points on a simply connected, compact, nonnegatively curved $4$-manifold $M$ with quotient space $(M/\crcl,d)$. Then the following is true:
\begin{itemize}
 \item[(a)] There exists a sequence of smooth, positively curved Riemannian metrics $(g_n)_{n \in \NN}$ on $M/\crcl$ such that  $(M/\crcl, g_n)$ has Gromov-Hausdorff limit $(M/\crcl,d)$.
 \item[(b)]  Assume additionally that the nonregular part of $M/\crcl$ contains a closed curve $c$. Then there exists a sequence of smooth, positively curved orbifold metrics $(h_n)_{n \in \NN}$ on $M/\crcl$ such that $(M/\crcl,h_n)$ has Gromov-Hausdorff limit $(M/\crcl,d)$ and the only singularities of $h_n$ are given by $\mathbb Z_2$-singularities along $c$.
\end{itemize}
\end{theorem}
For that we first prove the following
\begin{theorem}\label{umformulierung}
 Let $M$ be a closed, simply connected, $4$-dimensional Riemannian manifold admitting an isometric $\crcl$-action with isolated fixed points only. Assume that $\curv M^* \geq 0$ $(\curv M^* > 0)$. Then there exists a sequence of smooth Riemannian metrics $g_n$ on $M^*$ with $\sec_{g_n} \geq - 1/n$ $(\sec_{g_n} > 0)$ such that $(M^*,g_n)$ converges in Gromov-Hausdorff sense to $(M^*,d)$.
\end{theorem}

Note that part (a) of Theorem \ref{resolution thm} follows if $M$ has positive curvature. If $M$ only admits an invariant metric of nonnegative curvature, we use the Ricci flow to improve the curvature bounds of the approximating sequence. More precisely, in section \ref{ricci}, we use a result of Simon \cite{simon09} which tells us that there exists a Ricci flow $g_t$ on $M^*$ with singular initial metric $g_0 = g$ and $\sec_{g_t} \geq 0$ for $t > 0$. Since $M^*$ is simply connected, it follows by Hamilton \cite{hamilton 82}, \cite{hamilton86} that $g_t$ has in fact positive curvature if $t$ is positive. This implies Theorem \ref{resolution thm}, part (a).

In section \ref{final proof} we sketch how part (b) of Theorem \ref{resolution thm} is proven. The arguments are essentially the same as for part (a), modulo considering the two fold branched cover along the curve c and using some ideas from \cite{grove-wilking13}.

The proof of Theorem \ref{umformulierung} is given in sections \ref{quotient} to \ref{resolution}. In section \ref{quotient} we first discuss the basic geometric and topological structure of $M^*$; in short $M^*$ is a simply connected $3$-manifold whose nonregular part consists of a finite collection of isolated fixed points, which are possibly joined by singular geodesic arcs, each of these arcs having constant isotropy group $\ZZ_k$ (compare Figure \ref{triangle}). 
\begin{figure}
\centering
\def\svgwidth{0.2\textwidth}
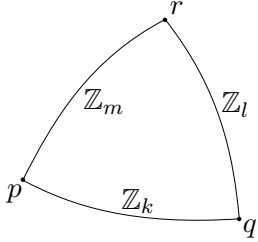
\caption{\small{The fixed points $p,q$ and $r$ in $M^*$ are connected by nonregular arcs of constant types.}}
\label{triangle}
\end{figure}
Also in section \ref{quotient} we give a detailed description of the possible spaces of directions at nonregular points of $M^*$. 

In section $2.2$ we show that we may assume that $g$ admits a polar $\crcl$-action in a neighborhood of the nonregular part and moreover, that there exists some $\rho > 0$ such that for every fixed point $p \in M^*$ the open ball $B_\rho(p)$ is isometric to the corresponding ball at a tip of the spherical suspension of $\Sigma_pM^*$. This is done via approximating the metric $g$ by a sequence of metrics with the desired properties and controlled lower curvature bounds.

For simplicity we loosely refer to a sequence of smooth metrics approximating some given singular metric on a space $X$ in Gromov-Hausdorff sense as a \textit{resolution of the singularities of $X$}. The resolution of the singularities of $M^*$ is then carried out in section \ref{resolution} in several major steps. Let $E^*$ denote the projection of $E$, the set of exceptional orbits, and $F$ denote the set of fixed points (which may be viewed as a subset of $M$ as well as $M^*$). We first resolve the singularities of $E^*$ in sections \ref{resolving brho} to \ref{resolving E}. For this we consider a singular arc in $E^*$ parametrized by a geodesic $\gamma$ connecting two fixed points $p$ and $q$. We cover this arc by coordinate neighborhoods $B_{\rho}(p)$, $B_{\rho}(q)$ and $U$, where $U$ is an appropriate neighborhood of the interior of $\gamma$ which is a good Riemannian orbifold. Then resolutions of the points of $\gamma$ in each of these neighborhoods are constructed in a consistent way to be able to glue these resolutions together to obtain a resolution of $\gamma$ and eventually glue this resolution back to $M^*$. This construction can be performed independently along every arc in $E^*$ yielding a resolution of the singularities of $E^*$. A more detailed overview of this construction is given in the introduction to section \ref{resolution}.

We resolve the remaining isolated singularities of $F$ in section \ref{smooth fp}. For that we isometrically  embed neighborhoods of the fixed points into $\sphere^4$ and then transfer the problem to $\RR^4$. In $\RR^4$ we can construct a smooth approximation via a convolution.

\section{\texorpdfstring{Geometry and topology of $M^4/\crcl$ and $\sphere^3/\crcl$}{Geometry and topology of M/circle and sphere/circle}}\label{quotient}

\subsection{\texorpdfstring{Geometry and topology of $M^4/\crcl$}{Geometry and topology of M/circle}}

Our arguments are based on the following result of Fintushel from \cite{fintushel}. 
\begin{lemma}\label{fintushel}
Let $M$ be a closed, simply connected, $4$-dimensional manifold with a smooth $\crcl$-action with isolated fixed points only. Then $M^*$ is a simply connected $3$-manifold, without boundary, and $F$ is nonempty and finite. The closure of $E^*$ is a finite collection of polyhedral arcs and simple closed curves. Components of $E^*$ are open arcs on which orbit types are constant and whose closures have distinct endpoints in $F$.
\end{lemma}
\begin{remark}
 If we assume that the action is isometric with respect to a Riemannian metric on $M$, it follows from the slice representation at a fixed point $p$ that the angle of two arcs in $E^*$ meeting at $p \in F$ is $\pi/2$ (note also that at most two such arcs meet at $p$). Moreover, if $M$ admits an invariant metric of nonnegative or positive curvature, it follows that there are at most 4 or 3 fixed points, respectively (cf. \cite{wilking07}, section 2).
\end{remark}
The infinitesimal geometry of $M^*$ at a singular point $p \in M^*$ is determined by the space of directions $\Sigma_pM^*$ at $p$. In the following we give a precise geometric description of the spaces of directions that possibly occur.

First let $p = \pi(\hat p)$ be not a fixed point. Thus $\hat p$ has isotropy group $\ZZ_k$, for $k \geq 1$, and the space $\Sigma_pM^*$ is isometric to the quotient space of a normal sphere  to $\crcl(\hat p) \cong \sphere^1$ by the action of $\crcl_{\hat p} = \ZZ_k$: 
\begin{align}\label{directions at gamma}
\Sigma_pM^* = \sphere^2/\ZZ_k.
\end{align}
Let $(r,\theta)$ denote polar coordinates on $D_{\pi} \subset \RR^2$, the closed disk of radius $\pi$ in $\RR^2$. Also let $D_\pi/\partial D_\pi := D_{\pi}/\sim$, where $x \sim y$ if $x$ and $y$ belong to the boundary of $D_\pi$. Observe that $\ZZ_k$ is acting via rotation on the round $\sphere^2$. Then the following proposition follows easily.
\begin{proposition}\label{directions orbi}
 Let $p \in M^*$ have isotropy group $\ZZ_k$ for $k \geq 1$. Then $\Sigma_pM^*$ is given by the spherical suspension of a circle of perimeter $2\pi/k$. Equivalently, $\Sigma_pM^*$ is isometric to $D_\pi/\partial D_\pi \cong \sphere^2$ equipped with the metric 
$$dr^2 + k^{-2}\sin^2(r)d\theta^2.$$
\end{proposition}
Note that this metric is singular at $r = 0$ and $r = \pi$ if and only if $k > 1$. The geometry of the space of directions $\Sigma_pM^*$ at a fixed point $p$ is more complicated. We give a detailed description in the following section.
\subsection{\texorpdfstring{Geometry and topology of $\sphere^3/\crcl$}{Geometry and topology of sphere/circle}}\label{sectionS^3/crcl}
Let $\crcl$ act isometrically on $M^4$ and $p$ be an isolated fixed point of $M^* = M^4/\crcl$. The space of directions at $p$ is isometric to $\sphere^3/\crcl$, where $\sphere^3$ is equipped with the standard round metric and $\crcl$ is acting orthogonally and, since $p$ is isolated,  without fixed points on $\sphere^3 \subset \CC^2$. Writing $\crcl = \{z \in \CC \mid \vert z \vert = 1\}$, with respect to some orthonormal basis the action of $\crcl$ on $\sphere^3$ is induced by the action of $\crcl$ on $\CC^2$ given by $z.(u,v) = (z^{m_-}u,z^{m_+}v)$ for $m_-,m_+ \in \NN$. Assuming the action to be effective, it follows that $m_-$ and $m_+$ are coprime. In the case $m_- = m_+ = 1$ we obtain the well-known Hopf action and $\sphere^3/\crcl$ is isometric to $\sphere^2(1/2)$, the sphere of radius $1/2$, and has constant curvature $4$. In general the action is almost free with isotropy group $\ZZ_{m_-}$ corresponding to the orbit of $(1,0)$, isotropy group $\ZZ_{m_+}$ corresponding to the orbit of $(0,1)$, and all other orbits have trivial isotropy group. It follows that the quotient space $\sphere^3/\crcl$ is homeomorphic to $\sphere^2$, and there are at most two singular points, given by the orbits $\crcl(1,0)$ and $\crcl(0,1)$. It is easy to see that $\measuredangle(\crcl(0,1),\crcl(1,0)) = \pi /2 = \operatorname{diam}(\sphere^3/\crcl)$.

Let $D_{\pi/2}$ denote the closed disk of radius $\pi/2$ in $\RR^2$ equipped with polar coordinates $(\theta,\alpha)$, where $\theta$ corresponds to radial, and $\alpha$ to angular direction.
\begin{proposition}\label{directions}
The quotient space $\sphere^3/ \crcl$ is isometric to $D_{\pi/2}/\partial D_{\pi/2} \cong \sphere^2$ equipped with the metric
\begin{align}\label{this}
d \theta^2 + R^2(\theta)d \alpha^2
\end{align}
for a smooth function $R : [0,\pi / 2] \rightarrow \RR$ satisfying 
$$R(0) = R(\pi / 2) = 0,$$
$$R'(0) = m_-^{-1},$$
$$R'(\pi /2) = m_+^{-1}$$
and
$$R^{(2k)}(0) = R^{(2k)}(\pi/2) = 0$$
for all $k \in \NN$. The metric is singular at $\theta = 0$ ($\theta = \pi/2$) if and only if $m_- > 1$ ($m_+ > 1$). Further, there exists a constant $a = a(m_-,m_+) > 0$ such that 
$$\sec(\theta,\alpha) > 1 + a,$$
whenever $0 < \theta < \pi/2$. In fact, the following identities hold:
\begin{align}
R(\theta) &= \frac{\sin(\theta)\cos(\theta)}{(m_+^2\sin^2(\theta) + m_-	^2\cos^2(\theta))^{\frac1 2}}\label{Rrrrr}\\
\sec(\theta,\alpha) = -R''(\theta)/R(\theta) &= 1 + \frac{3m^2_-m^2_+}{(m^2_-\cos^2(\theta) + m^2_+\sin^2(\theta))^2}\label{K,jo}
\end{align}
\end{proposition}
\begin{proof}
Consider the isometric action of $\crcl_2$ on $\sphere^3 \subset \CC^2$ given by $z.(u,w) = (u,zw)$. This action induces via left multiplication an effective isometric $\crcl_2$-action on $\sphere^3/\crcl$ with fixed point set $\{[1,0],[0,1]\}$. Let $c : [0,\pi/2] \to \sphere^3/\crcl$ be an arc length geodesic from $[1,0]$ to $[0,1]$. Introduce coordinates $(\theta,\alpha) \mapsto \alpha.c(\theta)$ with $(\theta,\alpha) \in [0,\pi/2] \times \crcl$. Due to the above action we see that the metric of $\sphere^3/\crcl$ is then given by 
$$d\theta^2 + R^2(\theta)d\alpha^2$$
for some smooth function $R :[0,\pi/2] \to \RR$. It is possible to derive formula \eqref{Rrrrr} directly from the geometry of the action of $\crcl$ on $\sphere^3$. But instead we first proof \eqref{K,jo}, using the Riemannian submersion $\sphere^3 \to \sphere^3/\crcl$ defined on the regular part of $M$, and from \eqref{K,jo} we deduce formula \eqref{Rrrrr}:

Let $x \in \sphere^3/\crcl$ be a regular point and $v,w \in T_x(\sphere^3/\crcl)$ be orthonormal. From the O'Neill formula
$$\sec(v \wedge w) = \sec(\overline v \wedge \overline w) + \frac 3 4 \vert \vert [\overline v, \overline w]^v \vert \vert ^2,$$
where $\overline v$ and $\overline w$ are horizontal lifts of $v$ and $w$ respectively and $\ ^v$ denotes the vertical part of a vector. Let 
$$\gamma(\theta) = (\cos(\theta),\sin(\theta)) \in \sphere^3 \subset \CC^2$$
for $\theta \in [0,\pi/2]$. $\gamma$ defines a horizontal arc length geodesic from $(1,0)$ to $(0,1)$. Let $T$ be the vector field on $\sphere^3/\crcl$ dual to $d\theta$. Then $T(c(\theta)) = \dot c(\theta)$. We may choose $\gamma$ such that $\pi \circ \gamma = c$. Hence
\begin{align}\label{TtT}
d \pi \dot \gamma (\theta) = \dot c(\theta) = T(c(\theta)).
\end{align}
The vertical space at $(z,w) \in \sphere^3$ is given by the $\RR$-span of the vector $$\frac{d}{dt}_{|t = 0} (e^{it}.(z,w)) = \frac{d}{dt}_{|t = 0} (e^{itm_-}z,e^{itm_+}w))  = i(m_-z,m_+w) =: V(z,w).$$
Along $\gamma$ let 
$$W(\theta) =  i(-m_+\sin(\theta),m_-\cos(\theta)).$$
By construction, $W(\theta)/||W(\theta)||$, $\gamma(\theta), \dot\gamma(\theta)$ and $V(\gamma(\theta))/||V(\gamma(\theta))||$ are pairwise orthonormal and
\begin{align}\label{WwW}
 S(\theta) := d\pi W (\theta)/||W(\theta)||
\end{align}
is a vector field along $c$ orthonormal to $T$ since $\pi$ is a Riemannian submersion. Set $\eta(\theta) := (m_+^2\sin^2(\theta) + m_-	^2\cos^2(\theta))^{\frac1 2} = ||W(\theta)|| = ||V(\gamma(\theta))||$. Using \eqref{TtT} we calculate
\begin{align}
\sec(c(\theta)) &= K(\overline T \wedge \overline S)(\gamma(\theta)) + \frac{3}{4}||[\overline T, \overline S]^v||^2(\gamma(\theta))\nonumber\\
 &= 1 + \frac{3}{4\eta^2(\theta)}\langle [\overline T, \overline S], V \rangle^2(\gamma(\theta))\nonumber\\
 &= 1 + \frac{3}{4\eta^2(\theta)}\langle \nabla_{\overline T}\overline S - \nabla_{\overline S}\overline T, V\rangle^2 (\gamma(\theta)) \nonumber\\
 &= 1 + \frac{3}{4\eta^2(\theta)}(-\langle \overline S,\nabla_{\overline T}V\rangle + \langle \overline T,\nabla_{\overline S}V\rangle)^2 (\gamma(\theta)) \nonumber\\
&= 1 + \frac{3}{4\eta^4(\theta)}(-\langle W,\nabla_{\dot \gamma}V\rangle + \langle \dot \gamma,\nabla_{W}V\rangle)^2(\theta). \nonumber
\end{align}
Here $\nabla$ denotes the Levi-Civita connection of $\sphere^3$. From the definitions we conclude
$$\nabla_{\dot \gamma(\theta)}V = i(-m_-\sin(\theta),m_+\cos(\theta))$$
and
$$\nabla_{W(\theta)}V = (m_-m_+\sin(\theta),-m_-m_+\cos(\theta)).$$
Hence
\begin{align*}
-\langle W,\nabla_{\dot \gamma}V\rangle(\theta) = \langle \dot \gamma,\nabla_{W}V\rangle(\theta) = -m_-m_+.
\end{align*}
Altogether it follows \eqref{K,jo}, i.e.
$$\sec(c(\theta)) = 1 + \frac{3m^2_-m^2_+}{(m^2_-\cos^2(\theta) + m^2_+\sin^2(\theta))^2}.$$
On the other hand, we see from Proposition \ref{curvature} that the curvature is also given by the formula
$$\sec(c(\theta)) = -R''(\theta)/R(\theta).$$
It is clear that $R(0) = 0$. Moreover, it follows from the slice theorem that the tangent cone of $\sphere^3/\crcl$ at $[1,0]$ is given by $\RR^2/\ZZ_{m_-}$ where $\ZZ_{m_-}$ acts via rotation. Consequently,  $R'(0) = 1/m_-$.
Thus we have the second order ODE
$$R''(\theta)/R(\theta) = \frac{-3m^2_-m^2_+}{(m^2_-\cos^2(\theta) + m^2_+\sin^2(\theta))^2} - 1$$
with initial conditions $R(0) = 0$ and $R'(0) = 1/m_-.$
Standard calculations show that the unique solution on the interval $[0,\pi/2]$ is given by
$$R = \frac{\sin(\theta)\cos(\theta)}{(m_+^2\sin^2(\theta) + m_-	^2\cos^2(\theta))^{\frac1 2}}.$$
Since $R$ is an odd function, we eventually see that all even derivatives at $0$ vanish.
\end{proof}
\vspace{20pt}
\begin{figure}[h!]
\centering
\def\svgwidth{0.4\textwidth}
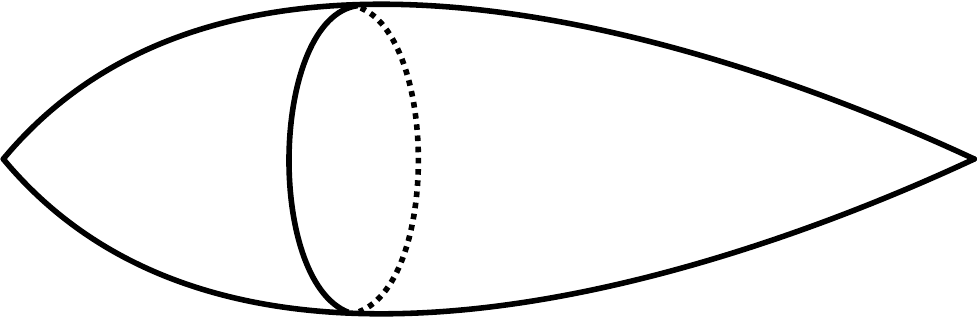
\caption{\small{A sketch of $\sphere^3/\crcl$ with $m_- = 2$, $m_+ = 3$.}}
\label{figure:directions}
\end{figure}
\section{\texorpdfstring{$\crcl$-symmetry near the singular set}{circle-symmetry near the singular set}}\label{section symmetry}
 
Let $M$ be a simply connected Riemannian $3$-manifold which contains a closed geodesic $\gamma$ whose image is a smooth submanifold. In this section we show that it is possible to deform the metric in a neighborhood of $\gamma$ with an arbitrary small decrease of lower curvature bounds, such that the new metric admits a polar $\crcl$-action in a neighborhood of $\gamma$ whose fixed point set is the image of $\gamma$. Using this result we approximate the quotient space of a simply connected Riemannian $4$-manifold by an isometric $\crcl$-action via metrics which itself admit polar $\crcl$-actions in a neighborhood of the nonregular part. The ideas in this section are in parts similar to the ones of \cite{dyatlov}. Indeed we were able to significantly simplify the proofs in this section using the gluing lemma \ref{gluing} which, as discussed in section \ref{section gluing}, relies on Lemma 4.3 from \cite{dyatlov}.
\\ \\
We use the following construction for the upcoming proofs: Let $\gamma : \RR \to M$ be a geodesic in a Riemannian manifold $M$, and denote by $P_t : T_{\gamma(0)}M \to T_{\gamma(t)}M$ parallel translation along $\gamma_{[0,t]}$. Given standard coordinates $\{x_i\}$ of $\RR^k = (N_{\gamma(0)}\gamma,g)$, the normal space to $\gamma$ at $\gamma(0)$, we define the map
\begin{align*}
f : \RR \times \RR^k &\to M \\
(t,x_1, \dots, x_k) &\mapsto \exp_{\gamma(t)}(P_t(x_1, \dots, x_k)).\nonumber
\end{align*}
Assuming that $\gamma_{[0,l]}$ is injective for some $l \in \RR$ there exists $\epsilon > 0$ such that $f$ is a diffeomorphism onto its image when restricted to the set 
$$\{(t,x_1 \dots, x_k) \mid 0 \leq t \leq l,\ \Sigma_{i = 1}^k x^2_i < \epsilon\}.$$
Then $f$ induces local coordinates, which are denoted by
\begin{align}\label{fermi}
 (t,x_1, \dots, x_k).
\end{align}

\begin{definition}
The coordinates $(t,x_1 \dots, x_k)$ are called \textit{Fermi coordinates adopted to $\gamma_{[0,l]}$.}
\end{definition}
We show that the metric is constant along $\gamma$ in these coordinates and its first derivatives vanish:

Denote by $\partial_t$ and $\partial_{x_i}$ the induced coordinate fields. First, it is clear that \begin{align}\label{une fille}
 g(\partial_{x_i}, \partial_{x_j}) = \delta_{ij},\ g(\partial_t,\partial_t) = 1,\ g(\partial_t,\partial_{x_i}) = 0
\end{align}
at points of $\gamma$. To calculate the derivatives let $g_{ij} = g(\partial_{x_i},\partial_{x_j})$, $g_{it} = g(\partial_{x_i},\partial_t)$, $g_{tt} = g(\partial_t,\partial_t)$ and $p = \gamma(t)$ for some $t \in [0,l]$. From \eqref{une fille} it is clear that
\begin{align*}
 \partial_tg_{it}(p) = 0 = \partial_tg_{tt}(p)
\end{align*}
for all $i \in \{1, \dots, k\}$.
 We claim that 
 $$\nabla_{\partial_{x_i}}\partial_{x_j}(p) = 0$$
 for all $i,j$: In a neighborhood of $p$ consider normal coordinates $(\tilde t, \tilde x_1, \dots, \tilde x_k)$  with respect to the frame $\{\partial_t(p), \partial_{x_1}(p), \dots, \partial_{x_k}(p)\}$. Then $\partial_{\tilde x_i} = \partial_{x_i}$ along the submanifold $N := \{\tilde t = 0\}$. Therefore $0 = \nabla_{\partial_{\tilde x_i}}\partial_{\tilde x_j}(p) = \nabla_{\partial_{x_i}}\partial_{x_j}(p)$ since the value of $\nabla_{\partial_{x_i}}\partial_{x_j}$ depends only on $\partial_{x_i}(p) = \partial_{\tilde x_i}(p)$ and the values of $\partial_{x_j}$ along $N$ because $\partial_{x_i}$ is tangent to $N$.
 
 Now
\begin{align*}
 \partial_{x_i}g_{jt}(p) = g_p(\nabla_{\partial_{x_i}}\partial_{x_j},\partial_t) + g_p(\partial_{x_j},\nabla_{\partial_{x_i}} \partial_t) = 0,
\end{align*}
since $\nabla_{\partial_{x_i}}\partial_t(p) = \nabla_{\partial_t}\partial_{x_i}(p) = 0$ by parallelity of $\partial_{x_i}$ along $\gamma$. So all its first derivatives vanish.
\\ \\
Now we state and prove the mentioned results. We divided the first step into the two separate Lemmas \ref{polar} and \ref{symmetry along geodesic} to make the proof more accessible. 

\begin{lemma}\label{polar}
 Let $(M,g)$ be a 3-dimensional Riemannian manifold admitting an isometric $\crcl$-action with a compact, $1$-dimensional fixed point component $\Gamma$. Let $\operatorname{sec}_M \geq c$. Then for all $\epsilon > 0$ there exists an $\crcl$-invariant Riemannian metric $\tilde g$ on $M$ with $\operatorname{sec}_{\tilde g} \geq c - \epsilon$ and an open neighborhood $U$ of $\Gamma$ such that $g$ and $\tilde g$ coincide on $M \setminus B^g_\epsilon (\Gamma)$, and the action restricted to $U$ is polar.
 \end{lemma}
 
\begin{proof}
Let $\epsilon > 0$. $\Gamma$ is a smooth embedded submanifold which may be parametrized by a curve $t \mapsto \gamma(t)$ with $||\dot \gamma || = 1$. Since the action is isometric, $\gamma : \RR \to M$ is a closed geodesic. We may assume that the normal exponential map of $\Gamma$ restricted to the set of normal vectors of length $< \epsilon$ is injective, otherwise we choose a smaller $\epsilon$. Let $(r,\theta)$ denote polar coordinates of $B_{\epsilon}(0) \subset \dot \gamma(0)^\perp$ with radial direction $r$ and $\theta$ corresponding to the action of $\crcl$. Restricting $\gamma$ to some subinterval $I$, such that it is injective we denote by $(t,r,\theta)$ the induced Fermi coordinates. Then $g$ is given as
\begin{align}\label{themetric}
 g(t, r,\theta) = \rho^2(t,r)dt^2 + dr^2 + \varphi^2(t,r)d\theta^2 + g_{t \theta}(t,r)(dt \otimes d\theta + d\theta \otimes dt). 
\end{align}
Note that the coefficient functions $\rho$ and $\varphi$ are independent of $\theta$, since the $\crcl$-action is isometric. Now define
\begin{align*}
h = \rho^2(t,r)dt^2 + dr^2 + \varphi^2(t,r)d\theta^2,
\end{align*}
that is we simply set the $g_{t \theta}$ coefficient to $0$. Clearly  $h$ is defined independently of the choice of $I$. Therefore, $h$ is a smooth metric on $B_\epsilon(\Gamma)$ such that the $\mathsf S^1$-action is isometric with respect to $h$. Choose $u \in \dot \gamma(0)^\perp$ with $||u|| = 1$. It is easy to see that the set $\{(t,su) \in \RR \times B_\epsilon(0) \mid s \in ]-\epsilon,\epsilon[\}$ gives a section, so the $\crcl$-action on $(B_\epsilon(\Gamma),h)$ is polar.

We prove that $g$ and $h$ coincide up to first order along $\Gamma$ and  that $\operatorname{sec}_h \geq c$ along $\Gamma$. Then the existence of our desired metric $\tilde g$ follows from the gluing lemma \ref{gluing}.

In fact, we show that $g$ and $h$ coincide up to second order along $\Gamma$ (so the curvature tensors of $g$ and $h$ coincide at points of $\Gamma$ and the proof is finished): It is convenient to use coordinates $(x,y)$ of $B_\epsilon(0)$ with respect to an orthonormal frame and the induced coordinates $(t,x,y)$ with corresponding coordinate fields $\partial_t, \partial_x$ and $\partial_y$. We write $g(\partial_x,\partial_t) = g_{xt}$, $h(\partial_x,\partial_t) = h_{xt}$ and analogously for the other coefficients of the two metrics. From the definition of $h$ it follows
$$h = g - g_{xt}(dx \otimes dt + dt \otimes dx) - g_{yt}(dy \otimes dt + dt \otimes dy).$$
It is clear that for all $p \in \Gamma$
\begin{align*}
 g_{xt}(p) = g_{yt}(p) = 0
\end{align*}
and therefore
\begin{align*}
 g_p = h_p.
\end{align*}

Since in our Fermi coordinate system all first derivatives of $g_{xt}$ and $g_{yt}$ vanish at points of $\Gamma$ and $h_{xt} = h_{yt} = 0$, all first derivatives of $g$ and $h$ also coincide at points of $\Gamma$.

To see that the second derivatives of $g$ and $h$ coincide at $p \in \Gamma$, again, we need to verify that all second derivatives of $g_{xt}$ and $g_{yt}$ vanish at $p$. Since $g$ coincides with $h$ up to first order along $\gamma$, it is clear that all second derivatives of $g_{xt}$ and $g_{yt}$ involving a differentiation in $t$-direction vanish. So it remains to prove
\begin{align}
 \partial^2_{\alpha \beta }g_{\delta t} = 0 \label{blub}
\end{align}
for any combination
\begin{align}
 (\alpha, \beta, \delta) \in \{x,y\}^3\label{blab}.
\end{align}
For example we have
\begin{align}
 \partial^2_{xy}(p)g_{xt} = \partial_x(p)(g(\nabla_{\partial_y}\partial_x,\partial_t) + g(\partial_x,\nabla_{\partial_y}\partial_t)) = g_p(\nabla_{\partial_x}\nabla_{\partial_y}\partial_x,\partial_t) + g_p(\partial_x,\nabla_{\partial_x}\nabla_{\partial_y}\partial_t),\nonumber
\end{align}
where the second equality holds since $\nabla_{\partial_x}\partial_t(p) = \nabla_{\partial_x} \partial_x (p)= 0$, again due to our choice of coordinates. Since $\mathsf S^1$ acts isometrically, the map $\tau : B_\epsilon(\Gamma) \to B_\epsilon(\Gamma)$ defined via
\begin{align}\label{reflection}
\tau(\exp_p(v)) = \exp_p(-v),
\end{align}
for $v \in N \Gamma$ of length less than $\epsilon$, is an isometry. Also $\tau$ satisfies $d\tau \partial_x = -\partial_x$, $d\tau \partial_y = -\partial_y$ and $d\tau \partial_t = \partial_t$. Thus 
\begin{align}
\partial^2_{xy}(p)g_{xt} = &g_p(\nabla_{\partial_x}\nabla_{\partial_y} \partial_x,\partial_t) + g_p(\partial_x,\nabla_{\partial_x}\nabla_{\partial_y}\partial_t)\\
= &g_p(d\tau \nabla_{\partial_x}\nabla_{\partial_y}\partial_x, d\tau \partial_t) + g_p(d\tau \partial_x, d\tau \nabla_{\partial_x}\nabla_{\partial_y}\partial_t)\nonumber \\
= &g_p(\nabla_{d\tau {\partial_x}}\nabla_{d\tau {\partial_y}}d\tau \partial_x,d\tau\partial_t) + g_p(d\tau \partial_x,\nabla_{d\tau {\partial_x}}\nabla_{d\tau {\partial_y}}d\tau \partial_t)\nonumber \\
= &-(g_p(\nabla_{\partial_x}\nabla_{\partial_y}\partial_x,\partial_t) + g_p(\partial_x,\nabla_{\partial_x}\nabla_{\partial_y}\partial_t)).\nonumber
\end{align}
So both sides are $0$. This argument works as well for other combinations of \eqref{blab} and consequently \eqref{blub} holds.
\end{proof}

We use this lemma to prove a more general result.



\begin{lemma}\label{symmetry along geodesic}
 Let $(M,g)$ be a $3$-dimensional Riemannian manifold with $\sec \geq c$. Assume there exists a  closed geodesic $\gamma$ whose image $\Gamma$ is a smooth submanifold and the normal bundle $N\gamma$ is trivial. Then for all $\epsilon > 0$ there exists a Riemannian metric $\tilde g$ on $M$ with $\sec_{\tilde g} \geq c - \epsilon$, which coincides with $g$ on $M\setminus B^g_\epsilon(\Gamma)$ and admits a polar $\mathsf S^1$-action in a neighborhood of $\Gamma$ with $\operatorname{Fix}(\crcl) = \Gamma$.
\end{lemma}
\begin{proof} Since $N\gamma$ is trivial, there exists a smooth $\mathsf S^1$-action on the normal bundle $N \Gamma$ obtained by orthogonally rotating the fibers. We may assume that the normal exponential map of $\Gamma$ restricted to the set of normal vectors of length $< \epsilon$ is injective. Via the exponential map we define a smooth $\mathsf S^1$-action on $B^g_\epsilon(\Gamma)$. Let $\crcl = \RR/2\pi\ZZ$ be equipped with the induced volume form $d\theta$ (so that $\crcl$ has length $2\pi$). For $p$ in $B^g_\epsilon(\Gamma)$ set
\begin{align} \label{definition of h}
 h_p := \frac 1 {2\pi} \int_{\crcl}\theta^*g_p\ d\theta.
\end{align}
Clearly $\mathsf{S}^1$ is acting isometrically with respect to $h$. We prove that $g$ and $h$ coincide up to first order along $\gamma$ and that $h$ has curvature bounded below by $c - \epsilon/2$ in some small neighborhood of $\Gamma$. Then the existence of the desired metric $\tilde g$ follows from the gluing lemma \ref{gluing}.
\\
As in the previous proof, let $\Gamma$ be parametrized by an arc length geodesic $\gamma$ and $(t,x,y)$ denote Fermi coordinates adopted to $\gamma$. At points of $\Gamma$ the slice representation of $\mathsf S^1$ is orthogonal with respect to $g$ as well as $h$. Thus $g = h$ at points of $\Gamma$ from \eqref{definition of h}. Let $p \in \Gamma$. To see that the first derivatives of $h$ and $g$ coincide at $p$, recall that
\begin{align*}
 \nabla^g_{\partial_x}\partial_y(p) = \nabla^g_{\partial_x}\partial_x(p) = \nabla^g_{\partial_y} \partial_y(p) = \nabla^g_{\partial_t}\partial_x(p) = \nabla^g_{\partial_t} \partial_y(p) = \nabla^g_{\partial_t}\partial_t(p) = 0
\end{align*}
in our coordinates. Rotating our coordinate system, $\nabla_{d\theta\partial_i}d\theta \partial_j(p) = 0$ analogously follows for all $\theta \in \crcl$ and hence
\begin{align*}
 \partial_i(p)h_{jk} &= \int_{\crcl}(d\theta_p\partial_i)g(d\theta \partial_j,d\theta \partial_k)d\theta\\
 &= \int_{\crcl}g_p(\nabla_{d\theta\partial_i}d\theta \partial_j,d\theta \partial_k) + g_p(d\theta \partial_j, \nabla_{d\theta \partial_i}d\theta \partial_k)d\theta = 0
\end{align*}
for $\{i,j,k\} \subset \{t,x,y\}$ and $p \in \Gamma$. Thus $g$ and $h$ coincide up to first order along $\Gamma$.

It remains to prove that in some neighborhood of $\Gamma$ the curvatures of $h$ are bounded below by $c - \epsilon/2$. For this it suffices to show that the sectional curvatures of $h$ are bounded below by $c$ at points of $\Gamma$: Consider the curvature tensor $R$ of $g$ and let $p \in \Gamma$. The map
\begin{align*}
 R_{.\ \partial_t}\partial_t :\ &T_pM \rightarrow T_pM\\
&v \mapsto R_{v \partial_t}\partial_t
\end{align*}
is selfadjoint with respect to $g$, and $\partial_t$ is an eigenvector with eigenvalue $0$. There is an orthonormal basis of $T_pM$ consisting of eigenvectors, which we may assume to be $\{\partial_t, \partial_x, \partial_y\}$, after possibly rotating our coordinate system. So at $p$ we have
$$ (g(R_{\partial_i \partial_t}\partial_t,\partial_j))_{ij} =  \left( \begin{array}{ccc} \lambda & 0 & 0 \\ 0 & \mu & 0 \\ 0 & 0 & 0 \end{array} \right), \{i,j\} \subset \{t,x,y\}$$
with $\lambda, \mu \geq c$. Let $R^h$ denote the curvature tensor of $h$. Again $\partial_t$ is a $0$-eigenvector of $v \mapsto R^h_{v \partial_t}\partial_t$, and since $\crcl$ acts isometrically, we see that at $p$
\begin{align}\label{Rhmatrix}
 (h(R^h_{\partial_i \partial_t}\partial_t,\partial_j))_{ij} =  \left( \begin{array}{ccc} \nu & 0 & 0 \\ 0 & \nu & 0 \\ 0 & 0 & 0 \end{array} \right), \{i,j\} \subset \{t,x,y\},
\end{align}
for some $\nu \in \RR$. It is straightforward to calculate that
\begin{align*}
 \nu = h_p(R^h_{\partial_x \partial_t}\partial_t,\partial_x) = \int_{\crcl}g_p(R_{d\theta \partial_x \partial_t}\partial_t,d\theta \partial_x)d\theta \geq c.
\end{align*}
Let us consider the curvature operator $\mathcal R^h_p$ of $h$ with respect to the basis 
\begin{align*}
 \{\partial_t \wedge \partial_x, \partial_t \wedge \partial_y, \partial_x \wedge \partial_y\}
\end{align*}
of $\Lambda^2T_pM$. As in the last section of the proof of Lemma \ref{polar} we obtain an isometric involution fixing $\Gamma$ (see \eqref{reflection}) and deduce
\begin{align*}
 \langle \mathcal R^h(\partial_x \wedge \partial_y),\partial_t \wedge \partial_x\rangle = \langle \mathcal R^h(\partial_x \wedge \partial_y),\partial_t \wedge \partial_y\rangle = 0.
\end{align*}
 From \eqref{Rhmatrix} also 
\begin{align*}
 \langle \mathcal R^h(\partial_t \wedge \partial_y),\partial_t \wedge \partial_x\rangle = 0.
\end{align*}
So in this basis $\mathcal R^h_p$ is given as
\begin{align*}
 \left( \begin{array}{ccc} \nu & 0 & 0 \\ 0 & \nu & 0 \\ 0 & 0 & \sec_p^h(\partial_x,\partial_y) \end{array} \right).
\end{align*}
Finally,
$$\sec^h_p(\partial_x,\partial_y) = h_p(R^h_{\partial_x \partial_y}\partial_y,\partial_x) = \int_{\crcl}g_p(R_{d\theta \partial_x,d\theta \partial_y}d\theta \partial_y,d\theta \partial_x)d\theta \geq c.$$
The proposition now follows from the gluing lemma \ref{gluing} and Lemma
 \ref{polar}.
\end{proof}
\begin{remark}
Let $\gamma : \RR \to N$ be an arbitrary geodesic of a Riemannian $3$-manifold $(N,g)$ and $[a,b] \subset \RR$ such that $\gamma_{[a,b]}$ is injective. Choose $\delta > 0$ such that the normal exponential map $\exp : N^{< \delta}{\gamma_{[a,b]}} \to M$ is a diffeomorphism onto its image $M \subset N$. Then Lemma \ref{symmetry along geodesic} holds analogously for $(M,g)$ and the geodesic $\gamma : [a,b] \to M$ (without the condition that $\gamma$ is closed). This can either be seen with the same proof, or alternatively by constructing a smooth Riemannian $3$-manifold $(\hat M, \hat g)$ which contains a closed geodesic $\hat \gamma$ in a way that there exists an isometric embedding $(M,g) \to (\hat M, \hat g)$ that injects $\gamma$ into $\hat \gamma$. Then the statement follows by applying Lemma \ref{symmetry along geodesic} to $(\hat M, \hat g)$ and $\hat \gamma$.
\end{remark}

Now we are able to give the main technical tool for our arguments:
\begin{proposition}\label{symmetry}
 Let $(M,g)$ be a closed, simply connected Riemannian $4$-manifold admitting an isometric $\crcl$-action with only isolated fixed points and quotient space $(M^*,d)$. Let $F$ denote the set of fixed points, $E$ denote the set of exceptional orbits and let $c \in \RR$ with $\sec_M \geq c$ (or more generally $\curv M^* \geq c$). Then there exists a sequence of Alexandrov metrics $d_n$ on $M^*$ satisfying the following properties:
 \begin{enumerate}
  \item $d_{GH}((M^*,d_n),(M^*,d)) < \frac {1}{n}$.\label{symmetry 1}
  \item $d_n$ is smooth on $M\setminus (F \cup E^*)$ and coincides with $d$ on $M^*\setminus B_{1/n}^d(F^* \cup E^*)$.\label{symmetry 2}
  \item $\curv (M^*,d_n) > c - \frac {1}{n}$.\label{symmetry 3}
  \item There exists $\rho > 0$ such that for every fixed point $p \in M^*$ the open ball $B_{\rho}(p)$ is isometric to the corresponding ball at a tip of $S(\Sigma_pM^*)$, the spherical suspension of $\Sigma_pM^* \cong \sphere^3/\crcl$, where $\crcl$ acts fixed point free on the round $\sphere^3$.\label{symmetry 4}
  \item For an arc $\gamma$ of $E^*$ with isotropy group $\ZZ_m$ there exists an open neighborhood $U$ of $\gamma$ which is a good Riemannian orbifold $\hat  U \to U = \hat U /\ZZ_m$ with $\gamma = \operatorname{Fix}(\ZZ_m)$. Moreover, the $\ZZ_m$-action on $\hat U$ is induced by a polar $\crcl$-action.\label{symmetry 5}
  \item An open neighborhood of the nonregular part of $M^*$ admits an isometric $\crcl$-action that fixes the nonregular part. This action is polar in a neighborhood of every regular point.\label{symmetry 6}
 \end{enumerate}
\end{proposition}
\begin{proof}
 Using Corollary \ref{constant curvature at fixed point}, we obtain a sequence of $\crcl$-invariant Riemannian metrics $\tilde g_n$ on $M$ with $\sec_{\tilde g_n} > c -1/2n$ such that $\tilde g_n$ coincides with $g$ outside an $1/n$-ball of the fixed point set and $\tilde g_n$ has constant curvature $1$ in some neighborhood of every fixed point. A neighborhood of a fixed point is therefore locally isometric to $\sphere^4$. Let $\tilde d_n$ denote the quotient metric of $\tilde g_n$. Since $\sphere^4$ is isometric to the spherical suspension of $\sphere^3$, it follows that a neighborhood $B_{\rho}(p)$ of a fixed point $p \in (M^*,\tilde d_n)$ is isometric to the corresponding ball at a tip of the spherical suspension of $\Sigma_p(M^*,\tilde d_n) = \sphere^3/\crcl_p$, where $\crcl_p = \crcl$ is acting orthogonally and fixed point free. Clearly $(M^*,\tilde d_n)$ converges in Gromov-Hausdorff sense to $(M^*,d)$ for $n \to \infty$. We thus constructed a sequence of metrics satisfying  properties \textit{1 -- 4}. Recall from Proposition \ref{directions} that for a fixed point $p \in M^*$ its space of directions $\Sigma_pM^*$ admits an isometric $\crcl$-action that fixes its singularities. Suspending this action we can extend it to an isometric $\crcl$-action on $B_{\rho}(p)$ that fixes all singularities of $B_{\rho}(p)$. This action is easily seen to be polar restricted to the regular part.
 
 Now let $\overline{\gamma} : [0,1] \to M^*$ be a nonregular arc in $M^*$ with constant isotropy group $\ZZ_k$ along its interior such that $p = \overline \gamma(0)$ and $q = \overline \gamma(1)$ are two different fixed points in $M^*$. Choose $s > 0$ such that $\gamma(s) \in B_{\rho}(p)$ and $\gamma(1 - s) \in B_{\rho}(q)$. Define the normal bundle 
 $$N^{< \delta}\gamma_{[s,1-s]} := \{v \in \Sigma_{\gamma(t)}M^* \mid s \leq t \leq 1 -s,\ \measuredangle (v,\dot \gamma(t)) = \pi/2\}.$$ 
 It follows from the slice theorem (or alternatively Theorem \ref{lytchak-thorbergson}) that for all sufficiently small $\delta > 0$ the exponential map $\exp : N^{< \delta}\gamma_{[s,1-s]} \to M^*$ is injective (in fact it can be shown that $\exp : N^{< \delta}\gamma_{]0,l[} \to M^*$ is injective for some $\delta > 0$, compare Lemma \ref{injective normal}). Let $U_\delta := \exp (N^{< \delta}\gamma_{[s,1-s]})$. By Theorem \ref{lytchak-thorbergson}, $(U_{\delta},\tilde d_n)$ is a good Riemannian orbifold;
$$(\hat U_\delta, \tilde h_n) \to (\hat U_\delta/\ZZ_k, \tilde h^*_n) = (U_{\delta},\tilde d_n).$$ The fixed point set of the $\ZZ_k$-action on $\hat U_\delta$ can be parametrized by a geodesic $\hat \gamma$ that projects to $\gamma$. Now we apply Lemma \ref{symmetry along geodesic} and its following remark to $(\hat U_\delta, \tilde h_n)$ and $\hat \gamma$ to obtain a smooth metric $h_n$ on $\hat U_\delta$ that coincides with $\tilde h_n$ up to an arbitrary small open neighborhood $W$ of $\hat \gamma$, admits a polar $\crcl$-action on $W$ that fixes $\hat \gamma$ and has curvature bounded below by $c - 1/n$. Let $d_n$ denote the metric on $U_\delta$ induced by $h_n$. From the construction of $h_n$ (see  \eqref{definition of h}) it is clear, possibly after change of orientation of the action, that the $\ZZ_k$-action on $\hat U_\delta$ is induced by the action of $\crcl$. Moreover $h_n$ and $\tilde h_n$ coincide at points that project to $B_{\rho}(p)$ or $B_{\rho}(q)$ since $\tilde h_n$ is already invariant under rotations fixing $\hat \gamma$ at these points. Therefore, the $\crcl$-action on $W$ descends to an isometric $\crcl$-action on an open neighborhood of $\gamma_{[s,1 - s]}$ in $(U_\delta,d_n)$ and, again up to orientation, both $\crcl$-actions on $B_{\rho}(p) \cap U_\delta$ and $B_{\rho}(q) \cap U_\delta$ coincide as well as the metrics $\tilde d_n$ and $d_n$. Choosing $W$ small enough it follows that $\tilde d_n$ and $d_n$ coincide along an open neighborhood of the boundary of $U_\delta$. Thus we can extend $d_n$ via $d$ to all of $M^*$ to a metric called $d_n$ as well. We can perform this construction along every nonregular edge of $E^*$ and obtain properties \textit{1 -- 5} of the proposition.
 
 To obtain an $\crcl$-action in a neighborhood of the nonregular part of $M^*$, we need to verify that all these actions defined along nonregular edges can be oriented in a consistent way. This is possible since $M^*$ is simply connected and hence orientable.
\end{proof}


\section{Resolution of the singularities}\label{resolution}
In this section we prove the following theorem.
\begin{theorem}\label{thm2.3}
Let $M$ be a closed, simply connected Riemannian $4$-manifold admitting an isometric $\crcl$-action with isolated fixed points only and quotient space $(M^*,d)$. Let $L < \infty$ denote the maximal order of a finite isotropy group of the action and assume that $\curv M^* \geq c$.
Then there exists a sequence of smooth Riemannian metrics $g_n$ on $M^*$ with $\sec_{g_n} \geq Lc - 1/n$ if $c \leq 0$ ($\sec_{g_n} > 0$ if $c > 0$) such that $(M^*,d_n) \to (M^*,d)$ in Gromov-Hausdorff sense, where $d_n$ denotes the metric induced by $g_n$.
\end{theorem}
This theorem clearly implies Theorem \ref{umformulierung}. Given $M$ with the conditions of Theorem \ref{thm2.3} it follows from Proposition \ref{symmetry} that there exists a sequence of metrics $d_n$ on $M^*$ satisfying the properties \textit{1 -- 6} of Proposition \ref{symmetry}. From the properties \textit{1} and \textit{3} it follows that in order to prove Theorem \ref{thm2.3} it is enough to prove the following proposition.
\begin{proposition}\label{30}
Let $M$ be a closed, simply connected Riemannian $4$-manifold admitting an isometric $\crcl$-action with isolated fixed points only, and quotient space $M^*$. Denote by $L < \infty$ the maximal order of a finite isotropy group of the action. Let $d$ be a metric on $M^*$ (possibly different from the quotient metric) which is induced by a smooth Riemannian metric $g$ on $M \setminus (F \cup E^*)$ with $\curv (M^*,d) \geq c$ and satisfies properties \textit{4 -- 6} of Proposition \ref{symmetry}. Then there exists a sequence of smooth Riemannian metrics $g_n$ on $M^*$ satisfying $\sec_{g_n} \geq Lc - 1/n$ if $c \leq 0$ ($\sec_{g_n} > 0$ if $c > 0$) such that $(M^*,d_n) \to (M^*,d)$ in Gromov-Hausdorff sense for the induced metrics $d_n$.
\end{proposition}

Therefore, for the rest of section $2.3$ let $M$ be a closed, simply connected, $4$-dimensional Riemannian manifold admitting an isometric $\crcl$-action with isolated fixed points only. We fix a metric $d$ on $M^*$ (possibly different from the quotient metric) which is induced by a smooth Riemannian metric $g$ on $M^* \setminus (F \cup E^*)$ and has curvature bounded below by $c$, for some $c \in \RR$. We assume that the properties \textit{4}, \textit{5} and \textit{6} of Proposition \ref{symmetry} are satisfied for $d$ (and refer to them by these numbers). We also fix $\rho > 0$ according to property \textit{4}.
\\ \\
The strategy to prove Proposition \ref{30} is the following: For simplicity let us assume that $c > 0$, so $M^*$ has positive curvature (the case $c \leq 0$ is very similar). Let $p$ and $q$ be two fixed points which are joined by a singular arc $\gamma$ of constant type $\ZZ_m$.  We cover $\gamma$ by coordinate neighborhoods $B_\rho(p)$, $B_\rho(q)$ and $U$, where $U$ is a tubular neighborhood of the interior points of $\gamma$ and a good orbifold. Also we choose $U$ small, such that we have a polar $\crcl$-action on $B_\rho(p) \cup B_\rho(q) \cup U$. Let $\Sigma$ denote a section of the action on the regular part. Then $g$ is determined by its restriction to the tangent bundle of $\Sigma$  and the function $\varphi = ||X||_g$ restricted to $\Sigma$, where $X$ denotes the Killing field of the action. On the other hand, given any positive smooth function $\psi : \Sigma \to \RR$, there exists a unique, smooth metric $g_\psi$ defined on the regular part of $B_\rho(p) \cup B_\rho(q) \cup U$ that admits a polar $\crcl$-action with section $(\Sigma,g)$ and whose Killing field at $s \in \Sigma$ has length $\psi(s)$. From Corollary \ref{curvature and Hessian} we observe that $g_{\psi}$ has positive curvature if and only if $\psi : \Sigma \to \RR$ is strictly concave.

Having this in mind we carefully deform the function $\varphi$ to obtain a smooth metric at interior points of $\gamma$. Since $\gamma$ has constant type at its interior points, it follows that the non-smoothness of $g$ at interior points of $\gamma$ is due only to the radial derivatives of $\varphi$ at $\gamma$ (the restriction of $g$ to the closure of $\Sigma$ is essentially smooth). Therefore, adding a small, nonnegative, smooth function $h : \Sigma \to \RR$ to $\varphi$, which satisfies some regularity assumptions near $\gamma$, we can extend $g_{\varphi + h}$ smoothly to interior points of $\gamma$. Geometrically speaking, we make the Killing field $X$ a bit longer near $\gamma$ so that the conical singularity at interior points of $\gamma$ is resolved. Moreover, if $\varphi + h : \Sigma \to \RR$ is strictly concave, we see that $g_{\varphi + h}$ has positive curvature.

The first aim is to resolve the singularities of $E^*$ via this above approach: In section \ref{resolving brho} we resolve the singularities of $\gamma$ lying inside $B_\rho(p)$ and $B_\rho(q)$. This is not hard, using that $B_\rho(p)$ and $B_\rho(q)$ are isometric to the corresponding balls at the tips of the spherical suspensions of their respective spaces of directions. We construct these resolutions in a way that they are in fact induced by families of functions $h_n : B_\rho(p) \cap \Sigma \to \RR$ and $\tilde h_n : B_\rho(q) \cap \Sigma \to \RR$ in the above sense such that $\varphi + h_n$ and $\varphi + \tilde h_n$ are strictly concave functions on $B_\rho(p)$ and $B_\rho(q)$, respectively. 

Then, in section \ref{resolving U}, we construct a resolution of $U$ via functions $\overline h_n : \Sigma \to \RR$ such that $\varphi + \overline h_n : \Sigma \to \RR$ is strictly concave at points of sufficiently large distance (say bigger than $\rho/4$) to the fixed points $p$ and $q$. The function $\overline h$ is constructed in a way that it satisfies the boundary value problem imposed by $h$ and $\tilde h$ on $\Sigma \cap V(p)$ and $\Sigma \cap V(q)$,  respectively, where $V(p)$ and $V(q)$ are appropriate neighborhoods of $p$ and $q$.

Therefore, putting the functions $\varphi + h_n$, $\varphi + \overline h_n$ and $\varphi + \tilde h_n$ together we obtain a continuous, but  non-smooth, function $\psi_n : \Sigma \to \RR$.  Moreover, it follows from our construction that $\psi_n$ is strictly convex (in fact, $\psi_n$ is given near the boundary of $V(p)$ by the minimum of the strictly convex functions $\varphi + \overline h$ and $\varphi + h$, and analogously for $V(q)$). Using a smoothing technique of Greene-Wu \cite{greene-wu76}, in section \ref{resolving E} we smooth the functions $\psi_n$. Via the resulting family of functions we obtain a resolution of the singularities of $g$ at interior points of $\gamma$ with controlled lower curvature bounds. This resolution can be performed independently along every singular arc of $E^*$ and leads to a resolution of $E^*$.

The resulting metrics are smooth at all points not lying in $F$. Moreover, in a small neighborhood of a fixed point $p$ each of this metrics is still locally isometric to the spherical suspension of $\Sigma_pM^*$ (with respect to the new metric), which is smooth and has lower curvature bound $1$. In section \ref{smooth fp} we show that such a neighborhood of $p$ can be isometrically embedded into $\sphere^4$. Using a Beltrami map we transfer the problem from $\sphere^4$ to $\RR^4$, where we can smooth the singularity at $p$ using a convolution. Finally we glue back the smoothed metric in a neighborhood of $p$, while almost preserving  lower curvature bounds. Then Proposition \ref{30} follows.
\\ \\
Let us assume that $E^*$ is nonempty (otherwise we can proceed with section \ref{smooth fp}). Then there exist two fixed points $p, q \in F$, which are joined by an arc in $E^*$. Let this arc be parametrized by a geodesic $\gamma : ]0,l[ \to M^*$ with $\vert \dot \gamma(t) \vert = 1$, $\lim_{t \to 0}\gamma(t) = p$ and $\lim_{t \to l}\gamma(t) = q$. The isotropy group at $\gamma(t)$ is independent of $t$  and isomorphic to $\ZZ_m$ for some $m \geq 2$. Let this $p,q,\gamma$ and $m$ be fixed from now on.
\subsection{\texorpdfstring{Resolution of $B_{\rho}(p)$ and $B_\rho(q)$}{Resolution of B(p) and B(q)}}\label{resolving brho}
In this section we construct resolutions of the singularities of $\gamma$ lying inside $B_\rho(p)$ and $B_\rho(q)$. This is achieved by first resolving the singularities of $\Sigma_pM^*$ and $\Sigma_qM^*$ corresponding to the respective directions of $\gamma$. Since $B_\rho(p)$ and $B_\rho(q)$ are modeled on the spherical suspensions of $\Sigma_pM^*$ and $\Sigma_qM^*$, we obtain the desired resolutions by suspending the resolutions of $\Sigma_pM^*$ and $\Sigma_qM^*$. Then at the end of this section we show that these resolutions are induced by a perturbation of the Killing field $X$ of the polar $\crcl$-action on $B_\rho(p) \cup B_\rho(q)$, which can be described by functions $h$ and $\tilde h$ defined on a section of $B_\rho(p)$ and $B_\rho(q)$, respectively. 
\\ \\
Let 
$$(\theta,\alpha)$$ be coordinates of $\Sigma_pM^* = \sphere^3/\crcl \cong \sphere^2 \cong D_{\pi/2}/\partial D_{\pi/2}$ as in section \ref{sectionS^3/crcl} ($\theta$ denotes radial direction on $D_{\pi/2}$ and $\alpha$ denotes angular direction) such that $\dot \gamma(0)$ corresponds to $\theta = 0$. Then the metric $d\sigma^2$ of $\Sigma_pM^*$ is given by
\begin{align}\label{smuth1}
d\sigma^2 = d\theta^2 + R^2(\theta)d\alpha^2,
\end{align}
for a smooth function $R : [0,\pi/2] \to \RR$ as described in Proposition \ref{directions}. By property \textit{5} of our assumptions the metric $g$ restricted to $B_{\rho}(p)$ is given by
\begin{align}\label{smuth69}
g = dr^2 + \sin^2(r)d\sigma^2,
\end{align}
for $0 \leq r \leq \rho$. We denote by
\begin{align*}
 (r,\theta,\alpha)
\end{align*}
the corresponding coordinates of $B_{\rho}(p)$. 
From \eqref{smuth1} and \eqref{smuth69} it is clear that the singularity of $g$ at the points of $\gamma$ that lie in $B_{\rho}(p)$ arises from the singularity of $\Sigma_pM^*$ at $\dot \gamma(0)$, i.e. the behavior of $R$ at $0$. In fact  we have (compare section 1.4 of \cite{petersen98}):
\begin{lemma}\label{differenti}
 Let $\omega : [0, \pi/2] \to \RR$ be a smooth function with $\omega(0) = 0$. Then
 $$d\theta^2 + \omega^2(\theta)d\alpha^2$$
 defines a smooth metric in a neighborhood of $\dot \gamma(0)$ if and only if $\omega'(0) = 1$ and all even derivatives of $\omega$ vanish at $0$.
\end{lemma}
Recall from Proposition \ref{directions} that $R(0) = 0$, $R'(0) = m^{-1}$ and $R^{(even)}(0) = 0$. Therefore, for every smooth function $\eta$ with $\eta(0) = 0$, $\eta'(0) = 1 - m^{-1}$ and $\eta^{(even)}(0) = 0$ the metric 
\begin{align}\label{mann mann}
d\theta^2 + (R + \eta)^2(\theta)d\alpha^2
\end{align}
is smooth in a neighborhood of $\dot \gamma(0).$ In the following Lemma we construct a family of such functions satisfying particular properties in order to be able to control additional data of a metric as in \eqref{mann mann}, for example its lower curvature bound.
\begin{lemma}\label{eta}
For all $0 < \tau < \pi/4$ and $\delta > 0$ there exists $0 < \tau(\delta)$ and a smooth function $\eta_{\tau,\delta} : [0,\pi/2] \rightarrow \RR$ satisfying the following properties:
\begin{itemize}
\item[(i)] $\eta_{\tau,\delta}(\theta) = (1 - m^{-1})\sin(\theta) \text{ for } 0 \leq \theta \leq \tau(\delta)$, \label{eta1}
\item[(ii)] $\eta_{\tau,\delta}'(\theta) > 0 \text{ for } 0 < \theta < \tau/2$, \label{eta2}
\item[(iii)] $\eta_{\tau,\delta}'(\theta) \leq 0 \text{ for } \theta \geq \tau/2$,\label{eta3}
\item[(iv)] $\eta_{\tau,\delta}(\theta) = 0 \text{ for } \theta \geq \tau$, \label{eta4}
\item[(v)] $\eta_{\tau,\delta}''(\theta)/\eta(\theta) \leq -1 \text{ for } 0 \leq \theta \leq \tau/2,$\label{eta5}
\item[(vi)] $|\eta_{\tau,\delta}(\theta)| + |\eta_{\tau,\delta}'(\theta)| + |\eta_{\tau,\delta}''(\theta)| < \delta \text{ for } \tau/3 \leq \theta \leq \pi/2.$\label{eta6}
\end{itemize} 
\end{lemma}
 Note that $0 \leq \eta_{\tau,\delta}(\theta) \leq \delta$ for all $\theta \in [0,\pi/2]$.
\begin{proof}
Let $0 < \tau < \pi/4$ and $\delta > 0$ be given. The idea for the construction of $\eta = \eta_{\tau,\delta}$ is illustrated in Figure \ref{figure:eta}. A more precise definition of $H$ follows.
\begin{figure}[h!]
\centering
\def\svgwidth{0.5\textwidth}
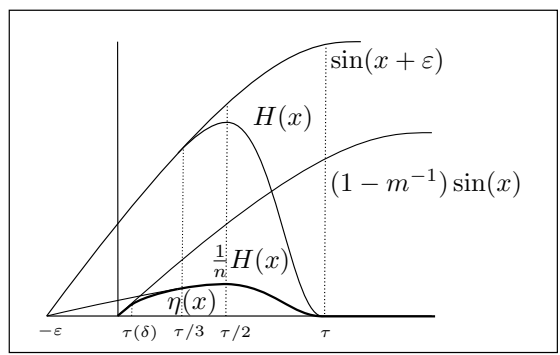
\caption{\small{$\eta$ is constructed via smoothing the function $\min\{(1 - m^{-1})\sin,n^{-1}H\}$.}}
\label{figure:eta}
\end{figure}

Fix $0 < \epsilon < \pi/8$. There exists a smooth function $h : [-\epsilon,\pi/2] \to \RR$ satisfying 
\begin{align*}
 &h(x) = \cos(x + \epsilon)\text{ for } x \in [-\epsilon,\tau/3],\\
 &h(x) > 0 \text{ and } h'(x) \leq -\sin(x + \epsilon) \text{ for } x \in [0,\tau/2[,\\
 &h(\tau/2) = 0,\\
 &h(x) < 0 \text{ for } x \in \ ]\tau/2,\tau[,\\
 &h(x) = 0 \text{ for } x \in [\tau,\pi/2],\\
 &\int_{-\epsilon}^\tau h(x)dx = 0.
 \end{align*}
 Such a function $h$ is easy to construct, compare Figure \ref{figure h}, and we define $$H(\theta) := \int_{-\epsilon}^{\theta}h(t)dt.$$
\begin{figure}
\centering
\def\svgwidth{0.5\textwidth}
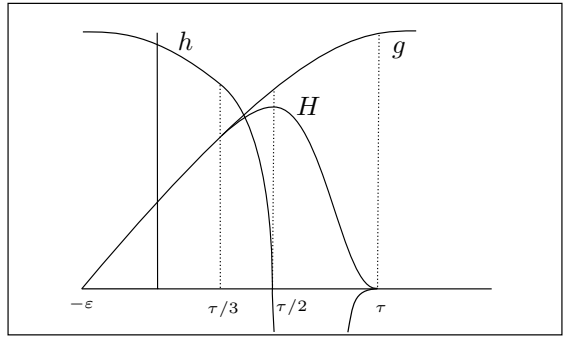
\caption{\small{$H$ is constructed by defining its derivative $h$ first, and then integrating it. $\int_{\tau/2}^{\tau} h(x)dx$ can be chosen as an arbitrary negative number. Therefore, it is easy to satisfy the condition $\int_{-\epsilon}^\tau h(x)dx = 0$.}}
\label{figure h}
\end{figure}

Then the properties $(ii), (iii)$ and $(iv)$ hold for the function $H$. Further, by construction, $0 \leq H(x) \leq \sin(x + \epsilon)$ and $H(x)'' = h'(x) \leq -\sin(x + \epsilon) < 0$ on $[0,\tau/2]$. Consequently
$$H''(x)/H(x) \leq H''(x)/\sin(x + \epsilon) \leq -1$$
for all $x \in [0,\tau/2].$
Thus $(v)$ holds for $H$ as well.
Set $H_n := n^{-1}H$ for $n \geq 1$. Then $H_n$ satisfies the conditions $(ii) - (vi)$ for all sufficiently large $n$. Now define 
\begin{align*}
\eta_n : [0,\pi/2] &\to \RR\\
 \theta &\mapsto \min\{(1 - m^{-1})\sin(x),H_n(x)\}.
\end{align*}
For all $n \in \NN$ sufficiently large the graphs of $\eta_n$ and $(1 - m^{-1})\sin$ intersect exactly once at some $s \in ]0,\tau/3[$. Fix such an $n$. Then $\eta_n$ satisfies the conditions $(i)$ to $(vi)$ at all points but $s$, where it fails to be differentiable. To obtain the desired function $\eta$ it remains to smooth $\eta_n$ at $s$ while keeping the conditions $(i)$ to $(vi)$ satisfied, compare Figure \ref{figure:eta}. For example, this can be done as follows: Let 
\begin{align*}
\nu_n : [0,\pi/2] &\to \RR\\
x &\mapsto
\begin{cases} 
(1 - m^{-1})\cos(x), &\text{for } 0 \leq x < s,\\
H_n'(x), &\text{for } s \leq x \leq \pi/2.
\end{cases}
\end{align*}
Then $\eta_n(t) = \int_0^t \nu_n(x) dx$. Let $\mu > 0$ be small and $\nu : [0,\pi/2] \to \RR$ a smooth function satisfying the following conditions:
\begin{align*}
\nu(x) = \nu_n(x) &\text{ for } x \in [0,\pi/2] \setminus [s-\mu,s + \mu],\\
\nu'(x) \leq \nu_n'(x) &\text{ for } x \in [0,\pi/2] \setminus \{s\},\\
\int_0^{\pi/2} \nu(x)dx &= \int_0^{\pi/2}\nu_n(x)dx.
\end{align*}
The existence of such a function is clear since $(1 - m^{-1})\cos(s) > H_n'(s)$, compare Figure \ref{figure nu}.
\begin{figure}
\centering	
\def\svgwidth{0.5\textwidth}
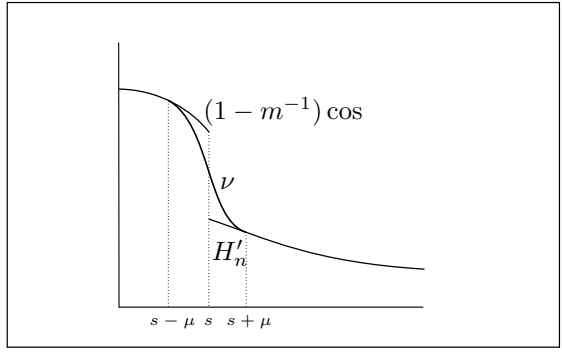
\caption{\small{The function $\nu$}}
\label{figure nu}
\end{figure}
Set 
$$\eta(x) := \int_0^x \nu(t)dt.$$
It is easily checked that $\eta$ is smooth and satisfies conditions $(i)$ - $(vi)$ for $\mu > 0$ sufficiently small.
\end{proof}
\begin{definition}
 For all $0 < \tau < \pi/4$ and $\delta > 0$ we fix a function $$\eta_{\tau,\delta} : [0,\pi/2] \to [0,\delta]$$
satisfying properties \textit{(i)} to \textit{(vi)} of Lemma \ref{eta}.
\end{definition}
 We can use this family of functions $\{\eta_{\tau,\delta}\}$ to resolve the singularity of $\Sigma_pM^*$ at $\dot \gamma(0)$ while keeping the curvature bounded below by $1$: 
\begin{proposition}\label{resolution of dotgamma}
For all $0 < \tau < \pi/4$ and $0 < \delta$ the metric
\begin{align}
d \sigma^2_{\tau,\delta} := d\theta^2 + (R + \eta_{\tau,\delta})^2(\theta)d\alpha^2\label{smoothSoD}
\end{align}
is smooth on $\Sigma_pM^* \setminus \{\theta = \pi/2\}$. There exists a constant $a > 0$ such that for all $0 < \tau < \pi/4$ and for all sufficiently small $\delta > 0$ (depending on $\tau$) we have
$$\curv (\Sigma_pM^*,d \sigma^2_{\tau,\delta}) \geq 1 + a.$$
\end{proposition}
\begin{proof}
That $d\sigma^2_{\tau,\delta}$ is smooth on $\Sigma_p \setminus \{\theta = \pi/2\}$ for all $0 < \tau < \pi/4$ and $\delta > 0$ follows from Proposition \ref{directions} and Lemma \ref{differenti}. Fix $0 < \tau < \pi/4$ and let $0 < \delta < 1$ be arbitrary. To simplify notation, let $\eta = \eta_{\tau,\delta}$ and $\tilde R = (R + \eta)$. The sectional curvature at $(\theta,\alpha)$ for $\theta \in \ ]0,\pi/2[$ is given by (see Lemma \ref{curvature})
$$\sec (\theta, \alpha) = -\tilde R''(\theta)/\tilde R(\theta) = (R'' + \eta'')(\theta)/(R + \eta)(\theta).$$
By Proposition \ref{directions} there exists a constant $\tilde a > 0$ with $-R''/R > 1 + \tilde a$, or equivalently
$$-R'' > R + \tilde a R.$$
From the construction of $\eta$ it follows that 
$$-\eta''(\theta) \geq \eta(\theta)$$
for all $\theta \in [0,\tau/2]$.

We claim that there exists a constant $0 < b < 1$, independent of $\tau$ and $\delta$, such that $$b \eta \leq R$$ for all $\tau,\delta$: In fact observe that $\eta$ attains its maximum at $\theta = \tau/2$, is concave restricted to $[0,\tau/2]$ and vanishes on $[\tau,\pi/2]$. Since $\eta'(0) = (1 - m^{-1})$ and $\tau \leq \pi/4$, it follows that 
$$\eta(\theta) \leq \mu(\theta) = \begin{cases}
(1 - m^{-1})\theta, &\text{ for } 0 \leq \theta \leq \pi/4,\\
0, &\text{ for } \pi/4 \leq \theta \leq \pi/2.
\end{cases}$$
Since $R$ is positive on $]0,\pi/2[$ and has positive derivative at $0$ it is clear that there exists some constant $0 < b < 1$ such that $b\mu \leq R$ and the claim follows.

Consequently for all $\theta \in ]0,\tau/2]$
\begin{align*}
 -(R'' + \eta'') &> R + \eta + \tilde a R\\
 &\geq R + \eta + \frac{\tilde a}{2} R + \frac{b\tilde a}{2}\eta \geq R + \eta + \frac{b\tilde a}{2}(R + \eta),
\end{align*}
so
\begin{align}\label{dima}
-\tilde R''(\theta)/\tilde R (\theta) > 1 + \frac{b\tilde a}{2}
\end{align}
for all $\theta \in \ ]0,\tau/2]$ (independently of $\delta > 0$). On the other hand on $[\tau/2,\pi/2]$ the function $-(R'' + \eta'')/(R + \eta)$ converges uniformly to $-R''/R$ for $\delta \to 0$ by property \textit{(vi)} of $\eta$. Since $-R''/R > 1 + \tilde a$, we find that \eqref{dima} holds for all sufficiently small $\delta > 0$, for all $\theta \in ]0,\pi/2[$.
\end{proof}
Via the spherical suspension of the metric \eqref{smoothSoD} we can define a singular Riemannian metric on $B_\rho(p)$ as
\begin{align}\label{smoothDoS}
g_{\tau,\delta} := dr^2 + \sin^2(r)d \sigma^2_{\tau,\delta}.
\end{align}
The induced metric on $B_{\rho}(p)$ is denoted by $d_{\tau,\delta}$.

\begin{definition}\label{suspended}
For a subset $W \subset \Sigma_pM^*$ we denote by $s_p(W) \subset B_\rho(p)$ the subset obtained by suspending $W$. In more detail, given  $W \subset \Sigma_pM^*$ consider $S(W) = (W \times [0,\pi])/\sim$ $\subseteq (\Sigma_pM^* \times [0,\pi])/\sim$. Then $s_p(W) := S(W) \cap B_{\rho}(p)$, considering $B_\rho(p)$ as a subset of $S(\Sigma_pM^*)$. Analogously we define $s_q(W) \subset B_\rho(q)$.
\end{definition}
Now we can state the main result of this section:
\begin{proposition}\label{resolution of Brhop}
Let $0 < r < \pi/4$ and $W = B_r(\dot \gamma(0)) \subset \Sigma_pM^*$. Then for all $0 < \tau < r$ the metric $g_{\tau,\delta}$ coincides with $g$ on $B_{\rho}(p) \setminus s_p(W)$ and is smooth on $s_p(W) \setminus \{p\}$. Further
$$(B_{\rho}(p),d_{\tau,\delta}) \xrightarrow{\tau \to 0} (B_{\rho}(p),d)$$
in Gromov-Hausdorff sense (independently of $\delta$). Moreover, for fixed $0 < \tau < \pi/4$ and for all sufficiently small $\delta > 0$ we have $\curv d_{\tau,\delta} \geq 1$.
\end{proposition}
\begin{proof}
 This follows from Proposition \ref{resolution of dotgamma} together with Lemma  \ref{cone and suspension}, and noting that $\eta_{\tau,\delta} \to 0$ uniformly for $\tau \to 0$, independently of $\delta$.
\end{proof}
The same way we obtain a resolution of the singularities along $\gamma$ lying inside $B_\rho(q)$. The arguments are completely analogous to the ones for $B_\rho(p)$, so we just state the needed terminology and the result. Let $\Sigma_qM^* = D_{\pi/2}/\partial D_{\pi/2}$ be equipped with coordinates
$$(\tilde \theta, \tilde \alpha),$$
where $\tilde \theta$ denotes radial and $\tilde \alpha$ angular direction. In this coordinates the metric of $\Sigma_qM^*$ is given by 
$$d\tilde \sigma^2 = d\tilde \theta^2 + \tilde R^2(\tilde \theta) d\tilde \alpha^2$$
(see Proposition \ref{directions}). We assume that $\frac d {ds}_{|s = 0}\gamma(l - s)$ corresponds to $\tilde \theta = 0$. Then $\tilde R'(0) = m^{-1}$, since the isotropy group is constant along $\gamma$. Also $g$ restricted to $B_\rho(q)$ is given by
$$g = d\tilde r^2 + \sin^2(\tilde r)d\tilde \sigma^2$$
for $0 < \tilde r < \rho$, with induced coordinates
$$(\tilde r,\tilde \theta, \tilde \alpha).$$
Set 
$$d\tilde \sigma^2_{\tau,\delta} := d\tilde \theta^2 + (\tilde R + \eta_{\tau,\delta})^2(\tilde \theta)d\tilde \alpha^2$$
and
$$\tilde g_{\tau,\delta} := d\tilde r^2 + \sin^2(\tilde r)d\tilde \sigma_{\tau,\delta}^2,$$
with induced metric $\tilde d_{\tau,\delta}$ on $B_\rho(q)$.
\begin{proposition}\label{resolution of Brhoq}
Let $0 < r < \pi/4$ and $W = B_r(\frac d {ds}_{|s = 0}\gamma(l - s)) \subset \Sigma_qM^*$. Then for all $0 < \tau < r$ the metric $\tilde g_{\tau,\delta}$ coincides with $g$ on $B_{\rho}(q) \setminus s_q(W)$ and is smooth on $s_q(W) \setminus \{q\}$. Further
$$(B_{\rho}(q),\tilde d_{\tau,\delta}) \xrightarrow{\tau \to 0} (B_{\rho}(q),d)$$
in Gromov-Hausdorff sense (independently of $\delta$). Moreover, for fixed $0 < \tau < \pi/4$ and for all sufficiently small $\delta > 0$ we have $\curv \tilde d_{\tau,\delta} \geq 1$.
\end{proposition}
Now we show that these resolutions are induced by perturbations of the Killing field $X$ of the $\crcl$-action on $B_\rho(p) \cup B_\rho(q)$. 

It follows from \eqref{smoothSoD} that the natural $\crcl$-action on $\Sigma_pM^*$ is isometric with respect to the metrics $d\sigma^2_{\tau,\delta}$ as well as $d\sigma^2$. The Killing field of this action is given by the coordinate field $\partial_\alpha$ and therefore has length $R(\theta)$ and $R(\theta) + \eta_{\tau,\delta}(\theta)$ at $(\theta,\alpha)$ measured with respect to $d\sigma^2$ and $d\sigma^2_{\tau,\delta}$, respectively. 
 Set
\begin{align*}
 h_{\tau,\delta} : B_{\rho}(p) &\to \RR\\
 (r,\theta,\alpha) &\mapsto \sin(r)\eta_{\tau,\delta}(\theta).\nonumber
\end{align*}
Hence, by \eqref{smoothDoS}, the Killing field $X$ of the $\crcl$-action on $B_\rho(p)$  in this metrics has lengths
$$||X||_g(r,\theta,\alpha) = \sin(r)R(\theta)$$
and 
\begin{align}\label{killingtaudelta}
||X||_{g_{\tau,\delta}}(r,\theta,\alpha) = \sin(r)(R(\theta) + \eta_{\tau,\delta}(\theta)) = \sin(r)R(\theta) + h_{\tau,\delta}(r,\theta,\alpha). 
\end{align}

Moreover, the action of $\crcl$ on $B_\rho(p)$ is polar with respect to the metrics $g_{\tau,\delta}$ as well as $g$. From \eqref{smoothDoS} and \eqref{smuth69} it follows that the sections of $g_{\tau,\delta}$ and $g$ further coincide and the restriction of $g_{\tau,\delta}$ to a section equals the restriction of $g$. In fact it is easy to prove the following lemma:
\begin{lemma}\label{curvature section}
Equip $B_\rho(p)$ with any of the metrics $g_{\tau,\delta}$ or $g$. Then $B_\rho(p)/\crcl$ is isometric to the ball of radius $\rho$ at a tip of the spherical suspension of the interval $[0,\pi/2]$. In particular, $B_\rho(p)/\crcl$ has constant curvature $1$ and the same holds for a section of the action.
\end{lemma}
Setting
\begin{align*}
 \tilde h_{\tau,\delta} : B_{\rho}(q) &\to \RR\\
 (\tilde r,\tilde \theta,\tilde \alpha) &\mapsto \sin(\tilde r)\eta_{\tau,\delta}(\tilde \theta),\nonumber
\end{align*}
this observations hold analogously on $B_\rho(q)$.

Since the metrics $g_{\tau,\delta}$, $\tilde g_{\tau,\delta}$ and $g$ are completely determined by its restrictions to a section of the action and the norm of its Killing fields along a section we extend the functions $h_{\tau,\delta}$ and $\tilde h_{\tau,\delta}$ along a neighborhood of $\gamma$ in order to extend the resolutions of $B_\rho(p)$ and $B_\rho(q)$. This is the subject of the following sections.
\subsection{\texorpdfstring{Resolution of $U$}{Resolution of U}}\label{resolving U}
In this section we construct an open neighborhood $U$ of $\gamma$ which is a good Riemannian orbifold and admits a polar $\crcl$-action with section $\Sigma$. Then we construct a family of functions $\overline h_{\tau,\delta} : \Sigma \to \RR$ which is used as in the previous section to define a resolution of $g$ on $U$ via smooth metrics $\overline g_{\tau,\delta}$ on $U$ induced by a small perturbation of the Killing field $X$. We are able to control the curvature of this resolution only at a certain distance from the fixed points $p$ and $q$. But $\overline h_{\tau,\delta}$ is constructed in a way that it satisfies the boundary value problems imposed by $h$ and $\tilde h$ for appropriate neighborhoods of $p$ and $q$, respectively. This is used afterwards to construct a smooth resolution of $\gamma$ with controlled lower curvature bound that extends to $M^*$.
\\ \

We first construct a the desired neighborhood $U$ of $\gamma$ and derive basic formulas for $g$ in this coordinates. For that recall that the normal bundle $N \gamma$ and parallel translation along $\gamma$ are naturally defined along the interior of $\gamma$ via the orbifold structure near $\gamma$ (property \textit{5} of our assumptions on the metric $d$ of $M^*$). We have the following lemma:
\begin{lemma}\label{injective normal}
 There exists $T > 0$ such that the normal exponential map
 $$\exp : N^{< T} \gamma \to (M^*,d)$$
 is injective ($N^{< T}\gamma$ denotes the set of normal vectors to $\gamma$ of length less than $T$). 
\end{lemma}
\begin{proof}
Assume such a $T > 0$ does not exist. Then there exist sequences $v_n$ and $w_n$ in $N \gamma$ with $v_n \neq w_n$, $|v_n| < 1/n$, $|w_n| < 1/n$ and $\exp(v_n) = \exp(w_n)$. After taking a subsequence we assume that the sequence $\exp(v_n)$ converges to a point $x \in M^*$. From the orbifold structure of a neighborhood of $\gamma$ it follows that either $x = p$ or $x = q$. We may assume that $x = p$: Then for all $n$ sufficiently large the geodesics $t \mapsto \exp(tv_n) = c_{v_n}(t)$ and $t \mapsto \exp(tw_n) = c_{w_n}(t)$ are contained in $B_{\rho}(p)$ for $0 \leq t \leq1$, have the same value at $t = 1$ and are orthogonal to $\gamma$. Let $B_{\rho}(p)^* = B_{\rho}(p)/\crcl$. Since $c_{v_n}$ and $c_{w_n}$ are orthogonal to $\gamma$, it follows that they project to geodesics $c_{v_n}^*$ and $c_{w_n}^*$ of $B_{\rho}(p)^*$. From \eqref{smuth69} it follows that $B_{\rho}(p)^*$ is isometric to a piece of the spherical suspension of the interval $[0,\pi/2]$, that is, to a convex subset of the round $\sphere^2$. Assume $c_{v_n}^* \neq c_{w_n}^*$. Since the triangle formed by $c_{v_n}^*$, $c_{w_n}^*$ and $\gamma^*$ has two right angles at $\gamma$, but two arbitrary short side lengths $|v_n|$ and $|w_n|$, we obtain a contradiction. Thus $c_{v_n}^* = c_{w_n}^*$. Moreover there exists some $t_0 > 0$ such that $c_{v}^*(t)$ is regular for all unit normal vectors of $\gamma$ close to $p$. In particular, $c_{v_n}$ and $c_{w_n}$ are minimal geodesics between its orbits on the interval $[0,t_0]$. Since $c_{v_n}(0) = c_{w_n}(0)$ it follows that $v_n = w_n$, for large $n$, a contradiction.
\end{proof}
We fix such a $0 < T < \rho$ in a way that the $\crcl$-action (property \textit{6}) is defined on $\exp(N^{< T} \gamma).$
Observe that $\exp(N^{< T}\gamma,g)$ is a good Riemannian orbifold. Let
$$U := \exp(N^{< T}\gamma),$$
and it follows that property \textit{5} holds with respect to $U$.
Then $U  \cong \ ]0,l[ \times B^2_T(0)$ and analogous to the construction of the Fermi coordinates \eqref{fermi} we obtain coordinates
\begin{align}
 (s,t,\phi) 
\end{align}
on $U$ with $(s,t,\phi) \in \ ]0,l[ \times [0,T[ \times [0,2\pi]$ ($s$ corresponds to the parameter of $\gamma$, $t$ corresponds to radial direction with respect to $\gamma$ and $\phi$ corresponds to the action of $\crcl$). Since the action of $\crcl$ on $U$ is polar, in these coordinates $g$ is given by
\begin{align}\label{smuth3}
 g = f^2(s,t)ds^2 + dt^2 + \varphi^2(s,t)d\phi^2
\end{align}
for smooth functions $f :\ ]0,l[ \times [0,T] \to \RR$ and $\varphi :\ ]0,l[ \times [0,T] \to \RR$. Note that 
\begin{align}\label{kill}
 \varphi(s,t) = ||X(s,t,\phi)||_g,
\end{align}
for the Killing field $X$ of the $\crcl$-action on $U$. Note that, resulting from the orbifold structure of $U$,
\begin{align}\label{hat g}
\hat g = f^2(s,t)ds^2 + dt^2 + m^2\varphi^2(s,t)d\phi^2
\end{align}
defines a smooth metric on $U$. 
A section of the polar action of $\crcl$ on $U\setminus \gamma$ is parametrized by
$$\{(s,t,0) \mid 0 < s < l, 0 < t < T\}.$$
Let $\Sigma$ denote its closure in $U$. Then the metric restricted to $\Sigma$ is given by
$$g_\Sigma = f^2(s,t)ds^2 + dt^2.$$
Observe that $(\Sigma,g)$ is isometric to $(U,g)/\crcl$. Since $B_{\rho}(p)/\crcl$ as well as $B_{\rho}(q)/\crcl$ have constant curvature $1$ (Lemma \ref{curvature section}) it follows that
$$f(s,t) = \cos(t)$$
for $(s,t,\phi) \in (B_{\rho}(p) \cup B_{\rho}(q)) \cap U$.
\\ \\
Now we describe how the metrics $g_{\tau,\delta}$ and $\tilde g_{\tau,\delta}$ yielding the resolutions of $B_\rho(p)$ and $B_\rho(q)$ along $\gamma$ are given in the coordinates $(s,t,\phi)$: The $\crcl$-action on $B_\rho(p) \cap U$ is polar with respect to the metrics $g_{\tau,\delta}$ as well as $g$ and the sections are the same for these metrics and have constant curvature $1$. Further, by \eqref{killingtaudelta}, for all $(s,t,\phi) \in U \cap B_\rho(p)$ we have
$$||X(s,t,\phi)||_{g_{\tau,\delta}} = ||X(s,t)||_g + h_{\tau,\delta}(s,t) = \varphi(s,t) + h_{\tau,\delta}(s,t),$$
and analogously
$$||X(s,t,\phi)||_{\tilde g_{\tau,\delta}} = \varphi(s,t) + \tilde h_{\tau,\delta}(s,t),$$ for $(s,t,\phi) \in B_\rho(q) \cap \Sigma$. Since $\varphi$, $h_{\tau,\delta}$ and $\tilde h_{\tau,\delta}$ are independent of the $\phi$ coordinate, which corresponds to the action of $\crcl$, we can view them as functions on $\Sigma$;
$$\varphi : \Sigma \to \RR,$$
$$h_{\tau,\delta} : \Sigma \cap B_{\rho}(p) \to \RR,$$
$$\tilde h_{\tau,\delta} : \Sigma \cap B_{\rho}(q) \to \RR.$$
Consequently at $(s,t,\phi) \in U \cap B_\rho(p)$ and $(s,t,\phi) \in U \cap B_\rho(q)$ we respectively have
$$g_{\tau,\delta} = f^2(s,t)ds^2 + dt^2 + (\varphi(s,t) + h_{\tau,\delta}(s,t))^2d\phi^2$$
and
$$\tilde g_{\tau,\delta} = f^2(s,t)ds^2 + dt^2 + (\varphi(s,t) + \tilde h_{\tau,\delta}(s,t))^2d\phi^2.$$

The aim is now to construct a family of functions $\overline h_{\tau,\delta} : \Sigma \to \RR$ in a way that the corresponding metric
$$\overline g_{\tau,\delta} = f^2(s,t)ds^2 + dt^2 + (\varphi(s,t) + \overline h_{\tau,\delta}(s,t))^2d\phi^2$$
on $U$ is smooth and we can control the lower curvature bound of $\overline g_{\tau,\delta}$ at a certain distance from the fixed points $p$ and $q$. Moreover, we want $\overline h_{\tau,\delta}$ to satisfy the boundary value conditions imposed by $h_{\tau,\delta}$ and $\tilde h_{\tau,\delta}$ on some neighborhoods $V(p) \subset B_\rho(p)$ and $V(q) \subset B_\rho(q)$ of $p$ and $q$, which are specified during the construction.
\\ \\
Restricting the coordinates $(s,t,\phi)$ of $U$ to $\Sigma$ we obtain coordinates 
$$(s,t)$$
of $\Sigma$ with $0 < s < l$ and $0 \leq t < T$. Via these coordinates we often implicitly identify $\Sigma$ with $]0,l[ \times [0,T[$. Restricting the coordinates $(r,\theta,\alpha)$ of $B_{\rho}(p)$ to $\Sigma$ we obtain coordinates 
$$(r,\theta)$$
of $\Sigma \cap B_{\rho}(p)$ with $0 \leq r < \rho$ and $0 \leq \theta \leq \pi/2$. Recall that in this coordinates we have $h_{\tau,\delta}(r,\theta) = \sin(r)\eta_{\tau,\delta}(\theta)$. To extend the function $h_{\tau,\delta}$ to $\Sigma$ we thus need to determine the relations between the coordinates $(s,t)$ and $(r,\theta)$. Since $B_\rho(p) \cap \Sigma$ has constant curvature $1$, these relations follow from the laws of spherical geometry (compare Figure \ref{figure:coordinates}):
\begin{figure}[h!]
\centering
\def\svgwidth{0.3\textwidth}
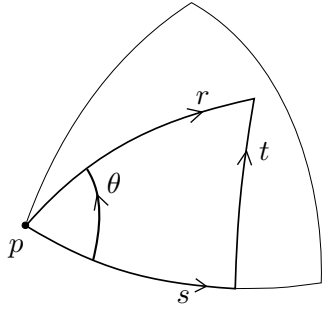
\caption{\small{Since $B_{\rho}(p) \cap \Sigma$ has constant curvature $1$, the relations between the coordinates are determined by the laws of spherical geometry.}}
\label{figure:coordinates}
\end{figure}
\begin{align}
 \cos(r) &= \cos(s)\cos(t), \label{cos(r)}\\ 
 \sin(t) &= \sin(r)\sin(\theta),\label{sin(t)}\\
 \cos(\theta) &= \frac{\cos(t) - \cos(r)\cos(s)}{\sin(r)\sin(s)} = \frac{\cos(t)\sin(s)}{\sqrt{1 - \cos^2(s)\cos^2(t)}}.\label{cos(theta)}
\end{align}

We may assume that $\rho < \pi/2$. Then the following lemma holds: 
\begin{lemma}\label{monotonicity}
 The function
\begin{align*}
 (s,t) &\mapsto \sin(r(s,t))
 \end{align*}
 is strictly increasing as a function in $s$ for a fixed $t \in [0,T[$ as well as a function in $t$ for a fixed $s \in [0,\rho[$. The function
 \begin{align*}
 (s,t) &\mapsto \theta(s,t)
  \end{align*}
 is strictly increasing as a function in $t$ for a fixed $s \in [0,\rho[$ and it is strictly decreasing as a function in $s$ for a fixed $t \in [0,T[.$
 \end{lemma}
To extend $h_{\tau,\delta}$ in a controlled way we determine the derivatives of $h_{\tau,\delta}$ with respect to the coordinates $(s,t)$.
\begin{lemma}\label{derivatives of h} Fix $0 < \tau < \pi/4$. Then 
\begin{itemize}
\item[(i)] $\partial_sh_{\tau,\delta} \geq 0$ for all $\delta > 0$,
\item[(ii)] $\partial_th_{\tau,\delta}(r,0) = (1 - m^{-1})$ and $0 < \partial_th_{\tau,\delta}(r,\theta) < (1 - m^{-1})$ for all $0 < \theta < \tau/2$, $0 < r < \rho$ and $\delta >0$,
\item[(iii)] $\partial^2_{t}h_{\tau,\delta}(r,\theta) < 0$ for all $0 < \theta < \tau/2$ for all $\delta > 0,$
\item [(iv)] for all $0 < r_0 < \rho$, for all $\delta > 0$ there exists $\epsilon(r_0,\delta) > 0$ with $\epsilon(r_0,\delta) \xrightarrow{\delta \to 0} 0$ and
$$|\partial_th_{\tau,\delta}(r,\theta)| + |\partial^2_{t}h_{\tau,\delta}(r,\theta)| \leq \epsilon(r_0,\delta)$$
for  all $r_0 \leq r \leq \rho$, for all $\tau/2 \leq \theta \leq \pi/2$.
\end{itemize}
\end{lemma}
\begin{proof}
Fix $0 < \tau < \pi/4$ and let $\delta > 0$ be arbitrary. Again we write $h = h_{\tau,\delta}$ and $\eta = \eta_{\tau,\delta}$ to avoid carrying indices. We have 
\begin{align}\label{h}
 h(s,t) = \sin(r(s,t))\eta(\theta(s,t)).
\end{align}
First we prove \textit{(i)}: From \eqref{h} we have for $(s,t) \in B_\rho(p)$
\begin{align}
\partial_sh(s,t) = \partial_s(\sin \circ r)(s,t)\eta(\theta(s,t)) + \sin(r(s,t))\eta'(\theta(s,t))\partial_s\theta(s,t).
\end{align}
By Lemma \ref{monotonicity} we have $\partial_s(\sin \circ r)(s,t) \geq 0$ and $\partial_s\theta(s,t) \leq 0$ for all $(s,t)$. Further $\eta \geq 0$ and $\eta'(\theta(s,t)) \leq 0$ whenever $\theta(s,t) \geq \tau/2$. Consequently, $\partial_sh(s,t) \geq 0$ for all $(s,t)$ with $\theta(s,t) \geq \tau/2$. Now let $(s,t)$ be given with $0 < \theta(s,t) < \tau/2$. Consequently, $\eta(\theta(s,t)) > 0$ and $\eta'(\theta(s,t)) > 0$. So we conclude that $\partial_sh(s,t) \geq 0$ if and only if
\begin{align}\label{estimation}
\left(\partial_s(\sin \circ r) + \frac{\eta' \circ \theta}{\eta \circ \theta}(\sin \circ r)\partial_s\theta\right)(s,t) \geq 0.
\end{align}
By \eqref{sin(t)}
 \begin{align*}
 0 =\partial_s\left(\sin(r)\sin(\theta) \right)
  = \sin(\theta)\partial_s\sin(r) + \cos(\theta)\sin(r)\partial_s\theta,
\end{align*}
so
\begin{align*}
\left(\partial_s(\sin \circ r) + \frac{\cos \circ \theta}{\sin \circ \theta}(\sin \circ r)\partial_s\theta\right) (s,t) = 0
\end{align*}
for all $(s,t)$. Since $\partial_s(\sin \circ r) \geq 0$, and $(\sin \circ r)\partial_s\theta\leq 0$, to prove \eqref{estimation} it is enough to prove that
\begin{align*}
\frac{\eta'(\theta)}{\eta(\theta)} \leq \frac{\cos(\theta)}{\sin(\theta)}
\end{align*}
for all $0 < \theta < \tau/2$, or equivalently
\begin{align}
 (\eta'\sin)(\theta) \leq (\eta \cos)(\theta)\label{tada}
\end{align}
for all $0 < \theta < \tau/2$. To prove \eqref{tada} first note that equality holds for all $0 < \theta < \tau(\delta)$ by definition of $\eta$. Further $\eta''(\theta) \leq -\eta(\theta)$ for all $0 < \theta < \tau/2$ and therefore
\begin{align}
 (\eta'\sin)'(\theta) &= \eta''(\theta)\sin(\theta) + \eta'(\theta)\cos(\theta)\\
 &\leq -\eta(\theta)\sin(\theta) + \eta'(\theta)\cos(\theta) = (\eta\cos)'(\theta)
\end{align}
for all $0 < \theta < \tau/2$. \eqref{tada} and therefore \textit{(i)} follows.

Next we prove \textit{(ii)} and \textit{(iii)}: One way to do this is via calculations similar to those above. However, the following argument is less technical (although a bit artificial): Consider the metric 
$$b(r,\theta,\alpha) = dr^2 + \sin^2(r)d\theta^2 + h^2(r,\theta)d\alpha^2 = dr^2 + \sin^2(r)(d\theta^2 + \eta^2(\theta)d\alpha^2).$$
on $B_{\rho}(p)$. From Lemma \ref{curvature} and property $(v)$ of $\eta = \eta_{\tau,\delta}$ it follows that the metric $d\theta^2 + \eta^2(\theta)d\alpha^2$ on $\Sigma_pM^*$ has curvature $\geq 1$ on the set $\{(\theta,\alpha) \mid 0 < \theta < \tau/2\}$. Since $b$ is the spherical suspension of the metric $d\theta^2 + \eta(\theta)d\alpha^2$, it follows from Lemma \ref{cone and suspension} that $b$ has positive curvature on the set $\{(r,\theta,\alpha) \mid 0 < \theta < \tau/2, 0 < r < \rho \}$. It is clear that $b$ admits a polar $\crcl$-action with section $\Sigma \cap B_{\rho}(p)$ and associated Killing vector field $\partial_{\alpha}$ with $||\partial_{\alpha}||_b = h(r,\theta)$. From Lemma \ref{Hessian} it follows that  $h$ is a strictly concave function on the set $\{(r,\theta) \in B_{\rho(p)} \cap \Sigma \mid 0 < \theta < \tau/2\}$. Since integral curves of the coordinate field $\partial_t$ are geodesics, it follows that $\partial^2_{t}h(r,\theta) < 0$ for all $0 < \theta < \tau/2$ for all $0 < r < \rho$. This proves \textit{(iii)}. Observe that 
$$\partial_th(r,\theta) = \partial_t(\sin \circ r)(r,\theta)\eta(\theta) + \sin(r)\eta'(\theta)\partial_t\theta(r,\theta) > 0$$
for all $0 < \theta < \tau/2$ by definition of $\eta$ and Lemma \ref{monotonicity}, and for all $0 < s < \rho$
$$\partial_th(s,0) = \sin(r(s,0))\eta'(0) \partial_t\theta(s,0) = \eta'(0) = 1 - m^{-1},$$
since 
\begin{align*}
 \partial_t\theta(s,t) = &\left(\sqrt{1 - \frac{\cos^2(t)\sin^2(s)}{1 - \cos^2(s)\cos^2(t)}}\right)^{-1}\frac{\sin(t)\sin(s)}{(1 - \cos^2(s)\cos^2(t))^{3/2}}\\
 = & \left(\sqrt{1 - \cos^2(s)\cos^2(t) - \cos^2(t)\sin^2(s)}\right)^{-1}\frac{\sin(t)\sin(s)}{1 - \cos^2(s)\cos^2(t)}\\
= & \left(\sqrt{1 - \cos^2(t)}\right)^{-1}\frac{\sin(t)\sin(s)}{1 - \cos^2(r)}\\
 = & \frac{\sin(s)}{\sin^2(r)}
\end{align*}
and $r(s,0) = s$. Now \textit{(ii)} follows from \textit{(iii)}.

To see that \textit{(iv)} holds, observe that for fixed $0 < \tau < \pi/4$ the sum
$|\eta_{\tau,\delta}(\theta)| + |\eta_{\tau,\delta}'(\theta)| + |\eta_{\tau,\delta}''(\theta)|$
converges uniformly to $0$ on the interval $[\tau/2,\pi/2]$ as $\delta$ goes to $0$. Therefore, also $|\partial_th_{\tau,\delta}(r,\theta)| + |\partial^2_{t}h_{\tau,\delta}(r,\theta)|$ converges to uniformly to $0$ on any subset  of $B_\rho(p)$ whose points have radius $r$ bounded below by a positive constant and angle bigger than $\tau/2$.
\end{proof}
After possibly decreasing $T$, we may additionally assume that $(\rho/2,t) \in B_\rho(p) \cap \Sigma$ for all $t \in [0,T[$. Then the following definition makes sense.
\begin{definition}
Let
\begin{align}
\overline h_{\tau,\delta} : \Sigma &\to \RR,\\
(s,t) &\mapsto h_{\tau,\delta}(\rho/2,t).
\end{align}
\end{definition}
Similarly to Proposition \ref{resolution of Brhop} we can use this family of functions to resolve the singularities of $U$, while keeping control of the curvature outside of a small neighborhood of the fixed points:
\begin{proposition}\label{resolution of U}
 $\overline h_{\tau,\delta} : \Sigma \to \RR$ is smooth for all $0 < \tau < \pi/4$ and $\delta > 0$ and
 \begin{align}
\overline g_{\tau,\delta}(s,t,\phi) := f^2(s,t)ds^2 + dt^2 + (\varphi + \overline h_{\tau,\delta})^2(s,t)d\phi^2
\end{align}
 defines a smooth metric on $U$. Moreover, there exists $0 < \tau_0 < \pi/4$ such that for fixed $0 < \tau < \tau_0$ and $\epsilon > 0$ on the set 
 $$S := \{(s,t, \phi) | \rho/4 \leq s \leq l - \rho/4\}$$
 the metric $\overline g_{\tau,\delta}$ has positive curvature if $c > 0$, and curvature bounded below by $mc - \epsilon$ if $c \leq 0$, for all $\delta > 0$ sufficiently small (depending on $\tau$). Independent of $c$, $\overline g_{\tau,\delta}$ has positive curvature on the set $S \cap (B_\rho(p) \cup B_\rho(q))$, again for all sufficiently small $\delta > 0$ (depending on $\tau$).

\end{proposition}
\begin{proof}
It is clear that $\overline h_{\tau,\delta} : \Sigma \to \RR$ is smooth for all $0 < \tau < \pi/4$ and $\delta > 0$. Since the metric $\hat g$ on $U$ is smooth (see \eqref{hat g}) it follows that $\overline g_{\tau,\delta}$ is smooth if and only if
 \begin{align}\label{label}
dt^2 + (\varphi + \overline h_{\tau,\delta})^2d\phi^2
 \end{align}
 is a smooth section of $Sym^2(TU)$. This is true for $t > 0$, so we have to show smoothness at $t = 0$, that is at points of $\gamma$. From the definition of $\eta_{\tau,\delta}$ and \eqref{sin(t)} we see that for all sufficiently small $t > 0$ and for all $s \in \ ]0,l[$ we have
 \begin{align}
 \overline h_{\tau,\delta}(s,t) = h(\rho / 2,t) &= \sin(r(\rho / 2,t))\eta_{\tau,\delta}(\theta(\rho / 2,t))\\
 &= (1 - m^{-1})\sin(r(\rho / 2,t))\sin(\theta(\rho / 2,t)) = (1 - m^{-1})\sin(t).\nonumber
 \end{align}
 Set 
 $$\xi(s,t) = (\varphi(s,t) + (1 - m^{-1})\sin(t))^2.$$
 Then we need to prove that 
\begin{align}
\omega(s,t,\phi) = (dt^2 + \xi(s,t)d\phi^2)(s,t,\phi)
\end{align}
is smooth. For the following arguments compare \cite{petersen98}, sections 1.3.4 and 1.4.1: Let $x = t \cos \phi$ and $y = t \sin \phi$. Then the coordinates $(s,x,y)$ are nonsingular and we have, for example,
$$\omega_{xx}(s,x,y) := \omega_{(s,x,y)}(\partial_x,\partial_x) = 1 + \frac{\frac{\xi(s,t)}{t^2} - 1}{t^2}y^2,$$
for $t = \sqrt{x^2 + y^2}$. Therefore, we have to check that the function 
\begin{align*}
 (s,x,y) \mapsto \frac{\frac{\xi(s,t)}{t^2} - 1}{t^2} =: \zeta(s,t)
\end{align*}
is smooth. Again from the smoothness of $\hat g$, it follows that $dt^2 + m^2\varphi^2(s,t)d\phi^2$ is smooth and therefore the following equations hold: 
$$\varphi(s,0) = 0,~\partial_t\varphi(s,t) = m^{-1},~\partial^{(even)}_t\varphi(s,0) = 0.$$
Consequently,
\begin{align*}
 \xi(s,0) = 0, \partial^{(odd)}_t\xi(s,0) = 0 \text{ and } \partial^2_t\xi(s,0) = 2.
\end{align*}
Since $\xi$ is a smooth function on $\Sigma$, considering its Taylor approximation with respect to $t$, it follows that $\zeta$ is smooth, considered as a function on $\Sigma$, and
\begin{align*}
\partial_t^{(odd)}\zeta(s,0) = 0.
\end{align*}
Standard calculus methods show that these conditions are sufficient (and necessary) for $(s,x,y) \mapsto \zeta(s,\sqrt{x^2 + y^2})$ to be smooth at $(s,0,0)$. Therefore, $\omega_{xx}$ is smooth. The smoothness of the remaining coefficients of $\omega$ follows analogously.
\\ \\
Now we calculate the curvature operator of $\overline g_{\tau,\delta}$. We write $\overline g= \overline g_{\tau,\delta}$ and $\overline h = \overline h_{\tau,\delta}$ for convenience. First, we calculate the Hessian $\nabla^2\overline h$ of $\overline h : (\Sigma,g) \to \RR$: Since the integral curves of the coordinate field $\partial_t$ are geodesics we have 
\begin{align*}
 \nabla^2\overline h(\partial_t,\partial_t) = \partial^2_{t}\overline h.
\end{align*}
Since $\overline h$ is constant in $s$-direction and $g_{\Sigma} = f^2ds^2 + dt^2$, it is clear that $\nabla \overline h = \partial_t\overline h\partial_t$. Hence
\begin{align*}
 \nabla^2\overline h(\partial_t,\partial_s) = \nabla^2\overline h(\partial_s,\partial_t) = g(\partial_s,\nabla_{\partial_t}(\partial_t\overline h\partial_t)) =g( \partial_s,\partial^2_{tt}\overline h\partial_t) = 0
\end{align*}
and
\begin{align*}
 \nabla^2\overline h(\partial_s,\partial_s) &= g( \partial_s,\nabla_{\partial_s}(\partial_t\overline h)\partial_t )\\
 &= \partial_t\overline hg( \partial_s,\nabla_{\partial_s}\partial_t)\\
 &= \partial_t\overline hg( \partial_s,\nabla_{\partial_t}\partial_s)\\
 &= \frac 1 2 \partial_t\overline h \partial_tg( \partial_s,\partial_s)  = f \partial_t\overline h \partial_t f.
\end{align*}
Altogether with respect to the basis $\{\partial_t,\partial_s\}$ of $T_{(s,t)}\Sigma$ with $t > 0$ we have
\begin{align}
            \nabla^2\overline h = 
            \begin{pmatrix}
             \partial^2_{t}\overline h & 0\\
             0 & f \partial_t\overline h \partial_t f
            \end{pmatrix}.\label{hessian overline h}
\end{align}

Let $(s,t,\phi) \in U$ with $t > 0$ and denote the ordered basis 
$$\{(\varphi + \overline h)^{-1}(s,t)\partial \phi, \partial_t, f^{-1}(s,t)\partial_s\}$$
of $T_{(s,t,\phi)}U$ by $\{b_1,b_2,b_3\}$. Then this basis is orthonormal with respect to $\overline g$ and $b_2$ and $b_3$ are tangent to $\Sigma$. By Corollary \ref{curvature and Hessian} the curvature operator $\overline{\mathcal R}$ of $\overline g$ at $(s,t,\phi)$ with respect to this basis is given by
\begin{align*} 
&(\langle \overline{\mathcal R}(b_i \wedge b_j),b_k \wedge b_l \rangle)_{(1 \leq i<j \leq 3,\ 1 \leq k<l \leq 3)}\nonumber \\
=& \begin{pmatrix}
   -(\varphi + \overline h)^{-1}(s,t)\nabla^2(\varphi + \overline h)(b_j,b_l)_{(j,l \in \{2,3\})} & 0 \\
   0  & \langle \overline R_{\Sigma}(b_2 \wedge b_3),b_2 \wedge b_3 \rangle\nonumber
   \end{pmatrix}.
\end{align*}
Since $\overline g_{\Sigma} = g_{\Sigma}$, it follows that 
$$\langle \overline R_{\Sigma}(b_2 \wedge b_3),b_2 \wedge b_3 \rangle \geq c$$ and 
$$\langle \overline R_{\Sigma}(b_2 \wedge b_3),b_2 \wedge b_3 \rangle = 1$$ given that $(s,t,\phi) \in (B_\rho(p) \cup B_\rho(q))$. Note that for the same reason $\nabla^2 = \nabla^2_g = \nabla^2_{\overline g}$ on $\Sigma$.

Let us first assume that $c > 0$. We fix $0 < \tau < \tau_0 < \pi/4$, where $\tau_0$ is specified later. As usual we write $\nabla^2 \overline h \geq d$ if $\nabla^2 \overline h (v,v) \geq d$ for every unit vector $v$. Then we have to show that there exists a constant $\tilde c > 0$ such that for all $\delta > 0$ sufficiently small we have 
$$-(\varphi + \overline h)^{-1}(s,t)\nabla^2(\varphi + \overline h) \geq \tilde c$$
for all $(s,t) \in W := [\rho/4,l - \rho/4] \times ]0,T[$. We have
\begin{align*}
&-(\varphi + \overline h)^{-1}(s,t)\nabla^2(\varphi + \overline h)\\
 =\ &\frac{\varphi(s,t)}{(\varphi + \overline h)(s,t)}\left(- \varphi^{-1}(s,t)\nabla^2\varphi  - \varphi^{-1}(s,t)\nabla^2\overline h \right)\\
\geq \ &\frac{\varphi(s,t)}{(\varphi + \overline h)(s,t)}\left(c - \varphi^{-1}(s,t)\nabla^2\overline h \right).
\end{align*}
 By Lemma \ref{derivatives of h} there exists a constant $c_1 > 0$, independent of $\delta$, such that $0 \leq \overline h(s,t) \leq c_1t$ for all $(s,t) \in \Sigma$. Therefore, there exists a constant $c_2 > 0$ such that
$$\frac{\varphi(s,t)}{(\varphi + \overline h)(s,t)} \geq \frac{\varphi(s,t)}{\varphi(s,t) + c_1t} \geq c_2$$
for all $(s,t) \in W$ (here it is important that $\rho/4 \leq s \leq l - \rho/4$ for $(s,t) \in W$, since $\varphi(s,t) \to 0$ for $s \to 0$ or $s \to l$ and any fixed $t$). Hence it suffices to show that 
$$ \varphi^{-1}(s,t)\nabla^2\overline h \leq c/2$$
for all $(s,t) \in W$. By \eqref{hessian overline h}, and since $\partial_t$ and $f^{-1}\partial_s$ are orthonormal, it is enough to prove that for all sufficiently small $\delta > 0$, for all such $(s,t) \in W$ we have
\begin{align}\label{url1}
(\varphi^{-1}\partial^2_{t}\overline h)(s,t) \leq c/2
\end{align}
and
\begin{align}\label{url2}
(\varphi^{-1}f^{-1} \partial_t\overline h \partial_t f)(s,t) \leq c/2.
\end{align}
First note that there exists a unique $0 < t_0 < T$ with $\theta(\rho/2,t_0) = \tau/2$. 

To prove \eqref{url1}, observe from Lemma \ref{derivatives of h} that $\partial_t^2\overline h (s,t) \leq 0$ for $(s,t) \in \ ]0,l[ \times [0,t_0]$, so \eqref{url1} holds for all $(s,t) \in W$ with $t \leq t_0$. Again from Lemma \ref{derivatives of h}, it follows that $\partial^2_t \overline h(s,t)$ converges uniformly to $0$ for $\delta \to 0$ on the set $]0,l[ \times [t_0,T[$. Since $\varphi^{-1}$ is uniformly bounded on $]0,l[ \times [t_0,T[$, it follows that \eqref{url1} holds on $[\rho/4,l-\rho/4] \times [t_0,T[$ for all sufficiently small $\delta > 0$. So it holds on $W$ for all $\delta > 0$ sufficiently small.

To prove \eqref{url2} we first claim that there exists a $0 < T_0 < T$ such that 
 \begin{align}
  f^{-1}(s,t)\frac{\partial_tf}{\varphi}(s,t) < 0
 \end{align}
for all $(s,t) \in [\rho/4,l - \rho/4] \times [0, T_0]$:  Since $f(s,0) = 1$, $\partial_t f (s,0) = 0$ and $\partial_t \varphi(s,0) = m^{-1}$ for all $s \in ]0,l[$, we see from Lemma \ref{curvature}:
 \begin{align}
  \lim_{t \to 0} f^{-1}(s,t)\frac{\partial_tf}{\varphi}(s,t) = m\partial^2_{t}f(s,0) = -m\sec_{\Sigma}(s,0)  \leq -mc < 0.
 \end{align}
From continuity and since $[\rho/4,l - \rho/4]$ is compact, the claim follows.

Now we choose $0 < \tau_0$ small so that we may assume that $\overline h$ is supported within $]0,l[ \times [0,T_0]$. Then $0 < t_0 < T_0$. Once more by Lemma \ref{derivatives of h}, it follows that $\partial_t\overline h (s,t) > 0$ for $(s,t) \in ]0,l[ \times [0,t_0[$, so
$$(\varphi^{-1}f^{-1} \partial_t\overline h \partial_t f)(s,t) < 0$$
for $(s,t) \in [\rho/4,l - \rho/4] \times [0,t_0]$, and $\partial_t \overline h$ converges uniformly to $0$ for $\delta \to 0$ on the set $]0,l[ \times [t_0,T_0]$, so 
$$(\varphi^{-1}f^{-1} \partial_t\overline h \partial_t f)(s,t) \leq c/2$$
for all $(s,t) \in [\rho/4,l - \rho/4] \times [t_0,T_0]$ for all sufficiently small $\delta > 0$, since $\varphi^{-1}f \partial_t f$ is bounded on $[\rho/4,l-\rho/4] \times [t_0,T_0]$. This finishes the proof in the case $c > 0$.
\\ \

Now we assume that $c \leq 0$.  Again we fix $0 < \tau < \tau_0 < \pi/4$. We have to show that 
$$-(\varphi + \overline h)^{-1}(s,t)\nabla^2(\varphi + \overline h) \geq mc - \epsilon$$
for all sufficiently small $\delta > 0$. We have by Lemma \ref{Hessian}, together with $c \leq 0$,
\begin{align*}
&-(\varphi + \overline h)^{-1}\nabla^2(\varphi + \overline h)\\
\geq \ & c - (\varphi + \overline h)^{-1}\nabla^2\overline h.
\end{align*}
Therefore, it suffices to show that 
\begin{align}\label{hadoken}
             -(\varphi + \overline h)^{-1}\partial^2_{t}\overline h \geq (m - 1)c - \epsilon
\end{align}
and
\begin{align}\label{shuryoken}
 -f^{-1}  \partial_t f (\varphi + \overline h)^{-1} \partial_t\overline h \geq (m - 1)c - \epsilon
\end{align}
for all $\delta > 0$ sufficiently small.

First it follows as in the case of positive curvature that in fact
$$ -(\varphi + \overline h)^{-1}\partial^2_{t}\overline h(s,t) > -\epsilon$$ for all $\delta > 0$ sufficiently small, for all $(s,t) \in [\rho/4,l - \rho/4] \times [0,T]$. So \eqref{hadoken} holds.

 To prove \eqref{shuryoken} we first claim that  there exists $0 < T_0 < T$  such that
 \begin{align}\label{dtsigmavarphi}
  -f^{-1}(s,t) \frac{\partial_tf}{(\varphi + \overline h)}(s,t) > mc - \frac{\epsilon}{(1 - m^{-1})}\end{align}
for all $(s,t) \in [\rho/4,l - \rho/4] \times [0, T_0]$: At points satisfying $\partial_t f \leq 0$ this is immediate, since $c \leq 0$. If $\partial_t f > 0$ we estimate 
$$-f^{-1}(s,t)\frac{\partial_tf}{\varphi + \overline h}(s,t) \geq -f^{-1}(s,t)\frac{\partial_tf}{\varphi}(s,t)$$
and the claim follows as in the case of positive curvature using
$$\lim_{t \to 0} -f^{-1}(s,t)\frac{\partial_tf}{\varphi}(s,t) = -m\partial^2_tf(s,0) \geq mc.$$ 

Again we choose $\tau_0$ small so that $\overline h$ is supported in $]0,l[ \times [0,T_0]$. By Lemma \ref{derivatives of h} we have $0 \leq \partial_t\overline h(s,t) \leq (1 - m^{-1}) < 1$ for $(s,t) \in  ]0,l[ \times [0,t_0]$ (with $t_0$ defined as in the case of positive curvature). Therefore,
$$-f(s,t)\frac{\partial_t f}{\varphi + \overline h}(s,t)\partial_t\overline h(s,t) > \left(mc - \frac{\epsilon}{1 - m^{-1}}\right)(1 - m^{-1}) = (m - 1)c - \epsilon$$
for $(s,t) \in [\rho/4,l - \rho/4] \times [0,t_0]$. Since $\partial_t \overline h$ converges uniformly to $0$ on the set $[\rho/4,l - \rho/4] \times [t_0,T_0]$, the estimate holds for all $(s,t) \in [\rho/4,l - \rho/4] \times [0,T[$ for all $\delta > 0$ sufficiently small. 

It remains to show that $\overline g$ has positive curvature on the set $\{(s,t,\phi) \mid \rho/4 \leq s \leq l - \rho/4\} \cap (B_\rho(p) \cup B_\rho(q))$ for all sufficiently small $\delta$. This is of course clear from the positively curved case since $\Sigma \cap (B_\rho(p) \cup B_\rho(q))$ has constant curvature $1$.
\end{proof}
\subsection{\texorpdfstring{Resolution of $E^*$}{Resolution of E*}}\label{resolving E}
In this section we glue the resolutions of $B_\rho(p)$, $B_\rho(q)$ and $U$ together. This is achieved by first fitting the functions $\varphi + h, \varphi + \tilde h$ and $\varphi + \overline h$ together (in a non-smooth way). For that we simply set
\begin{align}
 \psi_{\tau,\delta}(s,t) := \begin{cases}
                \varphi(s,t) + h_{\tau,\delta}(s,t) \text{ for } 0 < s \leq \frac \rho 2,\\
                \varphi(s,t) + \overline h_{\tau,\delta}(s,t)\text{ for } \frac \rho 2 \leq s \leq l - \frac \rho 2,\\
                \varphi(s,t) + \tilde h_{\tau,\delta}(s,t) \text{ for } l - \frac \rho 2 \leq s < l.
               \end{cases}
\end{align}
Recall that $(\tilde r, \tilde \theta)$ denote coordinates of $B_\rho(q) \cap \Sigma$, with radial direction $\tilde r$ and $\tilde \theta = 0$ corresponding $d/ds_{|s = 0}\gamma(l - s)$. Since $B_\rho(q) \cap \Sigma$ has constant curvature $1$ as well as $B_\rho(p) \cap \Sigma$, it follows $\theta(s,t) = \tilde \theta(l - s,t)$ and $r(s,t) = \tilde r(l - s,t)$ for all $(s,t) \in B_\rho(p) \cap \Sigma$. Thus 
\begin{align*}\overline h_{\tau,\delta}(\rho / 2,t) &= \overline h_{\tau,\delta}(l - \rho / 2,t)\\ &= h(\rho / 2,t)\\ &= \sin(r(\rho / 2,t))\eta(\theta(\rho / 2,t))\\
&= \sin(\tilde r(l - \rho / 2,t))\eta(\tilde \theta(l -\rho / 2,t))\\ &= \tilde h ( l - \rho / 2,t).
\end{align*}
Hence $\psi_{\tau,\delta} : \Sigma \to \RR$ is continuous for all $\tau, \delta$.
\begin{lemma}\label{psi}
For all sufficiently small $0 < \tau < \pi/4$ and for all sufficiently small $\delta > 0$ (depending on $\tau$) the following is true:
 \begin{itemize}
  \item[(i)] There exists an open neighborhood $W$ of $\gamma$ in $\Sigma$ such that $\psi_{\tau,\delta}(s,t) = \varphi(s,t) + (1 - m^{-1})\sin(t) = (\varphi + \overline h_{\tau,\delta})(s,t)$ for all $(s,t) \in W$. 
  \item[(ii)] The set of points at which $\psi_{\tau,\delta}$ is not smooth is contained in a compact set $K \subset (B_{\rho}(p) \cup B_{\rho}(q)) \cap \Sigma$ disjoint from $\gamma$. 
  \item[(iii)] ${\psi_{\tau,\delta}}_{|B_{\rho}(p) \cap \Sigma}$ and ${\psi_{\tau,\delta}}_{|B_{\rho}(q) \cap \Sigma}$ are strictly concave.
 \end{itemize}
 \end{lemma}
\begin{proof}
By definition of $h_{\tau,\delta}$ and $\eta_{\tau,\delta}$, together with \eqref{sin(t)}, we have
$$h_{\tau,\delta}(s,t) = \sin(r(s,t))\eta_{\tau,\delta}(\theta(s,t)) = (1 - m^{-1})\sin(r(s,t))\sin(\theta(s,t)) = (1 - m^{-1})\sin(t)$$
for all $(s,t) \in B_{\rho}(p) \cap \Sigma$ with $\theta(s,t) < \tau(\delta)$. Analogously it follows that $\tilde h_{\tau,\delta}(s,t) = (1 - m^{-1})\sin(t)$ for all $(s,t) \in B_{\rho}(q) \cap \Sigma$ with $\tilde \theta (s,t) < \tau(\delta)$, and also $\overline h_{\tau,\delta}(s,t) = h_{\tau,\delta}(\rho / 2,t) = (1 -m^{-1})\sin(t)$ for sufficiently small $t > 0$. Thus $(i)$ follows.

It is clear by definition of $\overline h_{\tau,\delta}$, that there exists $0 < t_1 < T$ such that $\overline h_{\tau,\delta}$ is supported in $]0,l[ \times [0,t_1]$ for all sufficiently small $0 < \tau < \pi/4$. It follows from \textit{(i)} that for all such $0 < \tau < \pi/4$ there exist $0 < t_2(\tau) < t_1$ such that the set of points at which $\psi$ is not smooth is contained in the set 
$$\{(s,t) \in \Sigma \mid (s,t) \in (\{\rho / 2\} \times [t_2(\tau),t_1]) \cup (\{l - \rho / 2\} \times [t_2(\tau),t_1])\}.$$
This proves $(ii)$.

By Lemma \ref{derivatives of h} we have $\partial_s h_{\tau,\delta} \geq 0$ for all $(s,t) \in B_{\rho}(p) \cap \Sigma$ and clearly $\partial_s \overline h_{\tau,\delta}(s,t) = 0$ for all $(s,t)$. Since $h_{\tau,\delta}(\rho / 2,t) = \overline h_{\tau,\delta}(\rho / 2,t)$, it follows that 
 \begin{align*}
 	 {\psi_{\tau,\delta}}_{|B_{\rho}(p) \cap \Sigma} = \min\{(\varphi + h_{\tau,\delta})_{|B_{\rho}(p) \cap \Sigma},(\varphi+ \overline h_{\tau,\delta})_{B_{\rho}(p) \cap \Sigma}\}.
\end{align*}
From Lemma \ref{Hessian} and Propositions \ref{resolution of Brhop}  and \ref{resolution of U} for fixed $0 < \tau < \pi/4$ and all sufficiently small $\delta > 0$  the functions $(\varphi + h_{\tau,\delta})_{|\Sigma \cap B_{\rho}(p)}$ and $(\varphi+ \overline h_{\tau,\delta})_{|\Sigma \cap B_{\rho}(p)}$ are strictly concave. Therefore, also their minimum is strictly concave. An analogous argument applies for ${\psi_{\tau,\delta}}_{B_\rho(q) \cap \Sigma}$.
\end{proof}

Now we smooth the function $\psi_{\tau,\delta}$ in a way that its concavity  properties are preserved. For this we use the following result by Greene-Wu \cite{greene-wu76}:
\begin{lemma}\label{riemannian convolution}
 Let $f : N \to \RR$ be a strictly convex (strictly concave) function on a Riemannian manifold $N$ and $L_1$, $L_2$ be compact subsets of $N$ with $L_1 \subset L_2$. Suppose that $f$ is smooth in a neighborhood of $L_1$. Then there exists a family $\{f_\epsilon : N \to \RR\}$ of strictly convex (strictly concave) functions such that
 \begin{itemize}
  \item[(i)] $f_\epsilon$ is smooth in a neighborhood of $L_2$,
  \item[(ii)] $f_\epsilon$ and any derivative of $f_\epsilon$ of order $k$ converge uniformly on $L_1$ to $f$ and the corresponding derivative of order $k$, respectively.
 \end{itemize}
 \end{lemma}
According to Lemma \ref{psi} let $\tau > 0$ and $\delta > 0$ such that $\psi_{\tau,\delta}$ is continuous, strictly concave when restricted to $(B_{\rho}(p) \cup B_{\rho}(q)) \cap \Sigma$ and there exists a compact subset $K \subset (B_{\rho}(p) \cup B_{\rho}(q)) \cap \Sigma$ which is disjoint from $\gamma$ and contains all points at which $\psi_{\tau,\delta}$ is not smooth. 

Let $L_2 = \overline{B_\epsilon(K)}$ denote a closed $\epsilon$-neighborhood of $K$ which is disjoint from $\gamma$ and contained in $(B_{\rho}(p) \cup B_{\rho}(q)) \cap \Sigma$. Set $L_1 := L_2 \setminus B_{\epsilon/2}(K)$. By Lemma \ref{riemannian convolution} there exists a family of strictly concave functions 
$$\{\psi^n_{\tau,\delta} : (B_{\rho}(p) \cup B_{\rho}(q)) \cap \Sigma \to \RR\}$$
that are smooth restricted to $L_2$ and $\psi^n_{\tau,\delta}$ as well as all its derivatives converge uniformly on $L_1$ to $\psi_{\tau,\delta}$ and the corresponding derivative, respectively. We glue the functions $\psi^n_{\tau,\delta}$ and $\psi_{\tau,\delta}$ along $L_1$:

For that let $j : (B_{\rho}(p) \cup B_{\rho}(q)) \cap \Sigma \to [0,1]$ be a smooth function such that $j \equiv 1$ on $L_2 \setminus L_1 = B_{\epsilon/2}(K)$ and $j \equiv 0$ on $((B_{\rho}(p) \cup B_{\rho}(q)) \cap \Sigma)\setminus L_2$. Set
\begin{align}\label{tilde varphi}
  \varphi^n_{\tau,\delta} = j \psi_{\tau,\delta}^n + (1 - j)\psi_{\tau,\delta}.
\end{align}
Then $\varphi^n_{\tau,\delta} : B_{\rho}(p) \cap \Sigma \to \RR$ is smooth. Further for all $p \in L_1$ and for all $v \in T_p\Sigma$ with $||v|| = 1$ we have (writing $\varphi^n = \varphi^n_{\tau,\delta}$)
\begin{align}
 \nabla^2 \varphi^n(v,v) &= j(p) \nabla^2\psi^n(v,v) + (1- j)(p)\nabla^2\psi(v,v)\\
 &+ 2v(j)v(\psi^n - \psi)\\ 
 &+ \nabla^2j(v,v)(\psi^n - \psi)(p).
\end{align}
Since $\nabla j$ and $\nabla^2j$ are continuous and supported in $L_1$, and $\psi^n$ and its derivatives converge uniformly on $L_1$ to $\psi$ and its corresponding derivatives, respectively, it follows that 
$$\nabla^2 \varphi^n_{\tau,\delta} < 0,$$
i.e. $\varphi^n_{\tau,\delta}$ is strictly concave, on $(B_{\rho}(p) \cup B_{\rho}(q)) \cap \Sigma$ for all sufficiently large $n$. Also $\varphi^n_{\tau,\delta}$ coincides with $\psi_{\tau,\delta}$ on $\Sigma \setminus B_\epsilon(K)$.  Using the functions $\varphi^n_{\tau,\delta}$ we can prove the following lemma.
\begin{lemma}\label{44}
 Let $\gamma :\ ]0,l[ \to (M^*,d)$ be an arc length geodesic, parameterizing an arc of $E^*$ and connecting two fixed points $p$ and $q$. Let $\epsilon > 0$ and $\gamma$ have constant type $\mathbb Z_m$. Then there exists a complete Alexandrov metric $d_\gamma^\epsilon$ on $M^*$ satisfying the following properties:
 \begin{itemize}
  \item[(i)] $d_\gamma^\epsilon$ is induced by a smooth Riemannian metric $g_\gamma^\epsilon$ on $(M^* \setminus (F \cup E^*)) \cup \gamma(]0,l[)$.
  \item[(ii)] $\curv d_\gamma^\epsilon > mc - \epsilon$ if $c \leq 0$  ($\curv d_\gamma^\epsilon > 0$ if $c > 0$).
  \item[(iii)] $d_{GH}((M^*,d^\epsilon_\gamma),(M^*,d)) \leq \epsilon$.
  \item[(iv)] Let $\dot \gamma^-(0) = \frac d {ds}_{|s = 0}\gamma (l - s) \in \Sigma_qM^*$ and 
  $$V_{\epsilon,\gamma} = s_p(B_\epsilon(\dot \gamma(0))) \cup  s_q(B_\epsilon(\dot \gamma^-(0))) \cup B_\epsilon(\gamma([\rho/2,l-\rho/2]))$$
  (compare Definition \ref{suspended}). Then $g_\gamma^\epsilon$ coincides with $g$ on $M^* \setminus V_{\epsilon,\gamma}$.
  \item[(v)] $\Sigma_p(M^*,d_\gamma^\epsilon)$ has curvature bounded below by $1$ and admits an effective isometric $\crcl$-action. There exists $\tilde \rho > 0$ such that $B_{\tilde \rho}(p)$ is isometric to a neighborhood of the tip of the spherical suspension of $\Sigma_p(M^*,d_\gamma^\epsilon)$. The corresponding statement holds for $B_{\tilde \rho}(q)$.
 \end{itemize}
 \end{lemma}
\begin{proof}
Consider the metric
\begin{align}\label{tilde g}
 G^n_{\tau,\delta}(s,t,\phi) = f^2(s,t)ds^2 + dt^2 + (\varphi^n_{\tau,\delta})^2(s,t)d\phi^2
\end{align}
defined on $U$. Since $\varphi^n_{\tau,\delta}$ coincides with $\psi_{\tau,\delta} = \varphi + \overline h_{\tau,\delta}$ in a neighborhood of $\gamma$ we see that $G^n_{\tau,\delta}$ is smooth by proposition \ref{resolution of U}. Also by construction $G^n_{\tau,\delta} = g$ on $U \setminus V_{\epsilon,\gamma}$ for all sufficiently small $\tau > 0$, independent of $n$ and $\delta$. Thus we can extend $G^n_{\tau,\delta}$ via $g$ to $M^*$ with induced metric $d^n_{\tau,\delta}$. Then the properties $(i)$ and $(iv)$ are satisfied for $d^n_{\tau,\delta}$. Also by construction $\varphi^n_{\tau,\delta} = (\varphi + \overline h_{\tau,\delta})$ on $U \setminus (B_\rho(p) \cup B_\rho(q))$. Therefore, according to Proposition \ref{resolution of U}, for all $0 < \tau$ sufficiently small there exists $\delta > 0$ such that $G^n_{\tau,\delta}$ restricted to $U \setminus (B_\rho(p) \cup B_\rho(q)$ has positive curvature, if $c > 0$, respectively curvature bounded below by $mc - \epsilon$, if $c \leq 0$. Fixing such $\tau$ and $\delta$ and choosing $n$ large such that $\varphi_{\tau,\delta}^n$ is strictly concave when restricted to $B_\rho(p) \cap \Sigma$ and $B_\rho(q) \cap \Sigma$, it follows that  $G_{\tau,\delta}^n$ has globally positive curvature or curvature bounded below by $mc - \frac 1 n$, if respectively $c > 0$ or $c \leq 0$. It is clear that property $(iii)$ holds for all sufficiently small $\tau > 0$. Moreover, property $(v)$ holds since $\varphi^n_{\tau,\delta} = \varphi + h_{\tau,\delta}$ near $p$ and $\varphi^n_{\tau,\delta} = \varphi + h_{\tau,\delta}$ near $q$.
\end{proof}
Now, for each arc $e_i$ of $E^*$ choose a parametrization by an arc length geodesic $\gamma_i$. Since there are only finitely many such arcs, we can find $\epsilon > 0$ such that for $ i \neq j$ the intersection $V_{\epsilon,\gamma_i} \cap V_{\epsilon,\gamma_j}$ is either empty, or equals one or two fixed points, at which the arcs $e_i$ and $e_j$ meet. Via the previous Lemma for each $i$ we can choose a smooth metric $g_i^\epsilon$ on $V_{\epsilon,\gamma_i}$ satisfying the above properties, yielding a smooth Riemannian metric $g^\epsilon$ on $\bigcup_i V_{\epsilon,\gamma_i} \setminus F$. This metric extends smoothly via $g$ to $M^* \setminus F$. Gluing back the points of $F$ (equivalently we perform the metric completion) we obtain a complete metric $d^\epsilon$ on $M^*$. Recall that $L$ denotes the maximal order of a finite isotropy group of the $\crcl$-action on $M$. It is then immediate from Lemma \ref{44} that we have finally found a resolution of $E^*$.
\begin{proposition}\label{resolution of E}
 Let $\epsilon > 0$. Then there exists a complete metric $d^\epsilon$ on $M^*$ satisfying the following properties:
 \begin{itemize}
  \item[(i)] $d^\epsilon$ is induced by a smooth Riemannian metric $g^\epsilon$ on $M^* \setminus F$.
  \item[(ii)] $\curv d^\epsilon > Lc - \epsilon$ if $c \leq 0$ ($\curv d^\epsilon > 0$ if $c > 0$).
  \item[(iii)] $d_{GH}((M^*,d^\epsilon),(M^*,d)) < \epsilon$.
  \item[(iv)] For all $p \in F$ the space $\Sigma_p(M^*,d^\epsilon)$ is smooth, admits an effective isometric $\crcl$-action and has curvature bounded below by $1$. There exists $\tilde \rho > 0$ such that $B_{\tilde \rho}(p)$ is isometric to a neighborhood of a tip of the spherical suspension of $\Sigma_p(M^*,d^\epsilon)$.
 \end{itemize}

\end{proposition}
In the following section we resolve the remaining singularities of $F$.
\subsection{\texorpdfstring{Resolution of $F$}{Resolution of F}}\label{smooth fp}
For this section we fix a metric $d^\epsilon$ on $M^*$ as in Proposition \ref{resolution of E}. To prove Theorem \ref{thm2.3} it remains to prove the following
\begin{proposition}\label{resolution of F}
There exists a smooth metric $\tilde g^\epsilon$ on $M^*$ with $\curv \tilde g^\epsilon > Lc - \epsilon$ if $c \leq 0$ ($\curv \tilde g^\epsilon > 0$ if $c > 0$) and $d_{GH}((M^*,\tilde g^\epsilon),(M^*,g^\epsilon)) < \epsilon$.
\end{proposition}
Let $p \in M^*$ be a fixed point. First we show that $\Sigma_p(M^*,d^\epsilon)$ admits an isometric embedding into $\sphere^3$ with its standard metric.

\begin{lemma}\label{embedding}
 Let $g$ be a smooth metric on $\sphere^2$ with $\sec g \geq 1$ that admits an effective and isometric $\crcl$-action. Then $(\sphere^2,g)$ admits an isometric embedding into $\sphere^3$ equipped with the standard metric.
\end{lemma}
\begin{proof}
Due to the isometric $\crcl$-action it is easy to see that there exist exactly two fixed points of the action and $g$ is given by the formula
$$g(r,\theta) = dr^2 + R^2(r)d\theta^2,$$
where $(r,\theta) \in [0,d] \times [0,2\pi],$ and $d$ denotes the distance of the two fixed points and $R : [0,d] \to \RR$ is a smooth function with $R(0) = R(d) = 0$. Let $X(r,\theta)$ denote the Killing field of the action. Then $||X(r,\theta)|| = R(r)$ and $||\nabla_{\partial_r}X(r,\theta)|| = R'(r)$. Since $g$ is smooth, we have $R'(0) = 1$ and it follows from the Rauch comparison theorem that 
\begin{align}\label{shmunk}
 R(r) \leq \sin (r).
\end{align}
 and $d \leq \pi$ (as it is also clear from Bonnet-Myers theorem). For $d = \pi$ it was shown in \cite{cheng75} that $g$ is isometric to the standard metric of $\sphere^2$ and thus admits an isometric embedding into $\sphere^3$. Thus we may assume that $d < \pi$ and there exists some $r_0 \in ]0,d[$ with $K^g(r_0,\theta) > 1$. Again from the Rauch comparison theorem, it follows that 
 \begin{align}\label{R < 1}
  R(r) < 1
 \end{align}
 for all $r \in [0,d]$. Now let $U := \sphere^1 \times B_{\pi/2}$, where $B_{\pi/2}$ denotes the open ball of radius $\pi/2$ in $\RR^2$. Equip $\sphere^1$ with the standard metric $ds^2$ (such that it has perimeter $2\pi$) and let $(t,\theta)$ denote polar coordinates on $B_{\pi/2}$ with radial coordinate $t$ and angular coordinate $\theta$. Consider the metric
\begin{align}
 h = \cos(t)^2ds^2 + dt^2 + \sin^2(t)d\theta^2
\end{align}
on $U$. Then $(U,h)$ is isometric to an open subset of $\sphere^3$, so it suffices to embed $g$ isometrically into $(U,h)$.

The idea is to realize $g$ via a surface of revolution.  Given a smooth function $v : [0,d] \to \RR$, let
\begin{align}
 u : [0,d] \times [0,2\pi] & \rightarrow U\nonumber\\
(r,\theta) & \mapsto \begin{pmatrix} v(r) \\ \arcsin(R(r)) \\ \theta \end{pmatrix},
\end{align}
where we choose $\arcsin : [-1,1] \to [-\pi/2,\pi/2]$ and the first, second and third entry of the vector on the right hand side corresponds to the $s,t$ and $\theta$ coordinate respectively. Note that $\arcsin(R(r))$ is well defined by \eqref{shmunk}.
We calculate
\begin{align}
&u^*h\partial_r,\partial_r)(r,\theta)\nonumber \\
= &h(du_{(r,\theta)}\partial_r,du_{(r,\theta)}\partial_r) \nonumber \\
= &h(v'(r)\partial_s + (\arcsin(R(r))'\partial_t,v'(r)\partial_s + (\arcsin(R(r))'\partial_t)(u(r,\theta)) \nonumber \\
= &(v'(r))^2\cos^2(\arcsin(R(r))) + (\arcsin(R(r))')^2\nonumber\\
= &(v'(r))^2(1 - R^2(r)) + \frac{R'(r)^2}{1 - R^2(r)},\nonumber
\end{align}
\begin{align}
u^*h(\partial_\theta,\partial_\theta)(r,\theta)
= h(du_{(r,\theta)}\partial_\theta, du_{(r,\theta)}\partial_\theta)
= h(\partial_\theta,\partial_\theta)(u(r,\theta)) = R^2(r)
\end{align}
and
\begin{align}
u^*h(\partial_r,\partial_\theta) = 0.
\end{align}
It follows that 
$$u^*h(r,\theta) = \left((v'(r))^2(1 - R^2(r)) + \frac{R'(r)^2}{1 - R^2(r)}\right)^2dr^2 + R^2(r)d\theta^2.$$
Thus $u^*h = g$ if and only if
$$1 = v'^2(1 - R^2) + \frac{R'^2}{1 - R^2},$$
or equivalently
\begin{align}
v' = \sqrt{\frac{1 - R^2 - R'^2}{(1- R^2)^2}}.\label{Ode}
\end{align}
We claim that $R'^2 + R^2 \leq 1$: $R$ is concave, since by Lemma \ref{curvature} $$1 \leq K^g(r,\theta) = -R''(r)/R(r).$$
Hence the set $I_+ := \{r \in [0,d] \mid R(r)' \geq 0\}$ is a connected interval containing $\{0\}$. From $R''/R \leq -1$ we see $R'' + R \leq 0$. Thus
$$(R'(r)^2 + R(r)^2)' = 2R'(r)(R''(r) + R(r)) \leq 0$$
for $r \in I_+$. Consequently $R'(r)^2 + R(r)^2 \leq R'(0)^2 + R(0)^2 = 1$ for all $r \in I_+$. Considering the function $r \mapsto R(d - r)$ an analogous argument shows that $R'(r)^2 + R(r)^2 \leq 1$ for $r \in I_- := \{r \in [0,d] \mid R(r)' \leq 0\}$ proving the claim.

Thus there exists a unique solution $v(r)$ to \eqref{Ode} on $[0,d]$ with $v(0) = 0$. Using standard methods it is not hard to show that $u$ is indeed an embedding (it follows from \cite{docarmo-warner70} that it is in fact enough to show that $u$ is an immersion). 
\end{proof}
In particular, there exists an isometric embedding $ \Sigma_p(M^*,d^\epsilon) \subset \sphere^3$. It was shown in \cite{docarmo-warner70} that an isometric embedding $N^n \subset \sphere^{n +1}$ of a Riemannian manifold $N^n$ with $\sec N^n \geq 1$ is either totally geodesic or $N^n$ is contained in an open hemisphere and bounds a convex body in $\sphere^{n + 1}$. Thus an isometric embedding $\Sigma_p(M^*,d^\epsilon) \subset \sphere^3$ bounds a convex body in $\sphere^3$ and is contained in an open hemisphere, since it is clearly not totally geodesic (for example due to $\operatorname{diam}(\Sigma_p(M^*,d^\epsilon)) = \pi/2$).
\begin{lemma}
 Let $\Sigma_p(M^*,d^\epsilon) \subset \sphere^3$ be an isometric embedding. Then $S(\Sigma_p(M^*,d^\epsilon))$ is isometrically embedded into $S(\sphere^3) = \sphere^4$ in the obvious way and $S(\Sigma_p(M^*,d^\epsilon))$ bounds a convex body in $\sphere^4$.
\end{lemma}
Therefore, we may view $(B_{\tilde\rho}(p),g^\epsilon)$ as a subset $(B_{\tilde\rho}(p),g^\epsilon) \subset \sphere^4$ that is a part of the boundary of a convex body in $\sphere^4$. Inspired by \cite{docarmo-warner70}, to resolve the singularity of $(B_{\tilde\rho}(p),g^\epsilon)$ at $p$ we transfer the problem to $\RR^4$ via a so called Beltrami map. 
\begin{definition}
 Let $v \in \sphere^n \subset \RR^{n + 1}$. Denote by $T_v = T_v\sphere^n + v$ the plane tangent to $\sphere^n$ at $v$ and by $H_v$ the open hemisphere centered at $v$. Then the \textit{Beltrami map} 
 $$\beta^v : H_v \to T_v = \RR^n$$
 is defined via central projection: For $q \in H_v$ let $r_q$ denote the ray passing through $q$ and $0$. Then define $\beta^v(q)$ as the point of intersection of $r_q$ and $T_v$.
\end{definition}
$\beta^v$ is a geodesic map, that is, geodesics of $\sphere^n$ are setwise mapped to geodesics of $\RR^n$. Also (modulo natural identifications)
$$(d\beta^v)_v = \operatorname{Id} : T_v \to T_v.$$ Moreover, it was shown in \cite{docarmo-warner70} that hypersurfaces in $H_v$ of curvature $\geq 1$ are mapped to hypersurfaces of $\RR^n$ with curvature $\geq 0$, and vice versa. 
\\ \\
Now we can give the \textit{proof of Proposition \ref{resolution of F}}. Consider $\beta^p : H_p \subset \sphere^4 \to \RR^4$. Set
$$V := \beta^p(B_{\tilde \rho}(p)) \subset \RR^4.$$
Since geodesics of $H_p$ are mapped to straight lines in $\RR^4$ it follows that $V$ is an open neighborhood of the vertex $0$ of a convex cone $C \subset \RR^4$ (in fact, since $(d\beta^p)_p$ is the identity, it follows that $C = T_pB_{\tilde \rho}(p)$). Since $B_{\tilde \rho}(p)$ is smooth at all points but $p$, it follows that $C$ is smooth at all points but $0$. We may thus assume that $C$ is the graph of a convex function
$$f : \RR^3 \to \RR$$
which is smooth at all points except for $0$. To smooth the singularity of $f$ at $0$ we use a convolution (compare \cite{ghomi02}): For $n \in \NN$ let $\sigma_n : \RR^3 \to \RR^{\geq 0}$ be a smooth function with support in $B_{1/n}(0)$ and $\int_{\RR^3} \sigma_n dx = 1$. Set 
$$f_n(x) = f \ast \sigma_n(x) = \int_{\RR^3} f(x - y)\sigma_n(y)dy.$$
Then $f_n : \RR^3 \to \RR$ is convex for all $n$. Further, $f_n$ and all its derivatives of arbitrary order converge uniformly to $f$ and the corresponding derivative, respectively, on $A_r := \overline{B_r(0)} \setminus B_{r/2}(0)$ for all $r > 0$. We glue $f$ and $f_n$ as in \eqref{tilde varphi}: For $r > 0$ let $j_r : [0,\infty[ \to \RR$ be smooth with $j \equiv 1$ on $[0,r/2]$ and $j \equiv 0$ on $[r,\infty[$ and set
$$f_{r,n}(x) = j_r(|x|)f_n(x) + (1 - j_r(|x|))f(x).$$
Then for fixed $0 < r <1$ it is clear that $f_{r,n}$ is convex on $\RR^4 \setminus A_r$, $f_{r,n} = f$ on $\RR^3 \setminus B_r(0)$ and $f_{r,n} \xrightarrow{n \to \infty} f$ on $A_r$ uniformly, together with all its derivatives.

Without loss of generality we assume that $V = f(B_2(0))$, so $B_{\tilde \rho}(p) = (\beta^p)^{-1}(f(B_2(0)))$. Let 
$$U_n := (\beta^p)^{-1}(f_{1,n}(B_2(0)).$$
Then $U_n$ is a smooth hypersurface of $\sphere^4$ and for all $n$
\begin{align}\label{rand}
 U_n \cap B_{\tilde \rho}(p) \supset (\beta^p)^{-1}(f(B_2(0)\setminus B_1(0))).
\end{align}
Clearly $U_n = (\beta^p)^{-1}(f_{1,n}(B_2(0)\setminus A_1)) \cup (\beta^p)^{-1}(f_{1,n}(A_1))$.
Since $f_{1,n}$ is convex on $B_2(0) \setminus A_1$, it follows that $(\beta^p)^{-1}(f_{1,n}(B_2(0)\setminus A_1))$ equipped with the induced metric has curvature bounded below by $1$ for all $n \geq 1$. On the other hand $(\beta^p)^{-1} \circ f_{1,n}$ converges uniformly to $(\beta^p)^{-1} \circ f$ on $A_1$, as well as all of its derivatives. Since $(\beta^p)^{-1}(f(A_1))$ has curvature bounded below by $1$ (it is an open submanifold of $B_{\tilde \rho}(p)$), it follows that $(U_n,g_{\sphere^4})$ has curvature bounded below by $1/2$ for all sufficiently large $n$. Using \eqref{rand}, we can glue $(U_n,g_{\sphere^4})$ smoothly to $(M^* \setminus B_{\tilde \rho}(p),d^\epsilon)$ to obtain a resolution of the singularity at $p$.

We perform this construction at every point of $F$ and obtain a smooth metric $\tilde g^\epsilon$ as in Proposition \ref{resolution of F}. \hfill$\square$

\begin{remark} The resolution of $p$ in the above argument can be performed while preserving the given isometric action of $\crcl$ on $B_{\tilde \rho}(p)$. Observe that this is satisfied when the maps $f_{r,n}$ share the symmetries of $f$. For that it is enough to choose a function $\sigma$ for the convolution in the above proof which is invariant under $\mathsf O(3)$.

\end{remark}
\section{\texorpdfstring{A Ricci flow on $M^*$}{A Ricci flow on M*}}\label{ricci}
In this chapter we show how Theorem \ref{resolution thm} follows from Theorem \ref{umformulierung} using results on the Ricci flow by Simon and Hamilton. Let us recall the basic
\begin{definition}
Let $\mathcal I \subseteq \RR$ be connected and $M$ be a smooth manifold together with a smooth $1$-parameter family of Riemannian metrics $g(t)_{t \in \mathcal I}$. Then $g(t)$ is a \textit{solution of the Ricci flow equation} if
$$\frac{\partial}{\partial t}g_{ij} = -2\Ric_{ij}.$$
We also simply say that $g(t)$ is a Ricci flow on $M$.
\end{definition}

The Ricci flow was introduced by Hamilton in \cite{hamilton82}, where he proved that for a closed Riemannian manifold $(M,g)$ there always exists a solution $(M,g(t))$ to the Ricci flow equation on some time interval $[0,T[$, with $T > 0$ and $g(0) = g$.  In \cite{simon09} it was shown that there exists a Ricci flow on a possibly singular Riemannian $3$-manifold given that it may be approximated by smooth metrics in a controlled way. The following Theorem is part of Theorem 7.2 in \cite{simon09}.
\begin{theorem}\label{simon}
Let $(M_n,g^n_0)$ be a sequence of closed $3$-manifolds satisfying 
\begin{align*}
 \operatorname{diam}(M_n,g^n_0) \leq d_0,\\
 \sec_{g^n_0} \geq 1/n,\\
 \operatorname{vol}(M,g^n_0) \geq v_0 > 0.
\end{align*}
Then there exists an $S = S(v_0,d_0) > 0$ such that for all $n$ there exists a solution to the Ricci flow $g_n(t)$ with $t \in [0,S]$ and $g_n(0) = g^n_0$. After taking a subsequence, there exists a Hamilton limit solution $(M,g(t))_{t \in ]0,S[} = \lim_{n \to \infty}(M,g_n(t))_{t \in ]0,S[}$ (see \cite{hamilton95}) satisfying 
\begin{align*}
 \sec_{g(t)} \geq 0
\end{align*}
for all $t \in \ ]0,S[$ and $(M,g(t))$ is closed for all $t \in \ ]0,S[$. If $(M_n,g_0^n) \to (M_\infty,d_\infty)$ in Gromov-Hausdorff sense for $n \to \infty$, then $(M,g(t)) \to (M_{\infty},d_\infty)$ for $t \to 0$. 
\end{theorem}
From Theorem \ref{umformulierung} we obtain a sequence $(M^*,g_n^0)$ satisfying the conditions of the above theorem. Thus we obtain a Ricci flow $(M^*,g(t))$ with $t \in \ ]0,S[$ such that $(M^*,g(t))$ has nonnegative curvature for all $t$ and $(M^*,g(t)) \to (M^*,g)$ in Gromov-Hausdorff sense as $t \to 0$. Now Theorem \ref{resolution thm}, part (a), follows from classical results by Hamilton (see \cite{hamilton82}, \cite{hamilton86} or Theorem 6.64 in \cite{chow-peng-li06}).
\begin{theorem}\label{hamilton}
Let $(M^3, g(t))$ with $t \in [0,T[$ and $T > 0$ be a solution to the Ricci flow on a closed, simply connected $3$-manifold. If $(M,g(0))$ has nonnegative curvature, then $(M,g(t))$ has positive curvature for all $0 < t < T$ (and $M^3$ is diffeomorphic to $\sphere^3$).
\end{theorem}
In particular we obtain that $M^*$ is diffeomorphic to $\sphere^3$.
\begin{corollary}
$M^*$ is diffeomorphic to $\sphere^3$.
\end{corollary}


\section{\texorpdfstring{Proof of theorem \ref{resolution thm}, part $(b)$.}{Proof of theorem resolution, part b}}\label{final proof}
Let $M$ be as in Theorem \ref{resolution thm} with quotient space $(M^*,d)$ and $c$ be a closed curve contained in $F \cup E^* \subset M^*$. From part $(a)$ of Theorem $1$ it follows that $M^* \cong \sphere^3$ (compare the previous section). As in \cite{grove-wilking13} let $(M^*_2(c), \hat d)$ denote the canonical two fold branched cover 
$$(M^*_2(c),\hat d) \to (M^*,d)$$
branched along $c$. As discussed in \cite{grove-wilking13}, $M^*_2(c)$ is an Alexandrov space with nonnegative curvature as well. Let $\iota : M^*_2(c) \to M^*_2(c)$ denote the isometric involution inducing $M^*_2(c) \to M^* = M^*_2(c)/\iota$. If we view $c$ as a curve in $M^*_2(c)$ we denote it by $\hat c$. Then $\hat c$ equals the fixed point set of $\iota$. We denote the points of $M^*_2(c)$ projecting to fixed points by $\hat F$. A point $\hat c(t)$ is a nonregular point of $M^*_2(c)$ if and only if its isotropy group has order strictly bigger than $2$. In particular all points of $\hat F$ are nonregular points of $M^*_2(c)$. We denote by $\hat E^*$ the set of nonregular points of $M^*_2(c)$ that are not contained in $\hat F$ (note that this notation might be misleading, since $\hat E^*$ does not project to $E^*$, if $c$ contains points with isotropy group $\mathbb Z_2$).

With some minor modifications and using ideas from \cite{grove-wilking13} the proof of part $(a)$ of Theorem \ref{resolution thm} carries over $\iota$-equivariantly to $M^*_2(c)$ yielding the following theorem.
\begin{theorem}\label{branched cover}
 There exists a sequence of smooth, positively curved metrics $\hat h_n$ on $M^*_2(c)$, invariant under $\iota$, such that 
 $$(M^*_2(c),\hat h_n) \xrightarrow{n \to \infty} (M^*_2(c),\hat d)$$
 in Gromov-Hausdorff sense.
\end{theorem}
This theorem is equivalent to Theorem \ref{resolution thm}, part $(b)$: On the one hand, a sequence $h_n$ as in Theorem \ref{resolution thm}, $(b)$, is obtained from a sequence $\hat h_n$ as in the above theorem via the quotient metrics $\hat h_n /\iota$. On the other hand given a sequence $h_n$ on $M^*$ as in Theorem \ref{resolution thm}, $(b)$, there exists a unique smooth metric $\hat h_n$ on $M^*_2(c)$ such that $\hat h_n / \iota = h_n$. Then $\hat h_n$ has the desired properties. 
\\ \\
We sketch what modifications are necessary to proof Theorem \ref{branched cover}:

First, using the techniques of section \ref{resolution}, we may assume that $F \cup E^*$ equals the image of $c$ (in fact, this is the case; compare Theorem 2.5 in  \cite{grove-wilking13}, theorem 2.5), since all singularities outside of $c$ can be resolved without changing the metric in a neighborhood of $c$. Now, consider a sequence of metrics $d_n$ on $M^*$ given by Proposition \ref{symmetry}. These metrics induce metrics $\hat d_n$ on $M^*_2(c)$ with properties analogous to the properties $1 - 6$ of Proposition \ref{symmetry}. The notable differences are that for $p \in \hat F$ in this situation $\Sigma_pM^*_2(c)$ is not isometric to $\sphere^3/\crcl$, but to a  two fold branched cover of it, and similarly a neighborhood $U$ of an arc $\hat \gamma \in \hat E^*$ is not a good orbifold, but a two fold branched cover of one with branching locus $\gamma$.

Then we can resolve a singular arc $\hat \gamma \subset \hat E^*$ corresponding to an arc $\gamma$ with isotropy group $\ZZ_m$ analogously to sections \ref{resolving brho} - \ref{resolving E}. Due to the above noted differences we only need to replace the constant $m^{-1}$ used there by $2m^{-1}$ (coming from the angle of the normal cone to $\hat \gamma$, which is $2$ times the normal angle at $\gamma$). Since $\hat \gamma$ is nonregular, it follows that $m \geq 3$ and we have $0 < 2m^{-1} < 1$. Then the arguments go through virtually unchanged. Also note that the involution $\iota$ is induced by the $\crcl$-action near $\hat c$. Therefore we obtain a sequence of $\iota$-invariant metrics $\hat g_n$ that are smooth at all points not contained in $\hat F$ and satisfy properties analogous to Proposition \ref{resolution of E}. The remaining isolated singularities can be resolved as in section \ref{smooth fp} (note that the convolution can be performed invariantly under $\iota$, compare the remark to the proof of Proposition \ref{resolution of F}). Therefore, we obtain an approximation of $\hat d$ as in Proposition \ref{branched cover}, but only via almost nonnegatively curved metrics $\hat h_{n,0}$. Using Theorem \ref{simon}, we obtain a solution to the Ricci flow $\hat g(t)_{t \in ]0,T[}$ on $M^*_2(c)$ such that $\hat g(t)$ has nonnegative curvature and $(M^*_2(c),\hat g(t))$ converges to $(M^*_2(c),\hat d)$ in Gromov-Hausdorff sense for $t \to 0$. Moreover, $\hat g(t)$ is invariant under $\iota$ for all $t$, since the Ricci flow of the metrics $\hat h_{n,0}$ preserves their isometries. To finish the proof it remains to show that $M_2^*(c)$ is simply connected, so that we can apply Theorem \ref{hamilton}.

For that we argue as in \cite{grove-wilking13}: Since $F \cup E^*$ contains a closed curve it follows that there are at least two fixed points $p,q \in M^*$, which are contained in $c$. Then the spaces of directions $\Sigma_{p}M_2^*(c)$ and $\Sigma_{q}M_2^*(c)$ are smaller than $\sphere^2(1/2)$, the $2$-sphere of radius $1/2$ (this means that choosing some isometric $\crcl$-action on $\sphere^2(1/2)$, there exist distance decreasing, $\crcl$-equivariant homeomorphisms $\sphere^2(1/2) \to \Sigma_{\hat p}M_2^*(c)$ and $\sphere^2(1/2) \to \Sigma_{\hat q}M_2^*(c)$). For a general nonnegatively curved, $3$-dimensional Alexandrov space, there are at most $4$ points whose spaces of directions are smaller than $\sphere^2(1/2)$ (see \cite{grove-wilking13}, lemma 2.6). Therefore, considering the universal cover of $M_2^*(c)$ which is an Alexandrov space locally isometric to $M_2^*(c)$, it follows that $\pi_1(M_2^*(c))$ has order at most $2$. Assume it has order $2$. Considering a nonnegatively curved metric $\hat g(t)$ from above on $M^*_2(c)$ and applying Theorem \ref{hamilton} to the universal cover of $M_2^*(c)$ shows that $M^*_2(c)$ is diffeomorphic to $\RR P^3$ and the involution $\iota$ is linear. But then there must be a second fixed point component of $\iota$, a contradiction.  Hence $M_2^*(c)$ is simply connected.
\hfill$\square$
\\ \\
Since the proof also shows that $M_2^*(c)$ is simply connected we deduce from Theorem \ref{hamilton}:
\begin{corollary}
$M_2^*(c)$ is diffeomorphic to $\sphere^3$. 
\end{corollary}

\begin{corollary}
$c$ is unknotted.
\end{corollary}
To conclude chapter $2$ we emphasize that these corollaries are obtained without making use of the Poincaré conjecture. In particular, this shows that the equivariant classification of closed, nonnegatively curved, simply connected $4$-manifolds with $\crcl$-symmetry in \cite{grove-wilking13} can be obtained independently of it.

\chapter{Nonnegatively curved fixed point homogeneous manifolds}\label{fph}
We begin recalling the basic definition.
\begin{definition}
 A Riemannian manifold $(M,g)$ is called \textit{fixed point homogeneous} if it admits an isometric action by a compact Lie group $\G$ such that the fixed point set $\Fix \G$ is nonempty and the induced action of $\G$ on a normal sphere to some fixed point component is transitive (equivalently $\dim M^* = \dim \Fix \G + 1).$
\end{definition}
Fixed point homogeneous $\crcl$-actions on positively curved manifolds were first considered by Grove-Searle in \cite{grove-searle94}. They give a classification of closed, positively curved manifolds whose isometry group has maximal possible rank and show that such manifolds admit fixed point homogeneous $\crcl$-actions. Later the same authors gave a general classification of closed, positively curved, fixed point homogeneous manifolds in \cite{grove-searle97}. 
As was noted in the introduction, a classification of closed, nonnegatively curved, fixed point homogeneous manifolds is clearly out of reach; given a closed, nonnegatively curved manifold $N$, the product manifold $N \times \sphere^2$, equipped with the product metric, admits a fixed point homogeneous $\crcl$-action. However, many of the techniques used in \cite{grove-searle94} and \cite{grove-searle97} are applicable to the nonnegatively curved case. We use these ideas to prove our main theorem, which is a generalization of the following structure Theorem used by Grove-Searle in \cite{grove-searle97} to obtain the  mentioned classification.
\begin{theorem} Let $M$ be a closed, positively curved, fixed point homogeneous Riemannian $\G$-manifold and $F$ be a fixed point component of maximal dimension. Then there exists a unique orbit $\G(\hat p)$ of maximal distance to $F$ and all orbits in $M \setminus (F \cup \G(\hat p))$ are principal. There is a $\G$-equivariant decomposition 
$$M = DF \cup_E D\G(\hat p),$$
where $DF$ and $D\G(\hat p)$ are the unit disk bundles of $F$ and $\G(\hat p)$, respectively, with common boundary $E$, viewed as tubular neighborhoods of $F$ and $\G(\hat p)$ in $M$. Moreover, the diffeomorphism type of $M$ is determined by the slice representation at $\hat p$.
\end{theorem}
We sketch some ideas of the proof of this Theorem to motivate our arguments: The quotient space $M^*$ is a compact, positively curved Alexandrov space. Moreover, $F \subset M^*$ is clearly contained in the boundary $\partial M^*$, since $\G$ acts transitively on every normal sphere to $F$. In fact $\partial M = F$. In \cite{perelman91} it was shown that the distance function to the boundary $d_{F}$ is strictly concave. Hence there exists a unique point $p \in M^*$ of maximal distance to the boundary and, therefore, a unique orbit $\G (\hat p)$ of maximal distance to $F$. Also from the concavity of $d_F$, it follows that every point $q \in M^* \setminus (F \cup \{p\})$ is a non-critical point for $d_F$, and that all points in $M^* \setminus (F \cup \{p\})$ are principal (for a discussion of critical point theory for distance functions see \cite{grove90}). The regularity of $d_F$ lifts to $M \setminus (F \cup \G(\hat p))$. Therefore, there exists a gradient-like vector field $X$ with respect to $F$ on $M \setminus (F \cup \G(\hat p)$, which is radial near $F$ and $\G(\hat p)$. Via the integral curves of $X$, a diffeomorphism can be constructed mapping the unit normal disk bundle $DF$ of $F$ to $M \setminus B_\epsilon(\G(\hat p))$. This yields the decomposition $M = DF \cup_\partial D\G(\hat p)$.

Our strategy is mostly the one of this argument. A major difference is that the distance function to a boundary component of $M^*$, in case of nonnegative curvature, is only concave, rather than strictly concave. Therefore, the set of maximal distance to it, say $C$, may not be a single point.  $C$ is a totally convex subset of $M^*$ and, as in the case of positive curvature, there exists a gradient-like vector field on $M^* \setminus (F \cup C)$ with respect to $F$, and $C$ contains all nonregular points of $M^* \setminus F$ (see Lemma \ref{regular points}).  Let $\hat C = \pi^{-1}(C)$. This shows, for example, that $(M \setminus \hat C)$ is diffeomorphic to $NF$, the normal bundle of $F$. To obtain stronger results, one needs to obtain information on the regularity of $\hat C$. 

We will show that $\hat C$ is a smooth submanifold of $M$, possibly with nonsmooth boundary, and the boundary of $\hat C$ consists of principal orbits of the action (Proposition \ref{smooth lift}). This enables us to perform an analogue of the soul construction of Cheeger-Gromoll for $\hat C$ and we obtain the main theorem of this chapter.
\begin{theorem}\label{ddb2}
Assume that $\G$ acts fixed point homogeneous on a complete nonnegatively curved Riemannian manifold $M$. Let $F$ be a maximal fixed point component of the action. Then there exists a smooth submanifold $N$ of $M$, without boundary, such that $M$ is diffeomorphic to the normal disk bundles $D(F)$ and $D(N)$ of $F$ and $N$ glued together along their boundaries;
\begin{align}\label{doubled}
 M \cong D(F) \cup_{\partial} D(N).
\end{align}
Further, $N$ is $\G$-invariant and contains all singularities of $M$ up to $F$.
\end{theorem}
A manifold obtained as in \eqref{doubled} via gluing two disk bundles along its boundaries is called a \textit{double disk bundle}.
\section{Preliminaries}
By an Alexandrov space we mean a complete length space with curvature bounded below and of finite dimension. If spaces are not complete we mention this explicitly (for example, a locally complete Alexandrov space is not assumed to be complete). We assume a basic background knowledge on Alexandrov spaces, for example as it is discussed in \cite{burago-burago-ivanov01}.

In section \ref{a with b} we recall two results from \cite{perelman91}. In section \ref{convex subsets} we obtain some technical results on convex subsets of Alexandrov spaces, in particular of quotient spaces of Riemannian manifolds.
\subsection{Alexandrov spaces with nonempty boundary}\label{a with b}
Analogously to the case of Riemannian manifolds, a function $f : A \to \RR$ on an Alexandrov space  is said to be concave if $f \circ \gamma$ is concave for every geodesic $\gamma$. A basic tool for our arguments is the following result from \cite{perelman91}.

\begin{lemma}\label{concave distance}
Let $A$ be an Alexandrov space with $\on{curv} \geq 0$ ($\curv \geq k > 0$) and nonempty boundary. Then the distance function to the boundary, or any component of it, is (strictly) concave.
\end{lemma}

In \cite{perelman91} only the distance function to the whole boundary is considered, the case for a component of the boundary can be obtained via the same arguments. Recall that the double $\tilde A$ of an Alexandrov space $A$ with nonempty boundary is the disjoint union of two copies of $A$ with its boundary points identified via the identity map; 
$$\tilde A := A \cup_\partial A.$$
The metric of $\tilde A$ is the unique intrinsic metric such that the two obvious inclusions of $A$ are isometric embeddings. Analogously one can glue single components of the boundary of $A$.
\begin{theorem}
If $A$ is an Alexandrov space with curvature bounded below by $c$, then $\tilde A$ is an Alexandrov space with curvature bounded below by $c$.
\end{theorem}
\subsection{Convex subsets of Alexandrov spaces and quotient spaces}\label{convex subsets}
The convexity properties of the set of maximal distance to the boundary of a quotient space $M^*$ of a fixed point homogeneous manifold are the basis for our arguments. In this section we collect some technical results on convex subsets of quotient spaces and more generally Alexandrov spaces of nonnegative curvature.
\begin{definition}
 Let $A$ be an Alexandrov space and $C \subset A$. Then $C$ is \textit{locally convex} if for all $p \in C$ there exists an open neighborhood $U \subset A$ of $p$ such that for all $q,r \in U \cap C$ there exists a minimal geodesic of $A$ from $q$ to $r$ which is contained in $C$. $C$ is \textit{convex} if for all $p, q \in C$ there exists a minimal geodesic $\gamma$ of $A$ from $p$ to $q$ that is contained in $C$. $C$ is \textit{totally convex} if for all $p,q \in C$ every geodesic $\gamma$ of $A$ connecting $p$ and $q$ is contained in $C$. 
\end{definition}
For example, every open subset of an Alexandrov space is locally convex. A closed subset $C$ of an Alexandrov space $A$ is convex if and only if $C$ equipped with the induced metric is intrinsic. Of course a totally convex set is convex and a convex set is locally convex.
\begin{lemma}
 A closed and connected locally convex (or convex or totally convex) subset $C$ of an Alexandrov space $A$ equipped with the induced intrinsic metric is an Alexandrov space with the same lower curvature bound. 
\end{lemma}
\begin{proof} Recall that for the induced intrinsic metric
$d^C(p,q)$ is given by the infimum over the length of all continuous curves $c$ in $C$ connecting $p$ and $q$. Since $C$ is locally convex and connected, it easily follows that $C$ is path connected. Since $C$ is closed, there exists a curve of minimal length between any two points $p, q \in C$ (a geodesic). Therefore, $(C,d^C)$ is a complete length space. Now let $p \in C$. By local convexity there exists an open neighborhood $U$ of $p$ in $C$ such that for all $q,r \in U$ a minimal geodesic between $q$ and $r$ is also a minimal geodesic considered as a curve in $A$. Hence, every geodesic triangle $\Delta \subset U$ of $C$ also defines a geodesic triangle of $A$ and the distances of points on $\Delta$, measured with respect to $A$ or $C$, coincide. Thus, the definition of being an Alexandrov space with curvature bounded below by $k$ carries over from $A$ to $C$.
\end{proof}
For a locally convex (not necessarily closed) subset $C \subset A$ and $p \in C$ the space of directions $\Sigma_pC$ and the tangent cone $T_pC$ are defined analogously to the case of Alexandrov spaces, as the completion of the set of directions of unit speed geodesics emanating from $p$, equipped with the angular metric. Clearly we can view $T_pC$ as a subset of $T_pA$.

For a metric space $(X,d)$, we call a subset $U \subset X$ \textit{locally closed} if for all $x \in U$ there exists $\epsilon > 0$ such that the $\epsilon$-ball of $p$ in $(U,d)$ is closed in the $\epsilon$-ball of $p$ in $X$. Then we have
\begin{lemma}\label{T_pC is convex}
 Let $A$ be an Alexandrov space and $C \subset A$ be locally closed and locally convex. Then $T_pC$ is convex in $T_pA$ for all $p \in C$.
\end{lemma}
\begin{proof}
Without loss of generality we may assume that $C$ is closed and convex. Recall that 
$$(T_pA,0) = \lim_{\lambda \to \infty}(\lambda A,p),$$
with convergence in the pointed Gromov-Hausdorff sense, and $\lambda A$ denotes the metric space $(A,\lambda d)$. Let $\gamma : [0,1] \to T_pC$ be a minimal geodesic in $T_pC$ with $\gamma(0) = v$. It is enough to show that $\gamma$ is also a minimal geodesic of $T_pA$: For all $1 > \epsilon > 0$ we have that $\gamma_{[0,1-\epsilon]}$ is the unique minimal geodesic of $T_pC$ from $v$ to $\gamma(1 - \epsilon)$. Let $p_\lambda, q_\lambda^\epsilon \in \lambda C \subset \lambda A$ such that $p_\lambda \to v$ and $q_\lambda^\epsilon \to \gamma(1 - \epsilon)$ for $\lambda \to \infty$. Let $\gamma_\lambda^\epsilon$ be a minimal geodesic of $C$ from $p_\lambda$ to $q_\lambda^\epsilon$. Possibly after taking a subsequence, $\gamma_\lambda^\epsilon$ converges to a minimal geodesic of $T_pC \subset T_pA$ from $v$ to $\gamma(1 - \epsilon)$. By uniqueness, $\gamma_\lambda^\epsilon$ converges to $\gamma_{[0,1-\epsilon]}$ for all $\epsilon > 0$. Since $\lambda C$ is convex in $\lambda A$, $\gamma_\lambda^\epsilon$, we can choose $\gamma_\lambda^\epsilon$ such that it is also a minimal geodesic of $\lambda A$. Thus the limit $\gamma_{[0,1-\epsilon]}$ is a minimal geodesic of $T_pA$ for all $\epsilon > 0$. Consequently, $\gamma$ is a minimal geodesic of $T_pA$.
\end{proof}
\begin{remark}
I do not know whether the tangent cone to a totally convex subset of an Alexandrov space is also totally convex. If this is true, some of the upcoming arguments can probably be simplified.
\end{remark}
For the following lemma we need to introduce some notation: Let $C$ be a locally closed and locally convex subset of an Alexandrov space $A$. We denote by $\partial C$ the \textit{intrinsic boundary} of $C$, that is the set of points whose spaces of directions have nonempty boundary (considered as an Alexandrov space). $\partial^t C$ denotes the \textit{topological boundary} of $C$ in $A$. By an \textit{interior point} of $C$ we always mean a point of the topological interior $\mathring C = C \setminus \partial^t C$. If $\dim C < \dim A$ it is not hard to show that $\partial^t C = C$. In the following lemma we consider the case $\dim C = \dim A$.
\begin{lemma}\label{boundary}
 Let $A$ be a locally complete Alexandrov space and $C \subseteq A$ be locally convex and locally closed with $\dim A = \dim C$. Then
$$\partial C = (\partial^t C \cup \partial A) \cap C.$$
Therefore, if $C$ is closed we have
$$\partial C = \partial^tC \cup (\partial A \cap C).$$ 
\end{lemma}
\begin{proof}
It is of course enough to prove the first statement. We argue by induction on $n = \dim A \geq 1$. In the case $n = 1$ the claim is easy to prove. So let $n \geq 2$ and the claim hold for smaller dimensions than $n$.

First we show that $\partial C \subseteq (\partial^t C \cup \partial A) \cap C$: Let $x \in \partial C$. If $x \in \partial^t C$ we are done. So assume that $x \notin \partial^t C$. Then $x$ is contained in $\mathring C$, the topological interior of $C$, and it is clear that $\Sigma_x C = \Sigma_xA$. In particular, $\Sigma_xA$ has nonempty boundary. Hence $x \in C \cap \partial A$.

Now we show that $\partial C \supseteq (\partial^t C \cup \partial A) \cap C$: Let $x \in (\partial^t C \cup \partial A) \cap C$. First assume $x \in \partial^t C \cap C$. Consider a sequence of points $x_n \in A \setminus C$ with $x_n \to x$. Since $C$ is locally closed, there exists a minimal geodesic $\gamma_n : [0,1] \to A$ for all sufficiently large $n$, with $\gamma_n(0) \in C$, $\gamma_n(1) = x_n$ and $\mathcal L(\gamma_n) = d(x_n,C)$. Then $\dot \gamma_n(0)$ forms  an angle $\geq \pi/2$ to every point of $\Sigma_{\gamma_n(0)}C \subset \Sigma_{\gamma_n(0)}A$. In particular $\Sigma_{\gamma_n(0)}C \subsetneq \Sigma_{\gamma_n(0)}A$ and it follows that $\Sigma_ {\gamma_n(0)}C$ has nonempty topological boundary in $\Sigma_{\gamma_n(0)}A$. Since $\dim \Sigma_ {\gamma_n(0)}C = \dim \Sigma_{\gamma_n(0)}A$, it follows from the induction hypothesis that $\Sigma_{\gamma_n(0)}C$ has nonempty boundary as an Alexandrov space. Hence $\gamma_n(0) \in \partial C$. Also clearly $\gamma_n(0) \to x$. So it follows that $x \in \partial C$, since $\partial C$ is closed in $C$. Now assume $x \in \partial A \cap C$. We argue by contradiction: Assume that $x \notin \partial C$. Then $\partial \Sigma_xC$ is empty. From the induction hypothesis we find that $\Sigma_xC \subseteq \Sigma_xA$ has empty topological boundary, i.e. $\Sigma_x C = \Sigma_xA$. Since $\partial \Sigma_xC$ is empty, this contradicts the fact that $x \in \partial A$.
\end{proof}
\begin{corollary}\label{ball}
 Let $A$ be a locally complete Alexandrov space with empty boundary and $C \subseteq A$ be locally convex and locally closed.  Then $p \in C$ is an interior point of $C$ in $A$ if and only if $T_pC = T_pA$. 
\end{corollary}
If $A$ has boundary the conclusion of this corollary is wrong; consider for example $A = \{(x,y) \in \RR^2 \mid y \geq 0\}$ and its convex subset $C$ a closed disk tangent to the boundary.
\\ \\
Now we discuss some results on convex subsets of quotient spaces.
\begin{lemma}\label{type}
 Let $\G$ be a compact Lie group acting isometrically on a compact Riemannian manifold $M$ with quotient space $M^*$ and let $C \subset M^*$ be convex. Then there exists a maximal orbit type in $C$ and the set of points of maximal type is open and dense in $C$.
\end{lemma}
\begin{proof}
 Let $\gamma$ be a geodesic in $C$. Since $\gamma$ is also a geodesic in $M^*$, it follows from the slice theorem that the orbit type is constant along the interior of $\gamma$. Therefore it suffices to show that there exists $p \in C$ such that the orbit type is constant on an open neighborhood of $p$ in $C$, since, by convexity, any point $q \in C$ is connected to $p$ via a geodesic of $M^*$ lying in $C$. To see this, let $q \in C$ be arbitrary. Again by the slice theorem, there exists an open neighborhood $U$ of $q$ in $C$ such that all points in $U$ have type bigger or equal to $q$. In the case there is no point in $U$ of bigger type we are done. If there is a point $q_2 \in U$ of bigger type than $q$ let $U_2$ be a neighborhood of $q_2$ in $C$ such that all points in $U_2$ have type bigger or equal to $q_2$. Iterating this process, after a finite number of steps we obtain a point $q_k$ with an open neighborhood $U_k$ in $C$ such that all points of $U_k$ have constant type, since there exist only finitely many orbit types.
\end{proof}
\begin{definition}
For $p \in M$ set
\begin{align*}
\nu_{p} : N_{p}\G(p) &\to T_{\pi(p)}M^*\\
v &\mapsto  \frac d {dt}_{|t =  0}\pi (c_{v}(t)).
\end{align*}
\end{definition}
Here we denote by $c_{v}$ the geodesic of $M$ with initial conditions $c_{v}(0) = p$ and $\dot c_{v} (0) = v$. Identifying $T_{\pi(p)}M^* = N_{p}\G(p)/\G_{p}$ this map coincides with the quotient map $N_{\hat p}\G(\hat p) \to N_{\hat p}\G(\hat p)/\G_{\hat p}$.

\begin{lemma}\label{submanifold}
Let $\G$ be a compact Lie group acting isometrically on a Riemannian manifold $(M,g)$ with quotient space $M^*$. Let $C \subset M^*$ be locally closed and convex. Then $\pi^{-1}(C)$ is a smooth submanifold of $M$ without boundary, if and only if for all $p \in M$ with $\pi(p) \in C$ the space $\nu_p^{-1}(T_{\pi(p)}C) \subset N_p\G(p)$ is a linear subspace.
\end{lemma}
\begin{proof}
 The only if part is easy, so we only prove the if part. For $p \in M$ with $\pi(p) \in C$ let $N_p := \nu_p^{-1}(T_{\pi(p)}C)$ and $V_p := T_p\G(p) \oplus N_p$. First we show that $\dim V_p$ does not depend on $p$: Let $d$ denote the dimension  of a principal isotropy group of the action of $\G$ on $\pi^{-1}(C)$, which is well defined by Lemma \ref{type}. Then
 \begin{align*}
  \dim V_p &= \dim T_p(\G(p)) + \dim N_p\\
  &= (\dim \G - \dim \G_p) + \dim T_{\pi(p)}C + (\dim \G_p - d)\\
  &= \dim \G + \dim C - d.
 \end{align*}
So $k := \dim V_p$ does not depend on $p$. Now fix $p \in \pi^{-1}(C)$ and let $\epsilon > 0$ such that the normal exponential map to $\G(p)$ is injective on vectors of length less than $\epsilon$.  Set $U_\epsilon := \G(\exp(B_\epsilon(0_p) \cap N_p))$. Then $U_\epsilon$ is a smooth submanifold of $M$ without boundary (since $N_p$ is linear and $\G_p$-invariant). Because our arguments are local, we may for simplicity assume that $C \subset B_\epsilon(\pi(p))$. From convexity of $C$ it follows that 
$$C \subseteq \exp (B_\epsilon(0_{\pi(p)}) \cap T_{\pi(p)}C).$$
Consequently, 
$$\pi^{-1}(C) \subseteq \G(\exp(B_\epsilon(0_p) \cap N_p)) = U_\epsilon.$$
Consider the Riemannian manifold $(U_\epsilon,g)$. Then $\dim U_\epsilon = k$ and $U_\epsilon$ is $\G$-invariant with quotient space $U_\epsilon^* \subset M^*$.  Further, $C \subseteq U^*_\epsilon$ is convex and locally closed as well. We claim that $\pi(p)$ is an interior point of $C$ in $ U^*_\epsilon$: Assume the contrary. Since $C$ is locally closed, there exists $x \in U^*_\epsilon \setminus C$ close to $\pi(p)$ and a geodesic $\gamma : [0,1] \to U_\epsilon^*$ with $\gamma(0) \in C$, $\gamma(1) = x$ and $\mathcal L(\gamma) = \min \{d^{U_\epsilon^*}(x,a) \mid a \in C\}$. Clearly $\dot \gamma(0) \notin T_{\gamma(0)}C$ and $\dot \gamma(0) \in T_{\gamma(0)}U_\epsilon^*$, in particular $T_{\gamma(0)}C \subsetneq T_{\gamma(0)}U^*_\epsilon$. Let $q \in U_\epsilon$ with $\pi(q) = \gamma(0)$. Then $V_q \subsetneq T_qU_\epsilon$.  On the other hand $V_q$ is linear and we have shown that $\dim V_q = k = \dim T_qU_\epsilon$, a contradiction.

Therefore, $p$ is an interior point of $\pi^{-1}(C)$ in $U_\epsilon$ and hence a neighborhood of $p$ in $\pi^{-1}(C)$ is a smooth submanifold of $U_\epsilon$, and thus also of $M$.
\end{proof}
\begin{definition}Let $\G$ be a compact Lie group acting isometrically on $M$. A \textit{horizontal geodesic} $\gamma$ of $M$ is a geodesic that is perpendicular to all orbits it meets. A \textit{horizontal geodesic} $\gamma$ of $M^*$ is a curve that is the projection of a horizontal geodesic of $\gamma$. A subset $C \subset M^*$ is \textit{horizontally convex} if for all $p, q \in C$ and all horizontal geodesics $\gamma : [0,1] \to M^*$ with $\gamma(0) = p$ and $\gamma(1) = q$ we have $\gamma \subset C$.
\end{definition}
Note that a horizontal geodesic of $M^*$ is not a geodesic of $M^*$ in general. A nice property is that a horizontal geodesic $c : [0,1] \to M^*$ can be extended to a horizontal geodesic $c : [0,\infty[ \to M^*$ if $M$ is complete. This follows, since a horizontal lift of $c$ can be extended infinitely.
\\ \\
For the rest of this section we fix an euclidean vector space $V$ together with a compact Lie Group $\G$ acting orthogonally on it. Let $B$ denote the open unit ball of $V$. Note that for $\lambda \geq 0$ the scaling map $v \mapsto \lambda v$ descends naturally to $V^*$. By $0 \in V^*$ we denote the projection of $0 \in V$.
\begin{lemma}
 Let $\gamma$ be a (horizontal) geodesic of $V^*$. Then $\lambda \gamma$ is a (horizontal) geodesic as well, for all $\lambda \geq 0$  .
\end{lemma}
\begin{proof}
This follows from the analogous statement for $V$, which is true, since the action is orthogonal.
\end{proof}
\begin{lemma}\label{unique direction}
 Let $C \subseteq V^*$ be closed and convex with $\partial V^*  \subset C$ and $0 \in \mathring C$. Then for all $v \in \Sigma_0C = \Sigma_0V^*$ there either exists a unique $\lambda_v > 0$ with $\lambda_v v \in \partial^t C$, or $\lambda v \in \mathring C$ for all $\lambda \geq 0$. In the first case $\lambda v \notin C$ for all $\lambda > \lambda_v$.
\end{lemma}
\begin{proof}
 We may assume that $\partial^t C \neq \emptyset$, since otherwise the statement is trivial. First we assume that $\partial V^* = \emptyset$. Then $\partial^t C = \partial C$. Set $\lambda_{min} := d(0,\partial C)$. Then $0 < \lambda_{min} < \infty$. Let $v \in \Sigma_0V^*$. Since $0 \in \mathring C$, by convexity of $C$ and closedness of $\partial C$ there exists a unique $\lambda_{min} \leq \lambda_v \leq \infty$ such that $\lambda_v v \in \partial C$ and $t v \in \mathring C$ for all $0 \leq t < \lambda_v$. Assume there exists $\mu > \lambda_v$ such that $\mu v \in  C$. Let $\gamma(t) = t\mu v$, so $\gamma : [0,1] \to V^*$ is a unique minimal geodesic with $\gamma(0) \in C$ and $\gamma(1) \in C$ and thus $\gamma$ is contained in $C$ by convexity. Again by convexity of $C$, the distance function to $\partial C$ restricted to $\gamma$ is concave. Further, $d(\partial C,\gamma(t))$ attains its minimum $0$ at $t = \lambda_v/\mu$. Since $0 < \lambda_v/\mu < 1$, and $\gamma(t) \in C$ for $t \in [0,1]$, it follows that $d(\gamma(t),\partial C) \equiv 0$, i.e. $\gamma (t) \in \partial C$ for all $t \in [0,1]$, contradicting the fact that $0 \in \mathring C$.

Now let $\partial V^* \neq \emptyset$. Since $\partial V^* \subset C$, it follows that $\tilde C \subset \tilde V^*$ is convex, where $\tilde V^*$ denotes the double of $V^*$ and $\tilde C$ denotes the preimage of $C$ under the projection $\tilde V^* \to V^*$. Also $\tilde 0$ is contained in the interior of $\tilde C$. Even though $\tilde V^*$ is not a quotient space an argument analogous to the first case still applies, establishing the claim for $\tilde C \subset \tilde V^*$ and hence also for $C \subset V^*$. 
\end{proof}
\begin{remark} The lemma holds without the assumption that $\partial V^* \subset C$. Then the proof is a bit more complicated.
\end{remark}
Of course the analogous lemma holds for $B^*$:
\begin{lemma}\label{unique direction ball}
 Let $C \subset B^*$ be closed and convex with $\partial B^* \subseteq C$ and $0 \in \mathring C$. Then for all $v \in  \Sigma_0C = \Sigma_0B^*$ there either exists a unique $0 < \lambda_v$ such that $\lambda_v v \in \partial^t C$, and $\lambda v \notin C$ for all $\lambda_v < \lambda < 1$, or $\lambda v \in \mathring C$ for all $0 < \lambda < 1$.
\end{lemma}

\begin{definition}
Let $C$ be a locally convex subset of an Alexandrov space $A$ and $p \in C$. Then the \textit{normal cone} to $C$ at $p$ is defined as
$$N_pC := \{v \in T_pA \vert \measuredangle (v,w) \geq \pi/2 \text{ for all } w \in \Sigma_pC\} \cup \{0_p\}.$$
\end{definition}
The main technical lemma we need is the following:

\begin{lemma}\label{all or normal}
For all $n \geq 1$ let $A_n \subset B^*$ be horizontally convex and closed with $0 \in \mathring A_n$ and $\partial B^* \subset A_n$. Let $A_\infty = \bigcap_{n \geq 1}A_n$ and $A_\infty \neq \{0\}$. Then the following are equivalent:
\begin{itemize}
 \item[(i)] $N_0A_\infty = \{0\}$.
 \item[(ii)] $T_0A_\infty = V^*$.
 \item[(iii)] $0 \in \mathring A_\infty$.
\end{itemize}
\end{lemma}

\begin{proof} We show that the complementary statements are equivalent, that is
$$(1)N_0A_\infty \neq \{0\} \Leftrightarrow (2)\ T_0A_\infty \neq V^* \Leftrightarrow (3)\ 0 \notin \mathring A_\infty.$$
Clearly $(1) \Rightarrow (2)$ and $(1) \Rightarrow (3)$, so it suffices to show $(3) \Rightarrow (1)$ and $(2) \Rightarrow (1)$: Assume that either $T_0A_\infty \neq V^*$ or $0 \notin \mathring A_\infty$. Without loss of generality we may further assume that $B^* \setminus A_n \neq \emptyset$ for all $n$ and $A_{n + 1} \subseteq A_n$ for all $n$, since otherwise we can consider the sequence $\tilde A_n = \bigcap_{k \leq n}A_k$. 

For all $n$ let $d_n := d(0,\partial^tA_n)$. Then $d_n > 0$ for all $n$, since $0 \in \mathring A_n$ and there exists $v_n \in \Sigma_0B^*$ such that $d_nv_n \in \partial^tA_n$. Possibly after taking a subsequence we may assume that $v_n \xrightarrow{n \to \infty} v_0$ for some $v_0 \in \Sigma_0B^*$. We claim that $v_0 \in N_0A_\infty$. For this we argue by contradiction: Assume $v_0 \notin N_0A_\infty$. Then there exist $u_0 \in \Sigma_0A_\infty, \epsilon > 0$ and $\delta > 0$ such that $2\epsilon u_0 \in A_\infty$ and $\measuredangle(u_0,v_0) < \pi/2 - 2\delta$. It follows that $t\epsilon u_0 \in A_n$ for all $n$ and for all $0 \leq t \leq 2$ and for all sufficiently large $n$ we have 
\begin{align}\label{small angle}
 \measuredangle (u_0,v_n) < \pi/2 - \delta.
\end{align}
 Note that each of the assumptions $0 \notin \mathring A_\infty$ and $T_0A_\infty \neq V^*$ implies that $d_n \xrightarrow{n \to \infty} 0$. Let $c_n: [0,1] \to B^*$ be a minimal geodesic with $c_n(0) = \epsilon u_0$ and $c_n(1) = d_nv_n$. Further, let $a$ be the segment from $0$ to $\epsilon u_n$ and $b_n$ be the segment from $0$ to $d_nv_n$. By horizontal convexity the triangle $\Delta_n$ formed by $a, b_n$ and $c_n$ is contained in $A_n$ and, since $V^*$ is a cone, the angles of $\Delta_n$ add up to $\pi$ for all $n$. Moreover, since $d_n \to 0$, the angle of $a$ and $c_n$ goes to $0$. From \eqref{small angle} we conclude that $\alpha_n > \pi/2$ for all $n$ sufficiently large, where $\alpha_n$ denotes the angle of $c_n$ and $b_n$. Consequently
\begin{align}\label{diff c}
 \frac{d}{dt}_{|t = 1} d(0,c_n(t)) < 0
\end{align}
 for all $n$ sufficiently large. We can extend $c_n$ to a horizontal geodesic $c_n : [0,1 + s_n] \to B^*$ for some positive time $s_n$. Since $d(0,c(1)) = d(0,d_nv_n) = d_n = d(0,\partial^tA_n)$, using \eqref{diff c}, we can assume, maybe after choosing a smaller $s_n > 0$, that $c_n(1 + s_n) \in \mathring A_n$. Consequently there exists $\lambda > 1$ such that $\lambda \epsilon u_0 \in A_n$ and $\lambda c_n(1 + s_n) \in A_n$. Since $A_n$ is horizontally convex and $\lambda c_n$ is a horizontal geodesic, it follows that $\lambda c_n \subset A_n$. In particular, $\lambda c_n(1) = \lambda d_nv_n \in A_n$, in contradiction to Proposition \ref{unique direction ball}.
\end{proof}
\begin{remark}Let $A \subset V^*$ be horizontally convex with $0 \in A$ and $\partial V^* \subset A$. Then from Corollary \ref{ball} applied to the double $\tilde V^*$ of $V^*$ it follows that $0$ is an interior point of $A$ if and only if $T_0A = V^*$.  Also clearly $T_0A \neq V^*$ if $N_0A \neq \{0\}$. So the condition in the lemma that $A = A_\infty$ can be approximated by horizontally convex sets of maximal dimensions is only needed to prove that $T_0A = V^*$ if $N_0A = \{0\}$. I do not know whether this condition is necessary.
\end{remark}
\section{The decomposition theorem and further results}
\subsection{The Proof of Theorem \ref{ddb2}}
After this preparations we are in the position to poof Theorem \ref{ddb2}. It will be useful later to state more general conditions on $M$ than the ones of Theorem \ref{ddb2}:

For this section, if not stated otherwise, we fix a compact Riemannian manifold $(M,g)$, possibly with nonsmooth boundary $\partial M$, equipped with an isometric action by a compact Lie group $\G$ such that the quotient space $M^*$  with the induced quotient metric $d$ is an Alexandrov space with $\curv  M^* \geq 0$ and nonempty boundary. Further, one of the following conditions holds:
\begin{itemize}
\item[(a)] $\partial M$ is empty and a component of $\partial M^*$ is a fixed point component of the action.
\item[(b)] $\partial M$ is nonempty, connected and all points in $\partial M$ are principal orbits. \end{itemize}
In case $(a)$ the action is fixed point homogeneous. In case $(b)$ a component of $\partial M^*$ is given by $\pi(\partial M)$, where as always $\pi : M \to M^*$ denotes the projection map. We denote by $F$ a component of $\partial M^*$ that is either given by a maximal fixed point component of the action, which we also denote by $F$, in case $(a)$ or $F := \pi(\partial M)$ in case $(b)$. Going through the arguments, the reader might want to picture the fixed point homogeneous case $(a)$, as case $(b)$ is only referred to twice in the rest of this section.

Let $d_F(p) = d(F,p)$ and
$$C := \{p \in M^* \mid d_F(p) \text{ is maximal}\}.$$
From Lemma \ref{concave distance} it follows that $d_F$ is concave and therefore we obtain the following observation.
\begin{lemma}
 $C$ is totally convex in $M^*$.
\end{lemma}
Recall that a point $p \in M^*$ is called a \textit{non-critical point} for $d_F$ if there exists a vector $v \in T_pM^*$ such that the angle between $v$ and the initial direction of any minimal geodesic from $p$ to $F$ is strictly bigger than $\pi/2$ (it is common to call such points 'regular', but we want to avoid any confusion with regular points of $M^*$). A discussion of the theory of critical points for distance functions is given in \cite{grove90}. Analogously to the arguments in \cite{grove-searle94} we have 
\begin{lemma}\label{regular points}
 All points $p \in M^*\setminus(F \cup C)$ are regular. The distance functions $d_C$ and $d_F$ are non-critical at every $p \in M^*\setminus(F \cup C)$.
\end{lemma}
\begin{proof}
Let $p \in M \setminus (F \cup C)$ and $c_1 : [0,1] \to M^*$ be a minimal geodesic from $p$ to $C$ and $c_2 : [0,1]  \to M^*$ be a minimal geodesic from $p$ to $F$. Since $p \notin C$, it follows $d(c_1(0),F) < d(c_1(1),F)$. By concavity of $d_F$ there exists $m > 0$ such that  $d(c_1(t),F) \geq d(c_1(0),F)  +mt$. It follows that $\measuredangle (\dot c_1(0),\dot c_2(0)) > \pi/2$. This shows that both $d_C$ and $d_F$ are non-critical on $M^* \setminus (F \cup C)$.

To see that all points $p \in M \setminus (F \cup C)$ are regular first note that there exists an open neighborhood $U$ of $F$ in $M^*$ such that all points in $U \setminus F$ are regular: In case $(b)$ this is immediate and in case $(a)$ this follows from the isotropy representation at points of $F$. Now let $p \in M \setminus F$ be a nonregular point of minimal distance to $F$. Then $0 < d(p,F)$. Assume that $d(p,F) < d(p,C)$. Again let $c$ be a minimal geodesic from $p$ to $C$ and $\hat p \in M$, $w \in N_{\hat p}\G(\hat p)$ with $\pi(\hat p ) = p$ and $\nu_{\hat p}(w) = \dot c(0)$. Then $V := \{v \in N_{\hat p} \G(\hat p) \mid \measuredangle (v, \G_{\hat p}(w)) > \pi /2\}$ is convex in $N_{\hat p}\G(\hat p)$, invariant under $\G_{\hat p}$ and,   by the regularity of $d_C$ at $p$, nonempty. $\Sigma_0V$ is a positively curved Alexandrov space with nonempty boundary and therefore, there exists a unique point of maximal distance to its boundary, which is fixed by $\G_{\hat p}$, since the boundary is invariant. This implies that there are nonregular points in $M^* \setminus (F \cup C)$ closer to $F$ than $p$, a contradiction.
\end{proof}
As mentioned in the beginning of this chapter, we need to obtain information on the regularity of $\pi^{-1}(C)$.  Referring to Lemma \ref{submanifold}, we have to determine at what points $p = \pi(\hat p) \in C$ the tangent cone $T_pC$ lifts to a linear space under $\nu_{\hat p} : N_{\hat p}\G(\hat p) \to T_pM^*$. The corresponding question for the normal cone $N_pC$ is easy to answer:
\begin{lemma}\label{linear normal}
 Let $p = \pi(\hat p) \in C$. Then $\nu^{-1}_{\hat p}(N_pC) \subset N_{\hat p}\G(\hat p)$ is linear if and only if $p$ is not a regular boundary point of $C$ (here we refer to the intrinsic boundary of $C$).
\end{lemma}
\begin{proof}
 First assume that $p$ is regular. Then an open neighborhood $U$ of $p$ in $M^*$,  equipped with the induced metric, is a smooth Riemannian manifold and $\pi : \pi^{-1}(U) \to U$ is a Riemannian submersion. Since $C$ is convex, it follows that $U \cap C$ is a smooth submanifold of $U$, possibly with nonsmooth boundary (see \cite{cheeger-gromoll72}). Then It is clear that $\nu^{-1}_{\hat p}(N_pC) \subset T_{\hat p}M$ is linear if and only if $N_pC \subset T_pU$ is linear, or equivalently (by convexity) $T_pC \subset T_pU$ is linear, i.e. $p$ is not a boundary point of $C$. 

 Now assume that $p$ is nonregular. Then 
$$\nu^{-1}_{\hat p}(N_pC) = \{v \in N_{\hat p}\G(\hat p) \mid \measuredangle (v,\nu^{-1}_{\hat p}(T_pC)) \geq \pi/2\} \cup  \{0\}.$$
Consequently $\nu^{-1}_{\hat p}(N_{p}C)$ is a closed and convex cone inside $N_{\hat p}\G(\hat p)$. Clearly $\nu^{-1}_{\hat p}(N_pC)$ is invariant under the action of $\G_{\hat p}$ on $N_{\hat p}\G({\hat p})$ and nonempty (since it contains the initial direction of a shortest geodesic to $F$). Further, by Lemma \ref{regular points}, the action of $\G_{\hat p}$ on $\nu^{-1}_{\hat p}(N_pC) \setminus \{0\}$ has no fixed points. As in the proof of Lemma \ref{regular points} it follows that $\nu^{-1}_{\hat p}(N_pC)$ is linear (otherwise, being a convex cone, it would have nonempty boundary and there would exist a nonzero vector fixed by $\G_{\hat p}$). 
\end{proof}

Note that the corresponding statement for the tangent cone $T_pC$ follows if 
$$N_pC^\perp := \{v \in T_pM^* \mid \measuredangle(v,N_pC) \geq \pi/2\} = T_pC.$$
Unfortunately this does not seem to be obvious. However, with some work we are able to show that Lemma \ref{linear tangent} holds analogously for $T_pC$ (cf. Proposition \ref{linear tangent}) implying the above equality. For that we use the approximation of $C$ by the superlevel sets 
$$C^s := \{p \in M^* | d_F(p) \geq s\}$$
to construct a similar approximation of $T_pC$ by convex sets of higher dimensions. Let 
$$a := \max_{p \in M^*} d_F(p).$$
\begin{lemma}\label{convex family}
 Let $p \in C$ and $B_p^*$ denote the open unit ball in $T_pM^*$. Then there exists a family $A_t \subseteq B_p^*, t \in [0,1]$ of closed and convex subsets satisfying the following conditions:
 \begin{itemize}
  \item[(i)] $A_{t_1} \subseteq A_{t_2}$ for $t_1 \leq t_2$,
  \item[(ii)] $A_0 = T_pC \cap B_p^*$,
  \item[(iii)] $A_1 = B_p^*$,
  \item[(iv)] The Hausdorff distance $d_H(A_{t_n},A_t)$ converges to $0$ for $t_n \to t$.
 \end{itemize}
\end{lemma}
\begin{proof} 
After possibly rescaling $M^*$ we may assume that $a > 1$. Let $\lambda_n > 1$ and $0 < s_n < a$ be sequences with $\lambda_n \to \infty$, $s_n \to a$ and 
$$d^{\lambda_n}_H(C^{s_n} \cap B^d_{1/\lambda_n}(p),C \cap B^d_{1/\lambda_n}(p)) = 1/2,$$
where $d^{\lambda_n}_H$ denotes the Hausdorff distance of the space $\lambda M^* = (M^*,\lambda d)$ and $B^d_r(p)$ denotes the open ball of radius $r$ in $(M^*,d)$. Such sequences do exist, since for fixed $\lambda_n$ the function $s \mapsto f_n(s) = d^{\lambda_n}_H(C^{s} \cap B^d_{1/\lambda_n}(p),C \cap B^d_{1/\lambda_n}(p))$ is continuous and decreasing with $f_n(0) = d^{\lambda_n}_H(B^d_{1/\lambda_n}(p),C \cap B^d_{1/\lambda_n}(p)) = \lambda_nd^d_H(B^d_{1/\lambda_n}(p),C \cap B^d_{1/\lambda_n}(p)) = \lambda_n \lambda_n^{-1} = 1$ and $f_n(a) = 0$. Possibly after taking a subsequence, we can assume that $(C^{s_n} \cap B^d_{1/\lambda_n}(p),\lambda d)$ converges to a closed set $L_{1} \subset B_p^*$ with $T_pC \cap B_p^* \subset L_{1}$.
Analogously to the proof of Proposition \ref{T_pC is convex} it follows that $L_{1}$ is convex in $B_p^*$. Similarly, there exists a sequence $s_{n,2} \to a$ with $s_{n,2} < s_n$ such that 
$$d^{\lambda_n}_H(C^{s_{n,2}} \cap B^d_{1/\lambda_n}(p),C^{s_n} \cap B^d_{1/\lambda_n}(p)) = 1/2.$$
Again after taking a subsequence, we may assume that $C^{s_{n,2}} \cap B^d_{1/\lambda_n}(p)$ converges to a closed convex subset $L_{2} \subset B_p^*$, and it follows $T_pC \cap B_p^* \subset L_{1} \subset L_{2} \subset B_p^*$. Iterating this process we construct a finite chain of closed convex sets
$$T_pC \cap B_p^* =:  L_{0} \subset L_{1} \subset L_{2} \subset \dots \subset L_{k} \subset B_p^* =: L_{k + 1} $$
with Hausdorff distances $1/2$ for all neighboring spaces but the last two, which have Hausdorff distance $\leq 1/2$ (the finiteness of this sequence follows from compactness of $\overline B_p^*$). The same way we construct new closed convex spaces 
$$L_i \subset L_{i,1} \subset \dots \subset L_{i,j(i)} \subset L_{(i + 1)}$$ for all $0 \leq i \leq k$ with neighboring spaces having Hausdorff distance $1/4$, respectively $\leq 1/4$ for the last two. Relabeling the indices, we inductively define a family of closed convex sets $L_i$, $i \in I \subset [0,1]$ with $I$ dense in $[0,1]$, such that 
\begin{itemize}
  \item[] $L_{t_1} \subsetneq L_{t_2}$ for $0 \leq t_1 < t_2 \leq 1$,
  \item[] $L_0 = T_pC \cap B_p^*$,
  \item[] $L_1 = B_p^*$.
 \end{itemize}
Finally, for $t \in [0,1] \setminus I$ let $t_n$ be a sequence in $I$ with $t_n > t$, $t_n \to t$ for $t \to \infty$, and set $L_t = \bigcap_{t_n} L_{t_n}$.
Setting $A_t = L_t$ yields the desired family $\{A_t\}_{t \in [0,1]}$.
\end{proof}
\begin{remark} It follows from the proof of the following proposition that all the sets $A_t$ in fact are horizontally convex.
\end{remark}
\begin{proposition}\label{linear tangent}
Let $p = \pi(\hat p) \in C$. Then $\nu^{-1}_{\hat p}(T_pC) \subset N_{\hat p}\G(\hat p)$ is linear if and only if $p$ is not a regular boundary point of $C$.
\end{proposition}
\begin{proof}
Let $p = \pi(\hat p) \in C$. If $p$ is regular, the same argument as in the regular case of the proof of Proposition \ref{linear normal} shows that $\nu^{-1}_{\hat p}(T_pC)$ is linear if and only if $p$ does not lie in the boundary of $C$. 

So assume that $p$ is nonregular. We have to prove that $\nu_{\hat p}^{-1}(T_pC)$ is linear. Let $\{A_t\}_{t \in [0,1]} \subset B_pM^*$ be a family as in Lemma \ref{convex family}. Set 
$$t_1 := \inf \{t \in [0,1] \mid 0 \in \mathring A_t\}.$$
Then $0 \leq t_1 < 1$. By Lemma \ref{regular points}, $T_pC$ contains all nonregular points of $T_pM^*$ and therefore, $A_t$ contains all nonregular points of $B_p^*$ for all $t \in [0,1]$. We claim that $A_t$ is horizontally convex for all $t_1 \leq t \leq 1$: 

Since $A_t$ contains all nonregular points, it is enough to show that $A_t$ is totally convex, because a horizontal geodesic is a piecewise geodesic with nonregular vertices. Also, since the intersection of totally convex sets is totally convex, we only have to prove that $A_t$ is totally convex for $t_1 < t \leq 1$. Let such a $t$ be given. First we show that all nonregular points are contained in the interior $\mathring A_t$ of $A_t$: Assume that a nonregular point $v$ is contained in $\partial^tA$, the topological boundary of $A_t$. Then $v$ is also a point of the boundary of $A_t$ considered as an Alexandrov space. Since $0 \in \mathring A_t$, we have $v \neq 0$. Let $\epsilon > 0$ such that $(1 + \epsilon)v \in B_p^*$. Since $v$ is nonregular, it follows that $(1 + \epsilon)v$ is nonregular as well (the slice representation is linear). Therefore, $(1 + \epsilon)v \in A_t$, in contradiction to Proposition \ref{unique direction ball}. Now let $\gamma : [0,1] \to B_p^*$ be a geodesic with endpoints in $A_t$. We have to show that $\gamma \subset A_t$: Since all nonregular points are contained in $\mathring A_t$, we may assume without loss of generality that all points of $\gamma$ are regular.  Let $\hat \gamma : [0,1] \to N_{\hat p}\G(\hat p)$ be a horizontal lift of $\gamma$. Again, since all points of $\gamma$ are regular, $\hat \gamma$ does not pass through $0$. Consider the triangle $\Delta$ with vertices $0$, $\hat \gamma (0)$ and $\hat \gamma(1)$. Since $\hat \gamma$ is horizontal, it follows that $\nu_{\hat p}$ is an isometric immersion when restricted to $\Delta \setminus \{0\}$. Let $\overline A_t = \nu_{\hat p}^{-1}(A_t) \cap \Delta.$ Since $A_t$ is convex, it follows that $\overline A_t \setminus \{0\}$ is locally convex in $\Delta \setminus \{0\}$.  In fact, since $0 \in \mathring A_t$, $\overline A_t$ is locally convex in $\Delta$ and connected. It is an elementary exercise in euclidean geometry to prove that a closed, connected and locally convex subset of $\Delta$ which contains two of the edges of $\Delta$ in fact equals $\Delta$. Since $\overline A_t$ contains the edges $\overline{0\hat \gamma(0)}$ and $\overline{0\hat \gamma(1)}$, it follows that $\overline A_t = \Delta$. In particular, $\gamma \subset A_t$ and $A_t$ is totally convex, proving the claim.

Clearly $0 \notin \mathring A_{t_1}$. We may assume that $C$ is not a single point, since otherwise the proposition holds trivially. Then $T_pC$ is at least $1$-dimensional and $N_0A_{t_1} \neq \{0\}$ by Lemma \ref{all or normal}. Next we claim that 
$$T_0A_{t_1} = (N_0A_{t_1})^\perp =: V_1^*:$$ 
Since $T_0A_t$ contains all the singularities of $T_pM^*$, it follows precisely as in the proof of Lemma \ref{linear normal} that $\nu^{-1}_{\hat p}(N_0A_{t_1})$ is linear and invariant under $\G_{\hat p}$. Thus $V_1 := (\nu^{-1}_{\hat p}(N_0A_{t_1}))^\perp$ is linear as well and invariant under $\G_{\hat p}$. Therefore, $ V_1/\G_{\hat p} = V_1^*$ is horizontally convex in $T_pM^*$, and so is $A_t \cap V_1^*$ for $t \in [t_1,1]$. It is clear that $T_0A_{t_1} \subseteq V_1^*$.  Assume $T_0A_{t_1} \subsetneq V_1^*$. Then it follows from Lemma \ref{all or normal}, applied to the family $\{A_t \cap V_1^*\}_{t \in [t_1,1]}$ of horizontally convex subsets of $V_1^*$, that $A_{t_1} \cap V_1^*$ has a nonzero normal vector $v$ at $0$ contained in $V_1^*$.  By convexity of $A_{t_1}$ we have $A_{t_1} \subset T_0A_{t_1} \subset (N_0A_{t_1})^\perp = V_1^*$, so
$$A_{t_1} \cap V_1^* = A_{t_1}.$$
Thus $v$ is normal to $A_{t_1}$. But then $v \in N_0A_{t_1}$, so $v$ is normal to $V_1^*$, a contradiction.

It follows that 
$$\nu^{-1}_{\hat p}(T_0A_{t_1}) = V_1 \subset N_{\hat p}\G(\hat p)$$
is a linear subspace. We distinguish two cases:

\textit{case 1.} Let $t_1 = 0$. Then $A_{t_1} = A_0 = T_pC \cap B_p^*$ and we are done.
 
\textit{case 2.} Let $t_1 > 0$. Set
$$t_2 := \inf \{t \in [0,1] \mid 0  \text{ is an interior point of } A_t \cap V_1^* \text{ in } V_1^*\}.$$
We claim that $$t_2 < t_1:$$ In fact, it is clear that $t_2 \leq t_1$ and if we assume that $t_2 = t_1$ it follows as above from Lemma \ref{all or normal} that there exists a nonzero normal vector $v$ to $A_{t_2} \cap V_1^*= A_{t_1}$ in $V_1^* = (N_0A_{t_1})^\perp$, a contradiction.

With the same argument as before applied to the family $\{A_t \cap V_1^*\}_{t \in [0,1]} \subset V_1^*$ it follows that $A_t \cap V_1^*$ is horizontally convex in $V_1^*$ for all $t_2 \leq t \leq 1$ and 
$$\nu^{-1}_{\hat p}(T_pA_{t_2}) =: V_2$$
is linear and invariant under $\G_{\hat p}$. Again, if $t_2 = 0$ we are done. Otherwise consider the family $\{A_t \cap V_2^*\}_{t \in [0,1]} \subset V_2^*$ and so on. After every step we have $\dim V_k < \dim V_{k + 1}$ and therefore, after a finite number of steps $t_k = 0$ holds and we are done.
\end{proof}
\begin{corollary}\label{nonreg bdry}
For all $p \in C$ the tangent cone $T_pC$ is horizontally convex and $T_pC = N_pC^\perp$.
\end{corollary}
\begin{corollary}\label{boundary points}
Let $p \in \partial C$ and $B$ denote the boundary component of $C$ containing $p$. Then $B$ consists of regular (nonregular) points exclusively if $p$ is regular (nonregular).
\end{corollary}
\begin{proof} It is enough to show that $B$ consists of nonregular points exclusively if $p = \pi(\hat p)$ is nonregular. Since the set of nonregular boundary points of $C$ is closed in $\partial C$, it suffices to show that is also open.

For that we perform the same construction as in the proof  of Lemma \ref{submanifold}: From Proposition \ref{linear tangent} it follows that there exists a linear $\G_{\hat p}$-invariant subspace $V \subset N_{\hat p}\G(\hat p)$ such that $T_pC = V/\G_{\hat p}$. Let $\epsilon > 0$ such that the normal exponential map to the orbit $\G(\hat p)$ is injective on vectors of length less than $\epsilon$. Then 
$$U_\epsilon := \G(\exp_{\hat p}(V \cap B_\epsilon(0_{\hat p})))$$
 is a smooth submanifold of $M$ and invariant under $\G$. Consider the Riemannian manifold $(U_\epsilon,g)$ equipped with the induced isometric $\G$-action. Since $U_\epsilon$ has curvature bounded from below, possibly after choosing a smaller $\epsilon > 0$, it follows that $(U_\epsilon^*,g^*) \subset (M^*,g^*)$ is a locally complete Alexandrov space with curvature bounded below. By construction 
\begin{align}\label{bald}
T_pU_\epsilon^* = T_pC
\end{align}
and by convexity of $C$ it follows that $C_\epsilon := (C \cap B_\epsilon(p)) \subset U^*_\epsilon$. Moreover, it is easy to see that $C_\epsilon$ is locally convex in $(U_\epsilon^*,g^*)$. 
Also, since $U_\epsilon$ is a smooth manifold without boundary, it follows that all boundary points of $U_\epsilon^*$ are nonregular. In particular,
\begin{align}\label{hab ichs}
\partial U_\epsilon^* \subset C_\epsilon.
\end{align}
 Now consider the double $\tilde U_\epsilon^*$ (in fact, we can apply the doubling theorem only to complete spaces. However, it is not hard to construct a complete Riemannian $\G$-manifold $\hat U_\epsilon$, together with an isometric, $\G$-equivariant embedding $U_\epsilon \to \hat U_\epsilon$. Then we can consider $C_\epsilon \subset \hat U^*_\epsilon$). From \eqref{hab ichs} it follows that $\tilde C_\epsilon$ is locally convex in $\tilde U_\epsilon^*$. From \eqref{bald} we see $T_p\tilde U_\epsilon^* = T_p{\tilde C} = T_p\tilde C_\epsilon$. Therefore, it follows from Corollary \ref{ball} that $p$ is an interior point of $\tilde C_\epsilon$ in $\tilde U_\epsilon^*$. But then also $p$ is an interior point of $C_\epsilon$ in $U_\epsilon^*$. Since the boundary of $U_\epsilon^*$ consists of nonregular points exclusively, it follows that there are no regular boundary points of $C$ arbitrary close to $p$.
\end{proof}
Thus it is justified to talk about regular and nonregular boundary components of $C$.
\begin{corollary}\label{smooth lift}
$\pi^{-1}(C) \subset M$ is a smooth submanifold, possibly with nonsmooth boundary. Its boundary components are the preimages of the regular boundary components of $C$.
\end{corollary}
\begin{proof}
Let $B$ denote the set of regular boundary points of $C$. From the Corollary \ref{nonreg bdry} it follows that $B$ is closed and there exists an open neighborhood $U$ of $B$ which consists of regular points only. Then, if $B$ is nonempty, $\pi^{-1}(U \cap C)$ is a smooth submanifold whose boundary components are the preimages of $B$. Moreover, for all $p \in C \setminus B$ the tangent cone $T_pC$ lifts under $\nu$ to a linear space by Lemma \ref{linear tangent}.
By Lemma \ref{submanifold} it follows that $\pi^{-1}(C \setminus B)$ is a smooth submanifold without boundary of $M$ as well.
\end{proof}
Now we prove the main theorem of this chapter:
\begin{theorem} \label{ddb}
Assume the action of $\G$ on $M$ is fixed point homogeneous. Then there exists a smooth submanifold $N$ of $M$, without boundary, such that $M$ is diffeomorphic to the normal disk bundles $D(F)$ and $D(N)$ of $F$ and $N$, respectively, glued together along their boundaries;
\begin{align}\label{doubled2}
 M \cong D(F) \cup_{\partial} D(N).
\end{align}
Further, $N$ is horizontally convex, $\G$-invariant and it contains all singularities of $M$ up to $F$.
\end{theorem}
\begin{proof} First, we may assume that $M^*$ has precisely one boundary component: Otherwise the second boundary component corresponds to a second maximal fixed point component $F_2$ and $M^*$ is isometric to $F \times [0,a]$ (compare \cite{searle-yang94}, Theorem 2). In this case, setting $N := F_2$ we are done.

We distinguish several cases to construct a smooth submanifold $N$ of $M$ with empty boundary. In each case $N$ will be obtained as the preimage of a horizontally convex subset $A \subset M^*$ disjoint from $F$, and there exists a gradient-like vector field $X$ with respect to $N$ on $M \setminus(F \cup N)$ which is radial near $F$ and $N$. Then the theorem follows.

\textit{case 1.} Assume $C$ has no regular boundary points. From Corollary \ref{smooth lift} it follows that $N := \pi^{-1}(C)$ is a smooth submanifold without boundary of $M$ and the distance function $d_N$ is non-critical on $M \setminus (F \cup N)$. Thus there exists a gradient-like vector field on $M \setminus (F \cup N)$ which is radial near $F$ and $N$.

\textit{case 2.}  Assume $C$ contains a regular boundary point. By Corollary \ref{smooth lift} there exists a component of $\partial C$ consisting of regular points exclusively. Let $B$ denote the collection of all such regular boundary components of $C$. Set
$$C_2 := \{p \in C \mid d(p,B) \text{ is maximal}\}.$$
Since the distance function $d_B : C \to \RR$ is concave, it follows that $C_2$ is totally convex in $C$ (and $M^*$). We claim that the distance function to $C_2$ is non-critical on $M \setminus (F \cup C_2)$ and $C_2$ contains all nonregular points of $M^*$ up to $F$:

Since $B$ consists of regular points exclusively, it is shown precisely as in Lemma \ref{regular points} that $d_{C_2}$ is non-critical on $C \setminus (B \cup C_2)$ and all points in $C \setminus C_2$ are regular. Now let $p \in B$. Since $p$ is regular and contained in $\partial C$, it follows that $T_pC$ is a convex cone with nonempty boundary inside the linear space $T_pM^*$. Then there exists some vector $v \in T_pM^*$ that has an angle $> \pi/2$ to every vector $w \in T_pC \setminus \partial T_pC$. Let $c :[0,1] \to M^*$ be a minimal geodesic from $p$ to $C_2$. Then $c \subset C$ by convexity. Therefore $\dot c(0) \in T_pC$. In fact, $\dot c(0) \in T_pC \setminus \partial T_pC$, since $d(c(t),B) \geq mt$ for some positive constant $m$. Hence $\measuredangle (v,\dot c(0)) > \pi/2$ and $p$ is a non-critical point for $d_{C_2}$. The same argument applies to every point $p \in M \setminus (F \cup C)$, considering the level set that contains $p$ in its boundary, rather than $C$. This proves the claim.

\textit{case 2.1.} Assume $C$ contains more than two boundary components. Let $B_1$ denote a regular boundary component of $C$. Again it follows that $C$ is isometric to $B_1 \times I$ for a closed interval $I$. So $C$ has precisely two boundary components and let $B_2$ denote the second boundary component. In this case it follows that $C_2 = B_2$. By Corollary \ref{smooth lift} the following two cases can occur:

\textit{case 2.1.1.} All points of $B_2$ are nonregular. Then $C_2 = B_2$ and it follows, again by Corollary \ref{smooth lift}, that $(\pi^{-1}(C),g)$ is a smooth Riemannian $\G$-manifold with (possibly nonsmooth) boundary whose quotient space $C$ is a nonnegatively curved Alexandrov space with nonempty boundary. Further $\pi^{-1}(C)$ has precisely one boundary component, which consists of regular points exclusively. Hence, our assumptions (case (b)) are satisfied, with $\pi^{-1}(C)$ in the role of $M$, and Lemma \ref{smooth lift} holds correspondingly for $C_2$. More precisely, $\pi^{-1}(C_2) \subset \pi^{-1}(C)$ is a smooth submanifold and, since all points of $C_2$ are nonregular, it has empty boundary. Setting $N := \pi^{-1}(C_2)$, we are done.

\textit{case 2.1.2.} All points in $B_2$ are regular. In this case $M^*$ is a smooth Riemannian manifold with boundary $F$ and $C$ is a submanifold of $M^*$ with boundary $B = B_1 \cup B_2$. Also $C_2$ is isometric to $B_1$ and therefore $\partial C_2 = \emptyset$. Since $C_2$ is convex, it is a smooth submanifold of $M^*$ with empty boundary. The same holds for $N := \pi^{-1}(C_2)$, since $\pi : M \setminus F \to M^* \setminus F$ is a smooth submersion.

\textit{case 2.2} Assume $C$ has precisely one boundary component $B$. As in case 2.1.1.\ it follows from Lemma \ref{smooth lift} that $\pi^{-1}(C_2)$ is a smooth submanifold of $M$, but possibly with nonempty boundary. If the boundary is empty, we are done. If the boundary is nonempty, it follows that $C_2$ contains a boundary component consisting of regular points exclusively. In this case we repeat the arguments of case $2$ with $C_2$ in the role of $C$ and $C$ in the role of $M^*$.

Iterating this process proves the theorem, since the dimension of the set $C_i$ drops in every step and therefore after a finite number of steps $\partial C_i = \emptyset$ holds.
\end{proof}
\begin{remark}
This proof differs slightly from the original proof given in \cite{spindeler14}. The proof given there is still correct but as it was noticed by Michael Wiemeler the constructed submanifold $N$ was not in general invariant under $\mathsf H = \operatorname{Iso}_F(M)$, the subgroup of $\operatorname{Iso}(M)$ which leaves $F$ invariant, what is necessary for lemma \ref{equivariant} to hold. With this proof it follows that $N$ is invariant under $\mathsf H$ as can be seen as follows: Keeping the notation of the proof of theorem \ref{ddb}, it is clear that $\pi^{-1}(C)$ is invariant under $\mathsf H$, since $F$ is invariant under $\mathsf H$. If the boundary of $\pi^{-1}(C)$ is empty we have $N = \pi^{-1}(C)$ and the claim follows. Otherwise the boundary of $\pi^{-1}(C)$, which is given by $\pi^{-1}(B)$, is invariant under $\mathsf H$ as well. Therefore, the same holds for the set of points of $\pi^{-1}(C)$ of maximal distance to the boundary, which equals the set $\pi^{-1}(C_2)$ and so on.
\end{remark}
If $M$ is simply connected and $\G$ is connected, we obtain a little information on the dimension of $N$.
\begin{lemma}\label{codim}
Let $M$ and $N$ be as in Theorem \ref{ddb} and assume that $M$ is simply connected and $\G$ is connected. Then $N$ has codimension $\geq 2$ in $M$.
\end{lemma}
\begin{proof}
We consider the double disk bundle decomposition 
$$M = D(F) \cup_E D(N).$$
Assume that $N$ has codimension $1$. Since $\G$ and $F$ are connected, it follows that $E = \partial D(N)$ is connected as well. Therefore, the projection map $p : E \to N$ is a two fold covering map with connected total space. Hence $\pi_1(N)/\pi_1(p)(E) \cong \ZZ_2$. 
 
In contradiction to this we show that $\pi_1(N)/\pi_1(p)(E)$ is trivial using the van Kampen theorem: Let $q : E \to F$ denote the projection map and set $U := \pi_1(p)(\pi_1(E))$. Then the following diagram commutes;
$$\begin{xy}
\xymatrix
{\pi_1(E) \ar^{\pi_1(p)}[r] \ar_{\pi_1(q)}[d] & \pi_1(N) \ar[d]^{[\ ]}\\
\pi_1(F) \ar_{0}[r] & \pi_1(N)/U.}
\end{xy}
$$
Here, $[\ ]$ denotes the quotient map (note that $U$ is normal in $\pi_1(N)$ since it has index $2$). By the van Kampen theorem $\pi_1(M)$ is the pushout of the maps $\pi_1(p)$ and $\pi_1(q)$. Therefore, there exists a morphism $h : \pi_1(M) \to \pi_1(N)/U$ such that the diagram
$$\begin{xy}
  \xymatrix{
      \pi_1(E) \ar[r]^{\pi_1(p)} \ar[d]_{\pi_1(q)}  &  \pi_1(N) \ar[d] \ar@/^/[ddr]^{[\ ]}  &  \\
      \pi_1(F) \ar[r] \ar@/_/[drr]_0  &  \pi_1(M) \ar[dr]^h  &  \\
      &  &  \pi_1(N)/U
  }
\end{xy}$$
commutes. Now, since $\pi_1(M)$ is trivial, it follows that the quotient map $[ \ ]$ is the $0$-map. Hence $\pi_1(N)/U$ is trivial.
\end{proof}
Fixed point homogeneous $\G$-actions frequently occur as subactions of larger isometric actions. Therefore, we conclude this section considering the question under which subgroups of $\operatorname{Iso}(M)$ the decomposition \eqref{doubled2} is invariant. A necessary condition for this is that $F$ is invariant. We can construct the decomposition of $M$ in a way that this is also sufficient. Recall that we denote by $\operatorname{Iso}_F(M)$ the subgroup of $\operatorname{Iso}(M)$ that leaves $F$ invariant. Then clearly $\G \subseteq \operatorname{Iso}_F(M)$. From the construction of $N$ it furher follows that $N$ is invariant under $\operatorname{Iso}_F(M)$ as well, as noted in the remark following the proof of theorem \ref{ddb}.
\begin{lemma}\label{equivariant}
 Let $M$ be a nonnegatively curved, fixed point homogeneous $\G$-manifold, $F,N \subset M$ as in Theorem \ref{ddb} and $\mathsf H := \operatorname{Iso}_F(M)$. Then there exists a $\mathsf H$-equivariant diffeomorphism $b : \partial D(F) \to \partial D(N)$ and $M$ is $\mathsf H$-equivariantly diffeomorphic to
 $$D(F) \cup_b D(N).$$
\end{lemma}
To prove this lemma we need the following observation.
\begin{proposition}\label{extension}
Let $\G$ act isometrically on a Riemannian manifold $M$. Assume that $v \in T_pM$ is fixed by $\G_p$. Then there exists a smooth, $\G$-invariant vector field $X$ in a neighborhood of $\G(p)$ with $X(p) = v$.
\end{proposition}
\begin{proof}
 Let $v \in T_pM$ be fixed by $\G_p$. First we assume that $\G_p = \G$. We have an orthogonal decomposition $T_pM = \RR v \oplus U$ which is preserved by the action of $\G$. Let $V$ denote the constant vector field $V(u) = v$ on $\RR^n = T_pM$. Since $\G$ acts orthogonally and $v$ is fixed, it follows that $V$ is invariant under $\G$. Let $\delta > 0$ such that ${\exp_p}_{|B_\delta(0_p)}$ is a diffeomorphism onto its image and set $X(\exp_p(u)) = d(\exp_p)_u V$. By equivariance of $\exp_p$ we calculate, for $q = \exp_p(u)$ and $g \in \G$,
\begin{align*}
 dg_qX &= dg_q d(\exp_p)_u V = d(g \circ \exp_p)_uV\\
 &= d(\exp_p \circ dg)_u V = d(\exp_p)_{dg(u)}V\\
 &= X(\exp_p(dg(u)) = X(g(\exp_p(u))\\
 &= X(g(q)).
\end{align*}
Thus $X$ is invariant under $\G$. Now assume $p$ is not fixed by $\G$. Let $S$ be a slice at $p$, so $\G_p$ is acting on $S$ fixing $p$. By the first case (equipping $S$ with the induced metric) there exists a smooth vector field $Y$ on $S$ invariant under $\G_p$. Let $q \in \G(S)$, so $q = g.m$ for some $m \in S$. We set $X(q) = dg_mY$. Then $X$ is well defined and smooth: Let $q = g.m = h.n$ for $g,h \in \G$ and $n,m \in S$. Then $g^{-1}h.n = m$, so $g^{-1}h \in \G_p$. Consequently $dh^{-1}_{hn}dg_mY = d(h^{-1}g)_mY = Y((h^{-1}g)m) = Y(n)$, that is $dh_nY = dg_mY$. Hence $X$ is well defined. Smoothness of $X$ follows easily. Further
\begin{align*}
 dh_qX = dh_{g.m}dg_mY = d(hg)_mY = X((hg).m) = X(h.q),
\end{align*}
so $X$ is $\G$-invariant.
\end{proof}
Now we give the \textit{proof of Lemma \ref{equivariant}}: Let $X$ be a gradient-like vector field on $M \setminus (F \cup N)$ with respect to $N$, which is radial near $F$ and $N$ (recall that the existence of $X$ was shown in the proof of Theorem \ref{ddb}). We claim that there exists such a vector field $\tilde X$ which is additionally invariant under $\mathsf H$.

The argument for this is similar to arguments we used before: Let $p \in M \setminus (F \cup N)$. Then $\measuredangle (X(p),\dot c(0)) > \pi/2$ for all minimal geodesics from $p$ to $N$. Since $\mathsf H$ acts isometrically and leaves $N$ invariant it follows that 
$$\measuredangle (dh_pX,\dot c(0)) > \pi/2$$
for all $h \in \mathsf H_p$, for all minimal geodesics $c$ from $p$ to $N$. Therefore, there exists a vector $v \in T_pM$, fixed by $\mathsf H_p$, such that 
\begin{align}\label{angl}
\measuredangle (v,\dot c(0)) > \pi/2
\end{align}
for all minimal geodesics $c$ from $p$ to $N$ (compare the proof of Lemma \ref{regular points}). According to Proposition \ref{extension}, there exists a smooth, $\mathsf H$-invariant vector field $Y$ in a neighborhood of $\mathsf H(p)$ with $Y(p) = v$. Since $\mathsf H$ acts isometrically and because of \eqref{angl}, it follows that $Y$ is gradient-like with respect to $N$ in a neighborhood of $\mathsf H(p)$. Thus there exists a $\mathsf H$-invariant, gradient-like vector field with respect to $N$ in a neighborhood of every $\mathsf H$-orbit of $M \setminus (F \cup N)$. Since the radial unit vector fields near $F$ and $N$ are clearly $\mathsf H$-invariant, we can construct $\tilde X$ as in the claim using a $\mathsf H$-invariant partition of unity.

Via the normal exponential map of $D(F)$ and the integral curves of $\tilde X$ one constructs an $\mathsf H$-equivariant diffeomorphism 
$$f : D(F) \to M \setminus B_\epsilon(N),$$
for a small $\epsilon > 0$, such that $\overline {B_\epsilon(N)}$ is $\mathsf H$-equivariantly diffeomorphic to $D(N)$ via the normal exponential map. Setting $b = \exp_{D(N)}^{-1} \circ f$, we obtain $\mathsf H$-equivariant diffeomorphisms (compare \cite{krankaanrinta07}, section 10)
$$M \cong f(D(F)) \cup_{id} \exp(D(N)) \cong D(F) \cup_{b} D(N).$$ \hfill$\square$

\subsection{Applications to rational ellipticity}\label{further results}
It is  conjectured, that a complete, simply connected manifold of nonnegative curvature is rationally elliptic. In this section we give two results related to this conjecture for nonnegatively curved fixed point homogeneous manifolds and torus manifolds. First recall
\begin{definition}
A closed manifold $M$ is called \textit{rationally $\Omega$-elliptic} if the total rational homotopy of the loop space
$$\pi_\ast (\Omega X,\ast) \otimes \mathbb Q$$
is finite dimensional. 	$M$ is called \textit{rationally elliptic} if it is rationally $\Omega$-elliptic and simply connected.
\end{definition}
Consider a double disk bundle $M = D(F) \cup_\partial D(N)$, where $D(F)$ and $D(N)$ are disk bundles over closed manifolds $F$ and $N$, respectively. Then $F$ is rationally $\Omega$-elliptic if and only $\partial D(F)$ is. Moreover, from \cite{grove-halperin87}, Corollary 6.1 it follows that a simply connected manifold which admits a double disk bundle decomposition is rationally $\Omega$-elliptic if and only if the boundary of one of the two disk bundles is rationally $\Omega$-elliptic. Therefore, from Theorem \ref{ddb} we have
\begin{theorem}\label{elliptic}
Let $M$ be closed, simply connected, nonnegatively curved fixed point homogeneous manifold and $F$ be a fixed point component of maximal dimension. Then $M$ is rationally $\Omega$-elliptic if and only if $F$ is rationally $\Omega$-elliptic. 
\end{theorem}
We conclude with an application of our results to nonnegatively curved torus manifolds. 
\begin{definition}
A \textit{torus manifold} is a smooth, connected, closed and orientable manifold $M$ of even dimension $2n$ admitting a smooth and effective action by the $n$-dimensional torus $\mathsf T^n$ with nonempty fixed point set. 
\end{definition}

\begin{theorem}
Let $M^{2n}$ be a closed simply connected torus manifold equipped with an invariant metric of nonnegative curvature. Then $M$ is rationally elliptic.
\end{theorem}
Before giving the proof we note that, using also the results of this chapter, Wiemeler shows in \cite{wiemeler14} that a compact and simply connected torus manifold admitting an invariant metric of nonnegative curvature is diffeomorphic to a quotient of a free linear torus action on a product of spheres.

\begin{proof}
Let $\T^n$ act effectively and isometrically with nonempty fixed point set on $M$. Let $p_0 \in \operatorname{Fix}(\T^n)$ and consider the orthogonal action of $\T^n$ on $\sphere^{2n - 1} \subset T_{p_0}M$ induced by the slice representation. It was shown in \cite{grove-searle94}, Theorem 2.2,  that there exists a $1$-dimensional torus $\T^1_1 \subset \T^n$ acting fixed point homogeneous on $\sphere^{2n - 1}$. Hence $\T_1^1$ also acts fixed point homogeneous on $M$ and there exists a maximal fixed point component $F$ containing $p_0$. Consequently
\begin{align}\label{decaf}
M \cong D(F) \cup_\partial D(N).
\end{align}
for a smooth submanifold $N$ with $\dim N \leq 2n -2$, by Theorem \ref{ddb} and Lemma \ref{codim}. We claim that $F$ is simply connected:

\textit{case 1:} Assume $\dim N  \leq 2n - 3$. Then, by transversality, $\pi_1(F) = \pi_1 (M \setminus N) = \pi_1(M)$. Hence $\pi_1(F) = 0$.

\textit{case 2}:  Assume $\dim N = 2n - 2$. Let $E := \partial D(F) \cong \partial D(N)$. Then, as shown in the proofs of Propositions 3.5 and 3.6 in \cite{galaz-garcia-spindeler12}, the following is true: $N$ is orientable and that there exists a $\T^1_2$-action on the normal bundle of $N$ obtained by orthogonally rotating the fibers. This action commutes with $\T^n$ and we obtain a smooth $\T^1_2$-action on $D(N) = E \times_{T_1^2} D^2$ which can be extended to $(E \times_{\T_1^1} D^2) \cup_E (E \times_{T_1^2} D^2) = D(F) \cup_E D(N) = M$.

Let $q \in \operatorname{Fix}\T^n$, $t \in \T^n$ and $g \in \T^1_2$. Then $t.(g.q) = g.(t.q) = g.q$. Hence $\operatorname{Fix} \T^n$ is invariant under $\T_2^1$. Since $\T^1_2$ is connected, and $\operatorname{Fix} \T^n$ is discrete, we see that $\operatorname{Fix} \T^n \subset \operatorname{Fix} \T_2^1$. It follows that $\T^2 := \T^1_1 \oplus \T^1_2$ acts on $M$ with $p_0 \in \operatorname{Fix} \T^2$. 

Consider the projections $f_1 : E \to E/\T_1^1 = F$ and $f_2 : E \to E/\T_2^1 = N$. From the homotopy sequences of this fibrations we obtain exact sequences
$$\dots \to \pi_1(\T_1^1) \xrightarrow{\pi_1(i_1)} \pi_1(E) \xrightarrow{\pi_1(f_1)} \pi_1(F) \to 1$$
and
$$\dots \to \pi_1(\T_2^1) \xrightarrow{\pi_1(i_2)} \pi_1(E) \xrightarrow{\pi_1(f_2)} \pi_1(N) \to 1,$$
where the maps $i_1$ and $i_2$ are the inclusions of the fibers over a given basepoint. Set $U_k = \pi_1(i_k)(\pi_1(\T^1_k))$ for $k = 1,2$.
So $\pi_1(F) \cong \pi_1(E)/U_1$, $\pi_1(N) \cong \pi_1(E)/U_2$ and we have a commutative diagram
\begin{align}\label{diagram}
\begin{xy}
\xymatrix
{\pi_1(E) \ar^{\pi_1(f_2)}[r] \ar_{\pi_1(f_1)}[d] & \pi_1(N) \ar[d]\\
\pi_1(F) \ar[r] & \pi_1(E)/U_1U_2.}
\end{xy}
\end{align}
Here the lower map is given via $\pi_1(F) \cong \pi_1(E)/U_1 \to \pi_1(E)/U_1U_2$, and analogously for the map on the right. By \eqref{decaf} and the van Kampen theorem there exists a morphism $h : \pi_1(M) \to \pi_1(E)/U_1U_2$ making the following diagram commute:
$$\begin{xy}
  \xymatrix{
      \pi_1(E) \ar[r]^{\pi_1(f_2)} \ar[d]_{\pi_1(f_1)}  &  \pi_1(N) \ar[d] \ar@/^/[ddr]  &  \\
      \pi_1(F) \ar[r] \ar@/_/[drr]  &  \pi_1(M) \ar[dr]^h  &  \\
      &  &  \pi_1(E)/U_1U_2
  }
\end{xy}$$
Since all the maps in \eqref{diagram} are surjective it follows that $h$ is surjective as well. Since $\pi_1(M) = 0$, it follows that $\pi_1(E) = U_1U_2$. Hence, $\pi_1(E)$ is generated by the orbits $\T^1_1(q)$ and $\T^1_2(q)$ for a given point $q \in E$. Therefore, the map $\tau_q : \T^2 \to E$, $g \mapsto g.q$ induces a surjection $\pi_1(\tau_q) : \pi_1(\T^2) \to \pi_1(E)$ for all $q \in E$. Consequently we obtain a surjection $\pi_1(f_1 \circ \tau_q) : \pi_1(\T^2) \to \pi_1(F).$

Pick $q_0 \in E$ such that $f_1(q_0) = p_0$. Since $p_0 \in \operatorname{Fix} \T^2$, the map $f_1 \circ \tau_{q_0}$ is constant. Thus $\pi_1(f_1 \circ \tau_{q_0}) = 0$ and it follows that $F$ is simply connected.

Because $F$ is totally geodesic, $F$ also has nonnegative curvature. Further $\T^{n - 1} = \T^n/\T^1_1$ acts effectively on $F^{2n - 2}$ with nonempty fixed point set. So $F$ is a nonnegatively curved, simply connected Torus manifold as well.

Now the proof follows by induction on $n$ using Corollary \ref{elliptic}.
\end{proof}

\addcontentsline{toc}{chapter}{Bibliography}

\bibliographystyle{alpha}
\bibliography{thesis_arxiv}

\end{document}